\documentclass[11pt,a4paper,final]{article}

\newif\ifsiamart
\makeatletter%
\@ifclassloaded{siamart171218}%
  {\siamarttrue}%
  {\siamartfalse}%
\makeatother%

\ifsiamart\else
\usepackage{authblk}
\fi
\usepackage{silence}
\usepackage[T1]{fontenc}
\usepackage[utf8]{inputenc}
\usepackage{csquotes}
\usepackage[UKenglish]{babel}
\usepackage{paralist}
\usepackage[shortlabels]{enumitem}
\usepackage[margin=.6in]{geometry}
\usepackage{xcolor}
\usepackage{graphicx}
\usepackage{bbm}
\usepackage{amsmath}
\usepackage{amssymb}
\usepackage{amsfonts}
\usepackage{stmaryrd}
\usepackage{mathrsfs}
\usepackage{mathtools}
\usepackage{epstopdf}
\usepackage{comment}

\usepackage{xspace}
\usepackage{float}
\usepackage{stackengine}

\ifsiamart\else
\usepackage{amsthm}
\usepackage{hyperref}
\hypersetup{%
    linktocpage=true,
    colorlinks=true,
    citecolor=red,
    linkcolor=blue,
    urlcolor=blue,
}
\usepackage[nameinlink,capitalise]{cleveref}
\newcommand{\email}[1]{\href{mailto:#1}{#1}}
\newenvironment{keywords}{\noindent \textbf{Keywords.}}{}
\newenvironment{AMS}{\noindent \textbf{AMS subject classification.}}{}
\fi

\usepackage{setspace}
\usepackage{amssymb,bbm,mathrsfs}
\usepackage{mathtools}
\usepackage{multicol}
\usepackage[normalem]{ulem}
\usepackage{caption}
\usepackage{subcaption}
\usepackage{array}

\DeclareMathOperator{\supp}{supp}

\def\W{\mathcal W}

\def\Hc{\mathcal H}
\def\N{\mathcal N}
\def\nat{\mathbb{N}}
\def\M{\mathcal M}

\def\V{\mathcal V}

\def\P{\mathcal P}
\def\tP{\tilde{\mathcal P}}
\def\R{\mathbb R}
\def\cR{\mathcal R}

\def\U{\mathcal U}
\newcommand{\Wass}{\mathscr{W}}

\newcommand{\rhot}{\tilde{\rho}}

\newcommand{\rhoinf}{\rho_\infty}
\newcommand{\lambdat}{\tilde{\lambda}}
\newcommand{\Lambdat}{\tilde{\Lambda}}
\newcommand{\rhob}{\bar{\rho}}

\newcommand{\Gr}{\text{Gr}\, }

\DeclareMathOperator*{\amin}{argmin}
\DeclareMathOperator*{\amax}{argmax}

\DeclareMathOperator{\id}{id}
\DeclareMathOperator{\Id}{I}

\newcommand{\br}{r}
\newcommand{\bx}{b}
\newcommand{\eps}{\varepsilon}

\newcommand{\Wbar}{\overline{\mathscr{W}}}

\newcommand{\m}{m}

\newcommand{\Hess}[1]{\text{Hess}\left( #1 \right)}

\DeclareMathOperator*{\argmin}{arg\,min}

\renewcommand{\d}{\mathrm{d}}
\newcommand{\ddt}{\frac{\d}{\d t}}

\newcommand{\dds}{\frac{\d}{\d s}}
\newcommand{\ddss}{\frac{\d^2}{\d s^2}}
\newcommand{\weakstar}{\overset{\ast}{\rightharpoonup}}
\newcommand{\li}{\liminf_{n\rightarrow\infty}}
\newcommand{\ls}{\limsup_{n\rightarrow\infty}}
\newcommand{\norm}[1]{\left\lVert#1\right\rVert}
\renewcommand{\div}[1]{\mathrm{div}\left( #1 \right)}
\newcommand{\pxi}{\partial_{x_i}}

\def\<#1,#2>{\left\langle #1,\,#2\right\rangle}

\ifsiamart
\newtheorem{remark}[theorem]{Remark}
\newtheorem{assumption}{Assumption}
\else
\theoremstyle{plain}

\newtheorem{theorem}{Theorem}
\newtheorem{lemma}[theorem]{Lemma}
\newtheorem{corollary}[theorem]{Corollary}
\newtheorem{proposition}[theorem]{Proposition}

\newtheorem{remark}[theorem]{Remark}
\newtheorem{definition}[theorem]{Definition}

\newtheorem{assumption}{Assumption}
\fi

\crefname{lemma}{Lemma}{Lemmas}
\crefname{remark}{Remark}{Remarks}
\crefname{assumption}{Assumption}{Assumptions}
\crefname{proposition}{Proposition}{Propositions}
\crefname{section}{Section}{Sections}
\crefname{subsection}{Subsection}{Subsections}
\crefname{equation}{}{}
\Crefname{equation}{Equation}{Equations}
\newlist{lemmaenum}{enumerate}{3}
\setlist[lemmaenum]{label=(\alph*),ref=\,(\alph*)}
\crefname{lemmaenum}{Lemma}{Lemmas}
\newlist{assumpenum}{enumerate}{5}
\setlist[assumpenum]{label=(\alph*), font={\bfseries}}
\newlist{auxenum}{enumerate}{2}
\setlist[auxenum]{label=(\alph*),ref=(\alph*)}
\crefname{auxenumi}{Item}{Items}
\crefname{enumi}{}{}
\crefname{equation}{}{}
\crefname{assumpenumi}{}{}
\crefname{assumpenumii}{}{}
\Crefname{assumpenumi}{Assumption}{Assumptions}
\Crefname{assumpenumii}{Assumption}{Assumptions}
\Crefname{assumpenumii}{Assumption}{Assumptions}
\crefrangeformat{assumpenumi}{#3#1#4--#5#2#6}
\Crefname{lemmaenumi}{Part}{Parts}
\Crefname{figure}{Figure}{Figures}
\numberwithin{equation}{section}
\numberwithin{theorem}{section}

\usepackage{tikz}
\usetikzlibrary{shapes.misc}
\usetikzlibrary{positioning}

\usepackage[url=true,backend=biber,style=trad-abbrv,url=false,doi=false,isbn=false]{biblatex}
\DeclareFieldFormat{volume}{volume \textbf{#1}}
\DeclareFieldFormat[article]{volume}{\textbf{#1}}

\ifsiamart\else
\let\oldparagraph=\paragraph
\renewcommand\paragraph[1]{\oldparagraph{#1.}}
\fi

\begin{document}
\title{Coupled Wasserstein Gradient Flows\\ for Min-Max and Cooperative Games}

\ifsiamart
    \author{%
    }
\else
    \author[1]{Lauren Conger$^{a,}$}
    \author[1]{Franca Hoffmann$^{b,}$}
    \author[1]{Eric Mazumdar$^{c,}$}
    \author[2]{Lillian J. Ratliff$^{d,}$}
    \affil[ ]{\footnotesize $^a$\email{lconger@caltech.edu},
        $^b$\email{franca.hoffmann@caltech.edu},
    $^c$\email{mazumdar@caltech.edu},
    $^d$\email{ratliffl@uw.edu}}
    \affil[1]{\footnotesize Department of Computing and Mathematical Sciences, Caltech, USA}
    \affil[2]{\footnotesize Department of Electrical and Computer Engineering, University of Washington, USA}
    \date{}
\fi

\maketitle

\begin{abstract}

We propose a framework for two-player infinite-dimensional games with cooperative or competitive structure. These games take the form of coupled partial differential equations in which players optimize over a space of measures, driven by either a gradient descent or gradient descent-ascent in Wasserstein-2 space.
We characterize the properties of the Nash equilibrium of the system, and relate it to the steady state of the dynamics.
In the min-max setting, we show, under sufficient convexity conditions, that solutions converge exponentially fast and with explicit rate to the unique Nash equilibrium.  Similar results are obtained for the cooperative setting. We apply this framework to distribution shift induced by interactions among a strategic population of agents and an algorithm, proving additional convergence results in the timescale-separated setting. 
We illustrate the performance of our model on (i) real data from an economics study on Colombia census data, (ii) feature modification in loan applications, and (iii) performative prediction.
The numerical experiments demonstrate the importance of distribution-level, rather than moment-level, modeling. 
\end{abstract}



\section{Introduction}\label{sec:introduction}
Gradient flows generalize gradient descent to infinite dimensional spaces, for instance when probability distributions evolve in the direction of steepest descent for a given functional and metric. Recent years have seen a surge in the application of gradient flow theory including work in Monte Carlo sampling \cite{samplingDurmus,samplingMa}, generative modeling \cite{kwon_score-based_2022}, image processing (image registration, warping, shape classification, image segmentation and image restoration) \cite{Burger2009,Haker2004,M2AN_2015__49_6_1745_0,RabinPapadakis}, 
modeling the behavior of plastic materials in material science \cite{mielke2015rate}, or biological systems \cite[Ch 8]{perthame_parabolic_2015}. The majority of the gradient flow literature considers a single species, that is, the evolution of a single distribution, yet in many applications settings
multiple distributions evolve simultaneously. In these \textit{multispecies} systems, each species still minimizes its own functional, but this functional can depend on the other species. This interdependence results in a system of coupled PDEs that globally can no longer be viewed as a gradient flow on a single functional. Due to the loss of the gradient flow structure, this class of PDEs can exhibit complex dynamics like cycling or chaos. Such systems model many types of phenomena, ranging from chemotaxis \cite{espejo_simultaneous_2009,horstmann_nonlocal_2009,wolansky_chemotactic_2016} to opinion dynamics \cite{during_boltzmann_2009}. More recently, such coupled structures have been shown to arise also in machine learning in the literature on distributionally robust optimization \cite{delage_distributionally_2010,duchi_learning_2021,netessine_wasserstein_2019,lin_distributionally_2022,ma_provably_2022} and distribution shift in machine learning \cite{conger_strategic_2023}. 
In such settings, natural questions that arise include the existence and characterization of solutions and steady states, and convergence of the coupled PDE system to these steady states (if they exist). Current results focus on proving existence of solutions for some of these systems 
\cite{alasio_existence_2022,carrillo_measure_2020,MR4572120,giunta_local_2022,jungel_nonlocal_2022},
but less is known about their long-time behavior.

In this work we study multispecies systems that arise from competitive and cooperative games. In these games the two species either seek to maximize the same functional, or compete with one maximizing and the other minimizing the functional---i.e., they are infinite dimensional players in a cooperative or min-max game. We show that the steady states of the resulting system of partial differential equations (PDEs) coincide with the equilibria of the underlying games and characterize their rates of convergence under natural structural assumptions on the corresponding functionals. The convergence of Wasserstein-2 gradient flows for min-max problems over spaces of measures was recently posed as an open question in \cite{wang_open_2024}, and our analysis provides an answer to these questions for displacement convex-concave functionals over unbounded sets.

We use these results to investigate the long-term effects of strategic interactions in driving distribution shift in real-world machine learning contexts.
In many machine learning systems, agents whose data is analyzed by the system are incentivized to manipulate their data to achieve a desired output. Additionally, distribution shift can occur naturally, or agents share information that causes other players to evolve. This behavior is not well-understood and has become a subject of recent interest; see for instance \cite{agarwal_minimax_2022,izzo_how_2021,lei_near-optimal_2021,miller_outside_2021,miller_effect_2020,perdomo_performative_2020,wiles_fine-grained_2021}. In settings where the objective of the learning algorithm opposes that of the agents, the update process can be modeled as a min-max problem over a large number of agents, which in a mean-field limit can be analyzed as an optimization problem over measures. In particular, we incorporate intra-agent interactions in the model via an interaction potential, exogenous shifts, and strategic responses to the algorithm. We illustrate how these model components capture rich distributional behavior (see Section \ref{subsec:colombia_data}) and can show disparate effects of retraining among subpopulations (see Section \ref{subsec:loan_applications}). 
The implementation in Section \ref{subsec:loan_applications} uses a particle method, highlighting how real-world economic settings of many agents can be analyzed via a mean-field description at the PDE level.

In this work, we consider two-species systems with an energy functional containing potential terms, self-interaction kernels, diffusion, and a coupling term which is linearly dependent on both species. Each species evolves according to a Wasserstein-2 gradient flow with respect to its own energy differential, in the direction of steepest ascent or descent. The resulting dynamical system is a joint gradient flow in the setting where both species descend, or a gradient descent-ascent flow, in the setting of opposing dynamics. In line with intuitive notions from game theory, we call the joint gradient flow setting the \textit{cooperative} setting, because one can view the resulting dynamics as a game in which both players aim to achieve the same objective of minimizing the same energy. We name the gradient descent-ascent case the \textit{competitive} setting, due to the zero-sum game structure in which one player aims to maximize a function which the other player aims to minimize.
Let the energy in the cooperative setting be defined as $F_a:\P(\R^{d_1})\times\P(\R^{d_2})\to\R\cup\{\infty\}$ and in the competitive setting as $F_c:\P(\R^{d_1})\times\P(\R^{d_2})\to\R\cup\{\infty\}$, where $\P(\R^d)$ is the space of probability measures on $\R^d$,
\begin{subequations}
\begin{align}
    F_a(\rho,\mu) &= \iint f(z,x) \d \rho(z) \d \mu(x) + \cR(\rho) + \U(\mu) \label{eq:energy_aligned}\,, \\
    F_c(\rho,\mu) &= \iint f(z,x) \d \rho(z) \d \mu(x) - \cR(\rho)+\U(\mu)\,, \label{eq:energy_competitive}
\end{align}
\end{subequations}
where \begin{align*}
    \cR(\rho) &=  \alpha \Hc(\rho) + \frac{1}{2}\int  W_1 \ast \rho(z) \,\d\rho(z) + \int V_1(z) \d \rho(z)\,, \\ 
    \U(\mu) &= \beta \Hc(\mu) + \frac{1}{2}\int  W_2 \ast \mu(x)\,\d\mu(x) + \int V_2(x) \d \mu(x)\,,
\end{align*}
with $\alpha,\beta\ge 0$. Here, we denote by $f:\R^{d_1}\times \R^{d_2}\to \R$ the function governing coupling forces between the species $\rho$ and $\mu$, by $\Hc(\eta) :\P(\R^d)\to \R\cup\{\infty\}$ the entropy functional
$$\Hc(\eta) = \begin{cases}
    \int \eta \log \eta & \text{if }\eta \ll \mathcal L^{d} \\
    +\infty& \text{otherwise}
\end{cases}\,,$$
for $\mathcal L^{d}$ the $d$-dimensional Lebesgue measure, by $V_i:\R^{d_i}\to\R$ external potentials, and by $W_i:\R^{d_i}\to \R$ interaction potentials. For $N_z$ interacting particles $\{z^{(i)}_t\}_{i=1}^{N_z}$ representing agents from species $\rho$, and $N_x$ interacting particles $\{x^{(j)}_t\}_{j=1}^{N_x}$ representing agents from species $\mu$, consider their evolution given by 
\begin{subequations}\label{eq:particles}
\begin{align}
    \d z_t^{(i)} &= \pm\frac{1}{N_x}\sum_{j=1}^{N_x}\nabla_z f\big(z_t^{(i)},x_t^{(j)} \big) \d t -\frac{1}{N_z}\sum_{k=1}^{N_z} \nabla_z W_1\big(|z_t^{(i)}-z_t^{(k)}|\big) \d t -\nabla V_1\big(z_t^{(i)}\big)\d t + \sqrt{2\alpha}\d B_t^{(i)} \,,\\
    \d x_t^{(j)} &= -\frac{1}{N_z}\sum_{i=1}^{N_z}\nabla_x f\big(z_t^{(i)},x_t^{(j)} \big)\d t -\frac{1}{N_x}\sum_{k=1}^{N_x} \nabla_x W_2\big(|x_t^{(j)}-x_t^{(k)}|\big) \d t-\nabla V_2\big(x_t^{(j)}\big)\d t + \sqrt{2\beta}\d B_t^{(j)}\,,
\end{align}
\end{subequations}
where the first equation appears with a minus sign in the cooperative setting, and with a plus sign in the competitive setting. In this work, we focus on the mean-field formulation of this game for all our theoretical results, and make use of the particle model above as a way of numerically approximating the corresponding mean-field system of coupled PDEs. Let us denote by $\Wass_2$ the Wasserstein-2 metric, and $\nabla_{\Wass_2,\eta} F$ the Wasserstein-2 gradient of $F$ with respect to $\eta$.
 Then the mean-field dynamics in the cooperative setting are
\begin{align}\label{eq:dynamics_aligned}
    \partial_t \rho = -\nabla_{\Wass_2,\rho} F_a(\rho,\mu)\,, \qquad \partial_t \mu = -\nabla_{\Wass_2,\mu}F_a(\rho,\mu)\,.
\end{align}
In the competitive case, the dynamics are
\begin{align}\label{eq:dynamics_competitive}
    \partial_t \rho = \nabla_{\Wass_2,\rho} F_c(\rho,\mu)\,, \qquad \partial_t \mu = -\nabla_{\Wass_2,\mu}F_c(\rho,\mu)\,.
\end{align}
Both systems can be viewed as a two-player game over measures, rather than a mean-field game, since the dynamics coincide with a mean-field game solution only if the optimal trajectory is a gradient flow. In a mean-field game, each player optimizes over a vector field trajectory, while in this setting the velocity field trajectory comes directly from the gradient of the first variation of the energy.
The analysis of the dynamics in the cooperative setting \eqref{eq:dynamics_aligned} proceeds similarly to the approach in \cite{carrillo_kinetic_2003}, in which an HWI inequality is proven for a single species and log-Sobolev and Talagrand inequalities follow.
However, because the dynamics in the competitive setting \eqref{eq:dynamics_competitive} no longer have a gradient flow structure, the classical gradient flow techniques no longer apply.

In this paper we show, under sufficient convexity conditions, that the competitive dynamics \eqref{eq:dynamics_competitive} converge exponentially fast in the joint Wasserstein-2 metric, with the rate dependent on the displacement convexity of $F_c$ with respect to $
\mu$ and displacement concavity of $F_c$ with respect to $\rho$.
We prove that any two solution pairs $(\rho,\mu)$ and $(\rhot,\tilde \mu)$ to \eqref{eq:dynamics_competitive} contract in $\Wass_2^2 \times \Wass_2^2$, the squared joint Wasserstein-2 metric. Based on this result, we then show existence of a unique steady state for the dynamics. In order to show uniform boundedness of the second moments for both species, we show that they converge exponentially fast to a ball, and then remain inside that ball for all time. 
Finally, we show that the steady state is in fact a critical point of $F_c$, and the unique Nash equilibrium.
While the convexity and smoothness assumptions can be generalized, even mild relaxations on the lower-bounds in finite dimensions do not give the same guarantees. For example, in Euclidean space, assuming that the energy satisfies a Polyak \L ojasiewicz (P\L) condition instead of convexity results in non-uniqueness of Nash equilibria. 
With respect to convexity, our results mirror the state-of-the-art guarantees existing for finite-dimensional games. However, relaxing the regularity assumptions on the functionals is likely possible, and an interesting direction of future research.

\subsection{Related Literature}
The existence and convergence results utilize tools from long-time behavior analysis of PDEs and game theory, and apply to multispecies PDEs with a cooperative or competitive structure. Our results extend what is known in min-max problems with applications in machine learning, DRO, and strategic distribution shifts. 

\paragraph{Coupled PDEs} 
The structure of the models \eqref{eq:dynamics_aligned} and \eqref{eq:dynamics_competitive} is closely related to multispecies PDEs in a variety of application areas, including chemotaxis \cite{espejo_simultaneous_2009,horstmann_nonlocal_2009,wolansky_chemotactic_2016,wolansky_multi-components_2002}, opinion formation \cite{during_boltzmann_2009}, pedestrian dynamics \cite{appert-rolland_two-way_2011}, population biology \cite{chen_minimal_2014,di_francesco_nonlocal_2016} and cell-cell adhesion \cite{MR4758375}. Recent progress focuses on well-posedness questions  \cite{alasio_existence_2022,MR4596268,carrillo_measure_2020,di_francesco_nonlinear_2018,francesco_measure_2013}, connections to related models via limiting procedures \cite{MR4572120,doumic_multispecies_2023}, and asymptotic pattern formation \cite{burger_sorting_2018,MR4650881}. The mathematical theory for multispecies PDEs is still nascent, and, as the above list of works demonstrates, even if the equations exhibit a gradient flow structure, general results characterizing the asymptotic behavior of solutions (especially results achieving exponential convergence with explicit rates) are rare and concern rather special modeling choices. Here, we present a framework through which long-time behavior can be analyzed. A number of recent works study two-species systems, some of which can be treated with our framework.
In \cite{carrillo_zoology_2018}, a numerical method for computing solutions to two-species non-local cross-diffusion models is used to analyze the steady states of these systems. The two-species model in \cite{duong_coupled_2020} considers potential terms in addition to the cross-diffusion and self-interaction terms, which is closer to our model, and proves a mean-field limit from the particle stochastic differential equation system to the PDE limit. In \cite{doumic_multispecies_2023}, the authors prove conditions under which measures converge to Diracs under cross-diffusion, without any self-diffusion. We allow for the possibility of diffusion which we show leads to Lebesgue-measurable steady states with support over the entire space. 

\paragraph{Particle Systems}
Viewing the coupled system of PDEs as a mean-field description of a large number of interacting agents of two types as in \eqref{eq:particles} offers another powerful analysis tool. A number of works \cite{cai_convergence_2024,lu_two-scale_2023,wang_exponentially_2023} provide particle methods to find mixed Nash equilibria. We study a more general energy functional but, like \cite{cai_convergence_2024}, we restrict the functional to have convexity-concavity structure, and use this to prove the existence of a Nash equilibrium rather than assuming it.

\paragraph{Game Theory}
Although the dynamical system we study is a system of PDEs, the steady state of the system can be analyzed through the lens of game theory and optimization. In game theory, existence of equilibria in min-max problems has been studied over spaces of (1) deterministic strategies or (2) probabilistic strategies over compact sets \cite{glicksberg_further_1952,simons_minimax_1995}.
Recent progress in developing algorithms to solve infinite-dimensional min-max optimization problems includes \cite{Razvan2024}, which proposes a mirror ascent-descent scheme to compute a solution to a min-max problem over convex sets of measures and assumes existence of an equilibrium, while \cite{garcia_trillos_adversarial_2024,liu_infinite-dimensional_2021,lu_two-scale_2023,wang_exponentially_2023} propose other gradient ascent-descent schemes either under the assumption that a Nash equilibrium exists or that the optimization is over a convex set.  
We build upon these works by proving existence of a unique Nash equilibrium for a general class of energy functionals over unbounded sets of measures, and prove that gradient ascent-descent in Wasserstein-2 converges exponentially to this equilibrium. This addresses the strongly displacement convex-concave setting of the open problem posed in \cite{wang_open_2024}. 

\paragraph{Applications in Machine Learning}
Solving min-max games is an important problem for many applications in machine learning which can be formulated as min-max problems over the space of distributions, such as distributionally robust optimization (DRO) \cite{lin_distributionally_2022}, strategic distribution shift~\cite{conger_strategic_2023,perdomo_performative_2020,zrnic_who_2021}, and generative adversarial networks (GANs) \cite{AGGARWAL2021100004,Goodfellow14}. Existing results in DRO \cite{delage_distributionally_2010,duchi_learning_2021,netessine_wasserstein_2019,lin_distributionally_2022,ma_provably_2022} for machine learning considers optimization over a \textit{bounded} set of measures. Our setting builds upon this by considering optimization over an \textit{unbounded} set of measures. Other applications of DRO include portfolio selection \cite{fonseca_decision-dependent_2023,WU2023513} and train freight optimization \cite{AN2024110228}. In support of developing numerical methods, duality structures of distributionally robust optimization problems are studied \cite{gao_distributionally_2023,zhang_short_2024} as well as sensitivity analysis \cite{bartl_sensitivity_2021}. By connecting the gradient-flow structure of min-max problems to multispecies PDEs, analysis tools and numerical methods from the PDE literature become available for these applications.

\subsection{Contributions}
Our contributions in this work are twofold; first on the theoretical side, our theorems extend what is known in the PDE literature and optimization literature. Secondly, we apply our framework to illustrate the importance of modeling distribution shift in strategic populations in machine learning.
\paragraph{Theory}
This framework sits at the intersection of PDE analysis and optimization, providing contributions to each of these fields. From an optimization perspective, the existence of a Nash equilibrium over measures on unbounded sets has been an open question. Since existence is unknown, there are no systematic tools for computing equilibria and in particular, convergence of gradient descent-ascent to an equilibrium is an open problem \cite{wang_open_2024}. We expand this area of game theory in two key ways.
\begin{enumerate}
    \item Classical proofs for existence of Nash equilibria assume optimization over compact spaces of measures; we prove results without this assumption by showing contraction in $\P_2 \times \P_2$.
    \item By analyzing distributions over action spaces rather than deterministic actions, the achieved equilibrium is not restricted to a pure Nash equilibrium; it can be a mixed Nash equilibrium. Outside of specific games, such as ones with a finite number of actions or structure that allows direct computation via calculus of variations, computing mixed Nash equilibria over continuous action spaces is difficult to solve in the general setting. Our results suggest that the gradient ascent-descent structure in Wasserstein-2 offers a solution.
\end{enumerate}

From a PDE perspective, this setting opens the door 
for using techniques from calculus of variations and gradient flows in metric spaces.
In particular, it can be framed in the language of  multi-species systems, a field for which only very few and recent results exist on asymptotics via entropy methods. We show the existence of a unique steady state and exponential convergence to it with explicit rates in four different two-species settings. This extends what is known about long-time asymptotics for systems of coupled PDEs; in particular, the technical contributions include the following.
    \begin{enumerate}
        \item In the cooperative setting, classical functional inequalities are extended to the case of multiple species.
        \item In the competitive setting, convergence is proven without the use of timescale separation; this requires entirely different proof techniques both for existence of the steady state and convergence. Direct differentiation of $\Wass_2$ results in convergence, and existence of a unique Nash equilibrium is shown via contraction. A dynamical systems-type argument is used for uniform estimates of the second moments.
        \item We demonstrate in a particular application setting how a Danskin-type result (also known as an envelope theorem in analysis) can be obtained from basic assumptions using a $\Gamma$-convergence argument (see Proposition~\ref{prop:gamma_cv}). This removes a key assumption in \cite{conger_strategic_2023} on the differentiability of the best response (see \cite[Lemma 29]{conger_strategic_2023}). Such a $\Gamma$-convergence approach is expected to generalize to other choices of functionals.
        
    \end{enumerate}
\paragraph{Application}
One particular setting in which models of type \eqref{eq:dynamics_aligned} and \eqref{eq:dynamics_competitive} appear is when machine learning algorithms interact with strategic populations \cite{conger_strategic_2023}. In many real-world settings, populations dynamically adapt their strategy based on algorithm behavior. Optimization methods for algorithms do not usually account for this data manipulation, and we provide examples illustrating why modeling distribution shift in the face of learning is critical for improved performance.
\begin{enumerate}
    \item We illustrate our model on real data from an economics study (see Section~\ref{subsec:colombia_data}), a setting in which agents manipulated data in response to the action of an algorithm, showing that our model is able to accurately capture such behavior.
    \item We show the importance of modeling distribution shift in detail. A state-of-the-art performative prediction method, based on mean shift, is outperformed when the classifier follows a simple gradient descent scheme. We also illustrate how modeling population interactions can be overlooked when looking at classifier accuracy, but these interaction terms matter when considering classifier performance on certain subpopulations.
\end{enumerate}

\subsection{Paper Structure}
In Section~\ref{sec:preliminaries}, we provide relevant definitions and notation. Section~\ref{sec:results} contains the key assumptions and main results for the cooperative and competitive settings. In Section~\ref{sec:application}, we discuss an application of the model to strategic distribution shift in machine learning, with timescale-separated convergence results. Numeric examples and insights are shown in Section~\ref{sec:numerics}. The proof of the key convergence results in the cooperative and competitive settings are postponed to Sections~\ref{sec:proof_thm_aligned} and \ref{sec:competitive_proof} respectively. Appendices~\ref{sec:fast_x_proof} and \ref{sec:fast_rho_proof} contain proofs for the timescale-separated settings and in Appendix~\ref{sec:extra_lemmas} we collect the supporting technical lemmas.

\section{Preliminaries}\label{sec:preliminaries}
This section provides definitions and notation used throughout the paper. $\Id_d$ denotes the $d\times d$ identity matrix, and $\id$ denotes the identity map. $\Hess{f}$ denotes the Hessian of $f$ in all variables, while $\nabla^2_{x}{f}$ denotes the Hessian of $f$ in the variable $x$ only. 
The notation $\mathbb{I}\{B\}$ is an indicator function for the set $B$. Unless otherwise specified, $\norm{\cdot}$ notes the Euclidean norm for vectors and $\norm{\cdot}_2$ is the induced 2-norm when the argument is a matrix. Let $L_+^1(\R^d)=\{\mu\in L^1(\R^d)\, : \, \mu\ge 0\ a.e.\}$.
The narrow topology is defined as convergence in duality with continuous bounded functions and the weak-* topology is defined in duality with continuous functions vanishing at infinity. Throughout the manuscript, we will use the following related notion of convergence, which we refer to as \emph{weak topology}.
\begin{definition}[Weak Convergence]\label{def:weakcv}
    A sequence of measures $(\rho_n)$ converges in the \emph{weak topology}, denoted by $\rho_n \rightharpoonup \rho$, when $(\rho_n)$ converges narrowly, that is, in duality with all continuous bounded functions, and there exists a uniform second moment bound for $(\rho_n)$.
\end{definition}
Note that weak convergence implies narrow convergence, which implies weak-* convergence; the converse however does not hold. If a sequence converging weak-* also has a uniform second moment bound, then tightness follows from Markov's inequality \cite{ghosh} and so the masses also converge; consequently, the sequence also converges narrowly and weakly.

The energy functionals we are considering are usually defined on the set of probability measures on $\R^{d}$, denoted by $\P(\R^{d})$. At times we abbreviate this as $\P$ if the underlying space is clear from context. The set $\P^{ac}(\R^d)$ denotes the set of probability measures on $\R^d$ that are absolutely continuous with respect to the Lebesgue measure. Throughout, we use the same notation for measures in $\P(\R^d)$ and their densities with respect to the Lebesgue measure. We also use $\tP(\R^d):=\P^{ac}(\R^d)\cup\{\rho\in \P(\R^d)\,:\, \rho=\delta_z \text{ for some } z\in\R^d\}$. If we consider the subset $\P_2(\R^{d})$ of probability measures with bounded second moment, 
$$\P_2(\R^{d}):=\left\{\rho\in \P(\R^{d})\,: \, \int_{\R^{d}} \|z\|^2\d\rho(z)<\infty\right \}\,,$$
then we can endow this space with the Wasserstein-2 metric,
    \begin{align*}
        \Wass_2(\mu, \nu)^2 = \inf_{\gamma\in\Gamma(\mu,\nu)} \int \norm{z-z'}^2 \d \gamma(z,z')
    \end{align*}
    where $\Gamma(\mu,\nu)\in \P_2(\R^{d}\times\R^{d})$ is the set of all joint probability distributions with bounded second moments and marginals $\mu,\nu$, i.e. $\mu (\d z) = \int \gamma(\d z,z')\d z'$ and $\nu(\d z') = \int \gamma(z,\d z') \d z$. 
 Throughout this paper, we set $z\in\R^{d_1}$ and $x\in\R^{d_2}$, and denote by $\Wbar$ the joint Wasserstein metric. 
\begin{definition}[Joint Wasserstein Metric]
    Denote by $\Wbar$ the metric over $\P_2(\R^{d_1})\times\P_2(\R^{d_2})$ given by
\begin{align*}
    \Wbar((\rho,\mu),(\tilde \rho,\tilde\mu))^2 = \Wass_2(\rho,\tilde\rho)^2 + \Wass_2(\mu,\tilde\mu)^2
\end{align*}
for all pairs $(\rho,\mu),(\tilde \rho,\tilde\mu) \in \P_2(\R^{d_1})\times\P_2(\R^{d_2})$. 
\end{definition}
Geodesic convexity in the Wasserstein-2 space $(\P_2(\R^{d_i}),\Wass_2)$ is known as displacement convexity.
\begin{definition}[Displacement Convexity \cite{mccann_convexity_1997}]\label{def:displacement_convexity}
    A functional $G:\P(\R^{d_i})\to\R$ is \emph{displacement convex} if for all $\rho_0,\rho_1\in\P_2(\R^{d_i})$ that are atomless we have
    \begin{align*}
        G(\rho_s) \leq (1-s)G(\rho_0) + s G(\rho_1)\,,
    \end{align*}
    where $\rho_s = [(1-s)\id+s\nabla \psi]_{\#}\rho_0$ is the displacement interpolant between $\rho_0$ and $\rho_1$  for all $s\in [0,1]$. 
    Further, $G:\P(\R^{d_i})\to\R$ is \emph{uniformly displacement convex with constant $\lambda>0$} if 
    \begin{align*}
        G(\rho_s) \leq (1-s)G(\rho_0) + s G(\rho_1)-s(1-s) \frac{\lambda}{2} \Wass_2(\rho_0,\rho_1)^2\,,
    \end{align*}
    where $\rho_s = [(1-s)\id+s\nabla \psi]_{\#}\rho_0$ is the displacement interpolant between $\rho_0$ and $\rho_1$.
\end{definition}

\begin{remark}
    In other words, $G$ is displacement convex (concave) if the function $G(\rho_s)$ is convex (concave) in $s\in[0,1]$ with $\rho_s = [(1-s\id+s\nabla \psi]_{\#}\rho_0$ being the displacement interpolant between $\rho_0$ and $\rho_1$. Contrast this with the classical notion of convexity (concavity) for $G$, where we require that the function $G((1-s)\rho_0+s\rho_1)$ is convex (concave). One can think of displacement convexity for an energy functional defined on $\P_2$ as convexity along the shortest path in the Wasserstein-2 metric (linear interpolation in the Wasserstein-2 space) between any two given probability distributions.
\end{remark}
 We will use $s$ to denote the interpolation parameter for geodesics, and $t$ to denote time related to solutions of \eqref{eq:dynamics_aligned}-\eqref{eq:dynamics_competitive}.
In fact, if the energy $G$ is twice differentiable along geodesics, then the condition $\ddss G(\rho_s) \geq 0 $ along any geodesic $(\rho_s)_{s\in[0,1]}$ between $\rho_0$ and $\rho_1$ is sufficient to obtain displacement convexity. Similarly, when $\ddss G(\rho_s) \geq \lambda \Wass_2(\rho_0,\rho_1)^2$, then $G$ is uniformly displacement convex with constant $\lambda>0$.
For more details, see \cite{mccann_convexity_1997} and \cite[Chapter 5.2]{villani-OTbook-2003}.
\begin{definition}[Relative Energy]\label{def:relative_energy}
    The relative energy of a functional $G$ is given by \newline $G(\gamma|\gamma_\infty)=G(\gamma)-G(\gamma_\infty)$, where $G(\gamma_\infty)$ is the energy at some reference measure $\gamma_\infty$.
\end{definition}

Using the first variation, we can express the gradient in Wasserstein-2 space, see for example \cite[Exercise 8.8]{villani-OTbook-2003}. More precisely,
    the gradient of an energy $G:\P_2(\R^{d_i})\to\R$ in the Wasserstein-2 space is given by
    $$
    \nabla_{\Wass_2} G(\rho)=-\div{\rho\nabla \delta_\rho G[\rho](x)}\,,
    $$
    where $\delta_\rho G[\rho](x)$ denotes the first variation of $G$ at $\rho$ (if it exists).
As a consequence, the infinite dimensional steepest descent in Wasserstein-2 space of a given energy $G:\P_2\to \R \cup \{+\infty \}$ can be expressed as the PDE
\begin{align}
    \partial_t \rho = -\nabla_{\Wass_2} G(\rho) = \div{\rho\nabla \delta_\rho G[\rho]}\,.
\end{align}
All the coupled gradient flows considered in this work have this Wasserstein-2 structure.

\paragraph{Steady states}

The main goal in our theoretical analysis is to characterize the asymptotic behavior for the models \eqref{eq:dynamics_aligned} and \eqref{eq:dynamics_competitive} as time goes to infinity.
The steady states of these equations are the natural candidates to be asymptotic profiles for the corresponding dynamics. Thanks to the gradient flow structure, we expect to be able to make a connection between critical points of the energy functionals and the steady states of the corresponding gradient ascent/descent dynamics. More precisely, any minimizer or maximizer is in particular a critical point of the energy, and therefore satisfies that the first variation is constant on disconnected components of its support. If this ground state also has enough regularity (weak differentiability) to be a solution to the equation, it immediately follows that it is in fact a steady state. To make this connection precise, we first introduce what exactly we mean by a steady state.

\begin{definition}[Steady states]\label{def:sstates}
    For $\rho_\infty,\mu_\infty\in \P_2$, the pair $(\rho_\infty,\mu_\infty)$ is a steady state for the systems 
\eqref{eq:dynamics_aligned}-\eqref{eq:dynamics_competitive} if 
    \begin{enumerate}
        \item [(i)] $\nabla W_1 \ast \rho_\infty \in L_{loc}^1(\R^{d_1})$, $\nabla W_2 \ast \mu_\infty \in L_{loc}^1(\R^{d_2})$,
        \item [(ii)] if additionally, $\alpha>0$, then $\rho_\infty\in W_{loc}^{1,2}(\R^{d_1}) \cap L^1_+(\R^{d_1})\cap L^\infty_{loc}(\R^{d_1})$ and $\|\rho_\infty\|_1=1$,
        \item [(iii)] if additionally, $\beta>0$, then $\mu_\infty\in W_{loc}^{1,2}(\R^{d_2})\cap L^1_+(\R^{d_1})\cap L^\infty_{loc}(\R^{d_1})$, $\norm{\mu_\infty}_1 = 1$,
        \item [(iv)]  $(\rho_\infty,\mu_\infty)$ satisfy \eqref{def:sstate_Ga} for dynamics \eqref{eq:dynamics_aligned} or \eqref{def:sstate_Gc} for dynamics \eqref{eq:dynamics_competitive}\,.
    \end{enumerate}
    The conditions \eqref{eq:sstates} are given by
    \begin{subequations}\label{eq:sstates}
    \begin{align}
   \nabla_z \delta_\rho F_a[\rho_\infty,\mu_\infty](z) = 0\,, \quad \nabla_x \delta_\mu F_a[\rho_\infty,\mu_\infty](x)&=0 \quad \forall z\in\supp(\rhoinf)\,, x\in\supp(\mu_\infty)\label{def:sstate_Ga} \\
   \nabla_z \delta_\rho F_c[\rho_\infty,\mu_\infty](z) = 0\,, \quad \nabla_x \delta_\mu F_c[\rho_\infty,\mu_\infty](x)&=0 \quad \forall z\in\supp(\rhoinf)\,,x\in\supp(\mu_\infty)\label{def:sstate_Gc}
\end{align}
\end{subequations}
    in the sense of distributions.
\end{definition}

We define the Nash equilibrium of a game, and later show that the steady state of the dynamics in the zero-sum setting is in fact a Nash equilibrium.
\begin{definition}[Nash Equilibrium]
    A pair of measures $\gamma_*=(\rho_*,\mu_*)\in\P(\R^{d_1})\times \P(\R^{d_2})$ is a Nash equilibrium for the competitive setting if it satisfies
    \begin{subequations}\label{def:F_nash}
        \begin{align}
        F_c(\rho_*,\mu_*) &\geq F_c(\rho,\mu_*) \quad \forall\ \rho\in\P(\R^{d_1}) \label{eq:nash_rho}\\
        F_c(\rho_*,\mu_*) &\leq F_c(\rho_*,\mu) \quad \forall\ \mu \in \P(\R^{d_2})\,. \label{eq:nash_mu}
    \end{align}
    \end{subequations}
\end{definition}

\section{Main Results}\label{sec:results}
The convergence analysis of these systems allows us to understand and predict the long-time behavior of the dynamics. 
The asymptotics are given by the ground and saddle states of the energy functionals $F_a$ and $F_c$ respectively. We prove existence and uniqueness of the critical points of the functionals and, under sufficient convexity criteria, convergence with explicit rates.  

\begin{remark}[Cauchy-Problem]
To execute the arguments on convergence to equilibrium, we require sufficient regularity of solutions to the PDEs under consideration. 
In fact, it is sufficient if we can show that equations \eqref{eq:dynamics_aligned} - \eqref{eq:dynamics_competitive} can be approximated by equations with smooth solutions. Albeit tedious, these are standard techniques in the regularity theory for partial differential equations; see for example \cite[Proposition 2.1 and Appendix A]{carrillo_kinetic_2003}, \cite{otto_generalization_2000}, \cite[Chapter 9]{villani-OTbook-2003}, and the references therein. Similar arguments as in \cite{desvillettes_spatially_2000} are expected to apply to the coupled gradient flows considered here to guarantee existence of smooth solutions with fast enough decay at infinity. In this work, we do not focus on the existence and regularity of solutions.
\end{remark}

\subsection{Assumptions}\label{subsec:assumptions}
The key results on existence and uniqueness of a ground state or saddle point, as well as the convergence behavior of solutions, depend on convexity (concavity) of the corresponding functionals. The notion of convexity that we will employ for energy functionals in the Wasserstein-2 geometry is \textit{(uniform) displacement convexity}, which is analogous to (strong) convexity in Euclidean spaces; see Definition~\ref{def:displacement_convexity}.

We use subsets of the following assumptions in the cooperative and competitive cases.

\begin{assumption}\label{assump:f_lower}
The coupling potential $f$ satisfies $f\in C^2(\R^{d_1} \times \R^{d_2}, \R)$, and for all $(z,x)\in\R^{d_1}\times\R^{d_2}$,
\begin{itemize} 
\item[(i)] Cooperative Setting: There exists $\lambda_f\in\R$ such that $\Hess{f}(z,x) \succeq \lambda_f \Id_{d_1\times d_2}$. That is, $f$ is $\lambda_f$-convex. Additionally, $f \ge 0$.
\item[(ii)] Competitive Setting: There exists $\lambda_{f,1}\in\R$ such 
that $-\nabla_z^2 f(z,x) \succeq \lambda_{f,1} \Id_{d_1}$ and $\lambda_{f,2}\in\R$ such that 
and $\nabla_x^2 f(z,x) \succeq \lambda_{f,2} \Id_{d_2}$. That is, $f$ is $\lambda_{f,1}$-concave in $z$ and $\lambda_{f,2}$-convex in $x$.
\end{itemize} 
\end{assumption}

\begin{assumption}\label{assump:V_lower}
    The external potentials $V_i:\R^{d_i}\to[0,\infty)$ are in $C^2$ and satisfy lower Hessian bounds:
    there exists $\lambda_{V,i}\in\R$ such that $\nabla^2 V_i \succeq \lambda_{V,i} \Id_{d_i}$.
\end{assumption}

\begin{assumption}\label{assump:W_lower}
The interaction potentials $W_i:\R^{d_i}\to [0,\infty)$ are in $C^2$, are symmetric
and satisfy lower Hessian bounds: 
there exists $\lambda_{W,i}\ge 0$ such that $\nabla^2 W_i \succeq \lambda_{W,i} \Id_{d_i}$.
\end{assumption}
In the timescale-separated competitive setting, we use additional assumptions which provide upper bounds on the Hessian terms; see Section~\ref{sec:application}.

\begin{remark}
    The assumptions above and in Section~\ref{sec:application} are not intended to be optimal; rather, they provide conditions analogous to conditions in the finite-dimensional setting under which convergence is guaranteed. For details on how some convexity assumptions can be weakened in combination with stronger assumptions on other terms, see \cite{carrillo_kinetic_2003}. The assumptions that the above potentials are in $C^2$ is strong and can likely be weakened without losing the main convergence guarantees. For the application settings considered in Sections~\ref{sec:application} and \ref{sec:numerics} and in \cite{conger_strategic_2023}, all potentials are in $C^2$; however, more singular potentials are common in other settings such as interacting species in math-biology.
\end{remark}

\subsection{The Cooperative Setting}
The cooperative setting can be viewed as a class that includes potential games. From the PDE perspective, the system has a joint gradient flow structure which is utilized to prove convergence.
\begin{theorem}\label{thm:PDE_aligned} 
   Suppose that Assumptions~\ref{assump:f_lower}(i), \ref{assump:V_lower}, and \ref{assump:W_lower} are satisfied with $$\lambda_a \coloneqq \lambda_f+\min\{\lambda_{V,1},\lambda_{V,2} \} > 0.$$ 
   Consider solutions $\gamma_t:=(\rho_t,\mu_t)$ to the dynamics \eqref{eq:dynamics_aligned} with initial condition satisfying $\gamma_0\in\P_2(\R^{d_1})\times \P_2(\R^{d_2})$, $F_a(\gamma_0)<\infty$, and
   \begin{align*}
       \int \norm{\nabla_z\delta_\rho F_a[\gamma_0](z)}^2 \d \rho_0(z) + \int \norm{\nabla_x\delta_\mu F_a[\gamma_0](x)}^2 \d \mu_0(x) < \infty \,.
   \end{align*}
   Then the following hold:
    \begin{enumerate}[topsep=0ex]
        \setlength\itemsep{0.01em}
        \item[(a)] There exists a unique minimizer $\gamma_\infty =(\rho_\infty,\mu_\infty)$ of $F_a$ in $\P \times \P$, which is also a steady state for equation \eqref{eq:dynamics_aligned}. Further, $\gamma_\infty\in \P_2 \times \P_2$.
        \begin{enumerate}
            \item [(i)] If $\alpha >0$, then $\rho_\infty\in  L_+^1(\R^{d_1}) \cap C^2(\R^{d_1})$ and $\supp(\rho_\infty)=\R^{d_1}$.
            \item [(ii)] If $\beta > 0$, then $\mu_\infty\in L_+^1(\R^{d_2})\cap C^2(\R^{d_2})$ and $\supp(\mu_\infty)=\R^{d_2}$.
        \end{enumerate} 
        \item[(b)] The solution $\gamma_t$ converges exponentially fast in $F_a (\cdot \,|\, \gamma_\infty)$ and $\Wbar$,
        \begin{equation*}
    F_a(\gamma_t\,|\,\gamma_\infty)\le e^{-2\lambda_a t} F_a(\gamma_0\,|\,\gamma_\infty)\, \quad \text{ and }\quad
    \Wbar(\gamma_t,\gamma_\infty) \le c e^{-\lambda_a t}
    \quad \text{ for all } t\ge 0\,,
\end{equation*}
where $c>0$ is a constant only depending on $\gamma_0$, $\gamma_\infty$ and the parameter $\lambda_a$.
    \end{enumerate}
\end{theorem} 
To prove existence and uniqueness, we leverage classical techniques in the calculus of variations. To obtain convergence to equilibrium in energy, our key result is an HWI-type inequality, providing as a consequence generalizations of the log-Sobolev inequality and the Talagrand inequality. Together, these inequalities relate the energy (classically denoted by $H$ in the case of the Boltzmann entropy), the metric (classically denoted by $W$ in the case of the Wasserstein-2 metric) and the energy dissipation (classically denoted by $I$ in the case of the Fisher information)\footnote{Hence the name HWI inequalities.}. Combining these inequalities with Gr\"onwall's inequality allows us to deduce convergence both in energy and in the metric $\Wbar$. See Section~\ref{sec:proof_thm_aligned} for a detailed proof.

\subsection{The Competitive Setting}
In the competitive setting, gradient descent by each player results in convergence to the unique Nash equilibrium.
\begin{theorem}\label{thm:convergence_Wbar_competitive}
    Suppose Assumptions~\ref{assump:f_lower}(ii), \ref{assump:V_lower} and \ref{assump:W_lower} are satisfied with $$\lambda_c\coloneqq \min \{\lambda_{f,1}+\lambda_{V,1},\lambda_{f,2}+\lambda_{V,2} \}>0.$$
    Consider solutions to \eqref{eq:dynamics_competitive} with initial condition $\gamma_0\in\P_2(\R^{d_1})\times\P_2(\R^{d_2})$ satisfying
    \begin{align*}
        \int \norm{\nabla_z\delta_\rho F_c[\gamma_0](z)}^2 \d \rho_0(z) + \int \norm{\nabla_x\delta_\mu F_c[\gamma_0](x)}^2 \d \mu_0(x) < \infty\,.
    \end{align*}
If $\alpha=0$, assume $\rho_0=\delta_{z_0}$ for some $z_0\in\R^{d_1}$. If $\beta=0$, assume $\mu_0=\delta_{x_0}$ for some $x_0\in\R^{d_2}$.
    Then the following hold:
    \begin{enumerate}
        \item[(a)] There exists a unique critical point $\gamma_*\in\tP_2\times \tP_2$ for $F_c$ which is also a steady state for equation~\eqref{eq:dynamics_competitive} and the unique Nash equilibrium.
        \begin{itemize}
            \item [(i)] If $\alpha >0$, then $\rho_*\in L^1_+(\R^{d_1})\cap C^2(\R^{d_1})$ and $\supp(\rho_*)=\R^{d_1}$.
            \item [(ii)] If $\beta > 0$, then $\mu_*\in L^1_+(\R^{d_2})\cap C^2(\R^{d_2})$ and $\supp(\mu_*)=\R^{d_2}$. 
        \end{itemize}
        \item[(b)] The solution $\gamma_t\coloneqq(\rho_t,\mu_t)$ to the dynamics \eqref{eq:dynamics_competitive} is in $\tP_2\times \tP_2$, has uniformly bounded second moments,
               \begin{align*}
               \exists\, K>0\,:\,\quad 
           \int \norm{z}^2 \d \rho_t(z) + \int \norm{x}^2 \d \mu_t(x) \le K \quad \forall\, t\ge 0\,,
       \end{align*}
        and converges exponentially fast in $\Wbar$ with rate $\lambda_c$,
        \begin{align*}
            \Wbar(\gamma_t,\gamma_*) \le e^{-\lambda_c t}\Wbar(\gamma_0,\gamma_*)\,.
        \end{align*}
    \end{enumerate}
\end{theorem}
This theorem tells us that the distributions $\rho_t$ and $\mu_t$ converge at a rate corresponding to the displacement concavity-convexity of the energy functional; we expect this to be true from similar analysis of finite-dimensional zero-sum games.

\begin{remark}\label{rem:mean_field_SDE}
    One can remove the assumption on the boundedness of the initial dissipation by using a coupling argument to show contraction of the dynamics rather than the expression for the derivative of the Wasserstein-metric as stated in \cite[Theorem 23.9]{Villani07}. Although our argument also gives exponential decay of the energy dissipation (Lemma~\ref{lem:velocities_in_L2}), which is stronger, we could instead skip this step and estimate $\d\left((Z_t,X_t)-(Z_t',X_t')\right)$ directly, where $(Z_t,X_t)$ and $(Z_t',X_t')$ both solve the mean-field SDE system corresponding to the PDE dynamics, with synchronous coupling of the noise.  The difference of these processes solves an ODE, and using the convexity assumptions on the the potentials, one also obtains exponential decay for this expression.
\end{remark}

\section{Application: Distribution Shift in Machine Learning}\label{sec:application}
Machine learning algorithms in real-world settings often update their parameters over time to improve performance; as more data is collected, it is natural for the algorithm to update based on more recent information. However, in common applications the distribution of data on which the algorithm is training may not be stationary over time. This phenomenon is known as \textit{distribution shift}, and is induced from a variety of causes, including mis-aligned incentives, interactions with other agents, and natural causes.

The model in \cite{conger_strategic_2023} proposes an energy functional which the algorithm seeks to minimize and which the population aims to maximize; some terms have dependence on both distributions while others model energies specific to the evolution of agents and algorithms respectively. The model in \cite{conger_strategic_2023} can be seen as a special case of the competitive setting considered here by choosing
\begin{align*}
    f(z,x) &= f_1(z,x)\,, \quad W_1(z) = W(z)\,, \quad \alpha>0\,, \quad \beta=0\,,\\
    V_1(z)&=-\alpha \log \rhot(z)\,, \quad V_2(x) = \int f_2(z,x) \d \pi(z) + \frac{\kappa}{2}\norm{x-x_0}^2\,,
\end{align*}
with all other terms set to zero and $d_1=d_2=d$. Here, $\rhot\in \P(\R^d)\cap L^1(\R^d)$ and $\pi\in \P(\R^d)$ are fixed reference measures.
Let us denote
\begin{align*}
     \V(\rho,\mu) = \iint f_1(z,x)\d \rho(z) \d \mu(x)+ \iint f_2(z,x) \d \pi(z) \d \mu(x)\,, \\
     \mathcal{R}(\rho)  =\frac{1}{2} \int W \ast \rho(z)\,\rho(z) + \alpha KL(\rho\,|\,\rhot)\,, \quad \mathcal{Q}(\mu) = \frac{\kappa}{2}\int \norm{x-x_0}^2\d \mu(x)\,.
\end{align*}
The functional $\V(\rho,\mu)$ models the cost which the algorithm seeks to minimize, and the population minimizes or maximizes depending on the setting. For example, when $\mu$ represents a binary classifier, the distribution $\pi$ models individuals carrying the true label 1, and the distribution $\rho(t)$ model individuals carrying a true label 0, where 0 and 1 denote the labels of two classes of interest. The term $\int f_1(z,x)\mu(t,\d x)$ represents a penalty for incorrectly classifying an individual with features $z$ with true label 0 when using the classifier $\mu(t,x)$. Analogously, the term $\int f_2(z,x)\pi(\d z)$ is large if $x$ incorrectly classifies the population $\pi$ that carries the true label 1. The functions $f_1$ and $f_2$ can be chosen according to the application; a standard choice for classification problems is the logarithmic loss \cite{miller_outside_2021}.

The functional $\mathcal{Q}(\mu)$ is a regularizer for the algorithm; this penalizes the classifier for selecting extreme learning parameters, and provides convexity for the loss function. The coefficient $\kappa>0$ parameterizes the strength of the regularizer.

The functional $\mathcal{R}(\rho)$ contains two terms, the Kulbeck-Leibler divergence (denoted $KL$), also called the relative entropy, and the interaction term driven by the potential $W$. The term $\alpha KL(\rho|\rhot)$ forces the evolution of $\rho(t)$ to approach $\rhot$. In other words, it penalizes (in energy) deviations from a given reference measure $\rhot$. In many application settings, we take $\rhot$ to be the initial distribution $\rho(t=0)$. The solution $\rho(t)$ then evolves away from $\rhot$ over time due to the other forces that are present. Therefore, the term $KL(\rho\,|\,\rhot)$ in the energy both provides smoothing of the flow and a penalization for deviations away from the reference measure $\rhot$.

The self-interaction term $W\ast\rho$ introduces non-locality into the dynamics, as the decision for any given individual to move in a certain direction is influenced by the behavior of all other individuals in the population. The choice of $W$ is application dependent. Very often, the interaction between two individuals only depends on the distance between them. This suggests a choice of $W$ as a radial function, i.e. $W(z)=\omega(|z|)$. A choice of $\omega:[0,\infty)\to\R$ such that $\omega'(r)>0$ corresponds to an \emph{attractive} force between individuals, whereas $\omega'(r)<0$ corresponds to a \emph{repulsive} force.

However, in many real-world applications of this model, either the population or the algorithm may update much faster than the other. For example, many government policies are updated on a much slower timescale compared with how quickly individuals can adjust their response. In settings like advertising algorithms, companies can update their algorithms faster than users will adjust search strategies or viewing patterns. We provide convergence results for these two settings for competitive objectives, taking the timescale separations to be large enough that one entity instantly minimizes or maximizes its objective while the other entity evolves over time. Here we consider the case $\mu(t)=\delta_{x(t)}$ where $x(t)$ solves $\dot x(t)=-\nabla_x L(\rho,x)$ with loss function $L(\rho,x)=\int f_1(z,x)\d\rho(z) + V_2(x)$. The energy in these two setting is denoted $G(\rho,x)$. 
The energy functional is given by
\begin{subequations}
    \begin{align*}
        G(\rho,x) \coloneqq  \int f_1(z,x)\d\rho(z) + \int f_2(z,x)\d\pi(z) + \frac{\kappa}{2}\norm{x-x_0}^2 -\mathcal{R}(\rho) \,.
    \end{align*}
\end{subequations}
When the algorithm updates quickly relative to the population, we consider dynamics given by
\begin{equation}\begin{aligned}\label{eq:dynamics_competitive_x_fast}
    \partial_t \rho = \nabla_{\Wass_2,\rho}G(\rho,x)|_{x=\bx(\rho)}\,, \qquad
   \bx(\rho) \coloneqq \amin_{\bar{x}\in\R^d} G(\rho,\bar{x})\,.
\end{aligned}\end{equation}

In the opposite setting where the population is fast relative the algorithm, we can consider the population immediately responding to the algorithm, which results in the dynamics
\begin{equation}\begin{aligned}\label{eq:dynamics_competitive_rho_fast}
    \ddt x = -\nabla_x G(\rho,x)|_{\rho=\br(x)}\,, \qquad
    \br(x) \coloneqq \amax_{\hat{\rho}\in\P} G(\hat{\rho},x)\,.
\end{aligned}\end{equation}

In this time-scale separated setting, model \eqref{eq:dynamics_competitive_x_fast} is a dynamic maximization of $G$ with respect to $\rho$ in Wasserstein-2 space, and an instantaneous minimization of $G$ with respect to the algorithm parameters $x$. Model~\eqref{eq:dynamics_competitive_rho_fast} is an instantaneous maximization of $G$ with respect to $\rho$ and a dynamic minimization of $G$ with respect to the algorithm parameters $x$. 

\subsection{Assumptions under Timescale Separation}
We begin by rewriting Assumptions~\ref{assump:f_lower}(ii), \ref{assump:V_lower} and \ref{assump:W_lower} in the timescale-separated setting (competitive objectives). 
From Assumption~\ref{assump:f_lower}(ii), we have $f_1\in C^2(\R^d\times\R^d)$ and (renaming parameters) assume
$$-\nabla_z^2 f_1(z,x) \succeq -\Lambda_1 \Id_{d}\,,\qquad \nabla_x^2 f_1(z,x) \succeq \lambda_1\Id_{d}$$ 
for some $\Lambda_1,\lambda_1\in\R$.
Assumption~\ref{assump:V_lower} is guaranteed by imposing 
$$-\nabla^2_{z} \log \rhot(z) \succeq \tilde\lambda \Id_d\,,\qquad \nabla^2_x f_2(z,x) \succeq \lambda_2 \Id_{d}$$ 
for some $\lambda_2, \lambdat \in \R$, together with $\log\rhot (\cdot),\, f_2(z,\cdot)\in C^2(\R^d)$, $0<\rhot\le 1$ and $0 \le \int f_2(z,\cdot)\d\pi(z)< \infty $. It follows that for any $R>0$ there exists a constant $c_2=c_2(R)\ge0$ such that
    \begin{equation}\label{eq:V_upper-b}
        \sup_{x\in B_R(0)} \int f_2(z,x) \d\pi(z) < c_2\,.
    \end{equation}
For Assumption~\ref{assump:W_lower}, we denote $\lambda_{W,1}$ by $\lambda_W$ and assume $W\in C^2(\R^d)$ is symmetric with $W\ge 0$. In summary,
\begin{gather*}
    \lambda_{f,1}=-\Lambda_1\,,\quad \lambda_{f,2}=\lambda_1 \,,\quad
    \lambda_{V,1}=\alpha\tilde\lambda\,,\quad 
    \lambda_{V,2}=\lambda_2+\kappa\,,\quad
    \lambda_{W,1}=\lambda_W\,,\quad \lambda_{W,2}=0\,.
\end{gather*}
To analyze timescale-separation, we use the following additional assumptions.
\begin{assumption}\label{assump:f_upper}
    The coupling potential $f_1$  satisfies for all $(z,x)\in\R^{d}\times\R^{d}$,
\begin{itemize}
\item[(a)] Upper Hessian bound: There exists $\ell_{1}\in\R$ such that $-\nabla_z^2 f_1(z,x) \preceq -\ell_1  \Id_{d}$. 
\item[(b)] Cross-terms: There exists $L \ge 0$ such that $\norm{\nabla_{xz}^2 f_1(z,x)}_2 \le L$. 
\end{itemize} 
\end{assumption}
Note that in \cite{conger_strategic_2023}, $\lambda_1=\ell_1$. In the results below, the value of $\ell_1$ does not actually play a role and can be chosen independently of $\lambda_1$.
\begin{assumption}\label{assump:V_upper} Upper Hessian bound for external potential:
    There exists $\Lambdat\in\R$ such that $$-\nabla^2_{z} \log \rhot(z) \preceq\Lambdat \Id_d\,.$$
\end{assumption}
\begin{assumption}\label{assump:W_upper}
Upper Hessian bound for the interaction potential:
there exists $\Lambda_W\in\R$ such that $$\nabla^2_z W \preceq \Lambda_{W} \Id_{d}\,.$$
\end{assumption}

\begin{assumption}\label{assump:f_loose_upper}
The functions $f_1,f_2$ satisfy for all $(z,x)\in\R^d\times \R^d$: There exist constants $a_i>0$ such that $$x \cdot \nabla_x f_i(z,x) \geq -a_i \quad \text{ for } i=1,2\,.$$
\end{assumption}

\subsection{Analysis of Competitive Objectives with Timescale Separation}  
In the timescale separated cases, the dynamics are 
\begin{align*}
    \partial_t \rho &= -\div{\rho \left[\nabla (f_1(\cdot,\bx(\rho)) - \alpha \log(\rho/\rhot) - W \ast \rho
 \right]}\,, \\
   \bx(\rho) &:= \amin_{\bar{x}\in\R^d} \int f_1(z,\bar x) \d \rho(z) + \int f_2(z',\bar x) \d \pi(z') + \frac{\kappa}{2} \norm{\bar x-x_0}^2\,
\end{align*}
for \eqref{eq:dynamics_competitive_x_fast}, and
\begin{align*}
        \ddt x &= -\nabla_x \left( \int f_1(z,x) \,\br(x)(\d z) + \int f_2(z',x) \d \pi(z') + \frac{\kappa}{2} \norm{x-x_0}^2 \right)\,, \\
    &\br(x) \coloneqq \amax_{\hat{\rho}\in\P} \int f_1(z,x) \d \hat\rho(z) - \alpha KL(\hat\rho|\rhot) - \frac{1}{2}\int \hat\rho W \ast \hat\rho\,.
\end{align*}
for \eqref{eq:dynamics_competitive_rho_fast}. Our results are summarized in the following theorems. 
\begin{theorem}[Fast Algorithm]\label{thm:convergence_fast_mu}
    Suppose Assumptions~\ref{assump:f_lower}(ii), \ref{assump:V_lower}, \ref{assump:W_lower}, 
    and \ref{assump:f_loose_upper}
    are satisfied with 
    $\lambda_b\coloneqq \alpha \tilde \lambda - \Lambda_{1}>0$ and $\lambda_d\coloneqq\kappa+\lambda_{1}+\lambda_2>0$.
    Define $G_b(\rho) \coloneqq G (\rho,\bx(\rho))$.
    Consider a solution $\rho_t$ to the dynamics \eqref{eq:dynamics_competitive_x_fast} with initial condition $\rho_0\in\P_2(\R^d)$ such that $G_b(\rho_0)<\infty$. Then the following hold:
    \begin{enumerate}[topsep=0ex]
    \setlength\itemsep{0.01em}
        \item[(a)] There exists a unique maximizer $\rho_\infty$ of $G_b(\rho)$, which is also a steady state for equation \eqref{eq:dynamics_competitive_x_fast}. Moreover, $\rho_\infty\in L^1_+(\R^d) \cap C(\R^d)$ with the same support as $\rhot$.
        \item[(b)] The solution $\rho_t$ converges exponentially fast to $\rho_\infty$ in $G_b(\cdot\,|\,\rho_\infty)$ and $\Wass_2$,
                \begin{equation*}
    G_b(\rho_t\,|\,\rho_\infty)\le e^{-2\lambda_b t} G_a(\rho_0\,|\,\rho_\infty)\, \quad \text{ and }\quad
    \Wass_2(\rho_t,\rho_\infty) \le c e^{-\lambda_b t}
    \quad \text{ for all } t\ge 0\,,
    \end{equation*}
    where $c>0$ is a constant only depending on $\rho_0$, $\rho_\infty$ and the parameter $\lambda_b$.
    \end{enumerate}
\end{theorem}

\begin{theorem}[Fast Population]\label{thm:convergence_fast_rho}
    Suppose Assumptions~\ref{assump:f_lower}(ii), \ref{assump:V_lower}-\ref{assump:W_upper}
    are satisfied with 
    $\lambda_b\coloneqq \alpha \tilde \lambda - \Lambda_{1}>0$ and $\lambda_d\coloneqq\kappa+\lambda_{1}+\lambda_2>0$.
    Define $G_d(x)\coloneqq G(\br(x),x)$. Then it holds:
    \begin{enumerate}[topsep=0ex]
    \setlength\itemsep{0.01em}
        \item[(a)] There exists a unique minimizer $x_\infty$ of $G_d(x)$ which is also a steady state for \eqref{eq:dynamics_competitive_rho_fast}.
        \item[(b)] The vector $x(t)$ solving the dynamics \eqref{eq:dynamics_competitive_rho_fast} with initial condition $x(0)\in\R^d$ converges exponentially fast to $x_\infty$  in $G_d$ and in the Euclidean norm:
        \begin{align*}
            \|x(t)-x_\infty\|&\le e^{-\lambda_{d} t}  \|x(0)-x_\infty\|\,,\\
            G_d(x(t))-G_d(x_\infty)&\le e^{-2\lambda_{d} t}\left(G_d(x(0))-G_d(x_\infty)\right)
        \end{align*}
        for all $t\ge 0$. Moreover, $G_d\in C^1(\R^d)$.
    \end{enumerate}
\end{theorem}

Theorems \ref{thm:convergence_fast_mu} and \ref{thm:convergence_fast_rho} state that we will observe exponential convergence in the timescale-separated cases. 
The proof of Theorem~\ref{thm:convergence_fast_mu} hinges on proving a generalized version of Danskin's Theorem, that is, showing that $\delta_\rho G_b(\rho) = \left(\delta_\rho G(\rho,x) \right)|_{x=b(\rho)}$. Displacement convexity of $G_b$ follows from this, and then standard HWI techniques apply. In order to prove a similar Danskin's result as an ingredient in the proof of Theorem~\ref{thm:convergence_fast_rho}, we first show that $G_d(x)$ is differentiable. The best response function $r(x)$ is defined implicitly as the result of a minimization problem, and we employ Gamma convergence tools to obtain the regularity result. For the proofs, see Appendices~\ref{sec:fast_x_proof} and \ref{sec:fast_rho_proof}.

For the time-scale separated setting, we only require additional assumptions in the variable that optimizes instantaneously: Assumption \ref{assump:f_loose_upper} concerning the $x$-variable is used in Theorem~\ref{thm:convergence_fast_mu}, and Assumptions~\ref{assump:f_upper}-\ref{assump:W_upper} concerning the $z$-variable appear in Theorem~\ref{thm:convergence_fast_rho}.
In particular, Assumption~\ref{assump:f_loose_upper} is used in Theorem~\ref{thm:convergence_fast_mu} to prove that the norm of the best response $\norm{\bx(\rho)}$ is uniformly bounded, while Assumptions~\ref{assump:f_upper}-\ref{assump:W_upper} are used in Theorem~\ref{thm:convergence_fast_rho} to prove convergence of the second moment of a sequence $(\rho_n)$ in order to obtain the $\Gamma$-convergence result.

\section{Insights from Numerical Experiments}\label{sec:numerics}
In this section, we provide examples of numerical experiments that illustrate how our framework models real-world data and implications of using various algorithm learning strategies while interacting with a dynamic, strategic population. First, we model data from an economic study of how local government officials are incentivized to misreport census data. Then we show how individuals modify loan application data to achieve a more desirable outcome from a classifier algorithm, comparing different interaction models among agents and showing how subpopulations are affected differently. Finally, we illustrate our strategic population model under a state-of-the-art performative prediction algorithm, showing that it is critical to consider distribution shift beyond only mean shift when optimizing performance. The numerical experiments are implemented using the finite volume method from \cite{carrillo_finite-volume_2015}.
\subsection{Census Data in Colombia}\label{subsec:colombia_data}
A study of Colombia census data \cite{camacho_manipulation_2011} from 1995 to 2003 investigates how local officials misreported data in order to obtain lower poverty index scores for their constituents. Households with a poverty index score below a given threshold receive government aid, a desired outcome of the census data collection. The algorithm for the poverty index score and threshold was release in 1997, and the distribution of scores shifts from a Gaussian-like shape to having a sharp drop-off above the threshold.

We model this dynamical system as a classification problem in the competitive setting with a suitable energy $G$; an algorithm with parameter $x$ (government aid threshold) aims to separate poverty index scores into ones which qualify for aid and ones that do not. In this setting, we keep the algorithm fixed as the population adjusts, that is, $\ddt x = 0$. Each family aims to be classified as qualifying for aid, regardless of their true label.
\begin{figure}
    \centering
    \begin{subfigure}[b]{0.47\textwidth}
        \includegraphics[width=\textwidth]{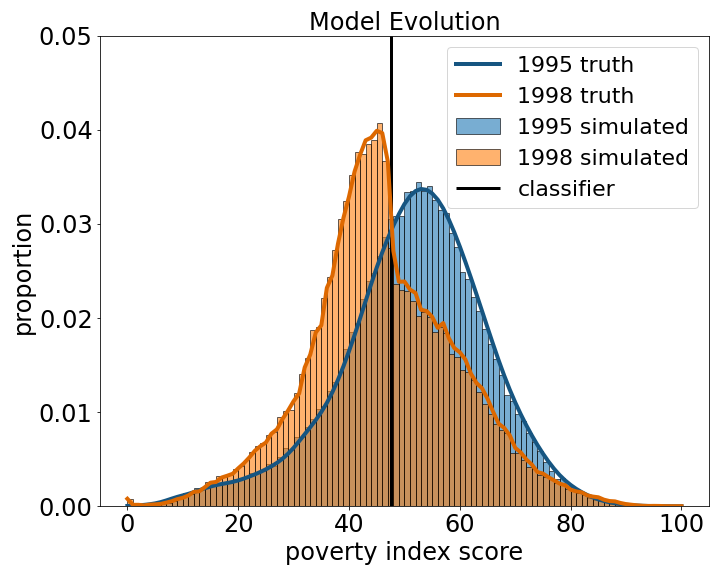}
    \end{subfigure}
    \begin{subfigure}[b]{0.47\textwidth}
        \includegraphics[width=\textwidth]{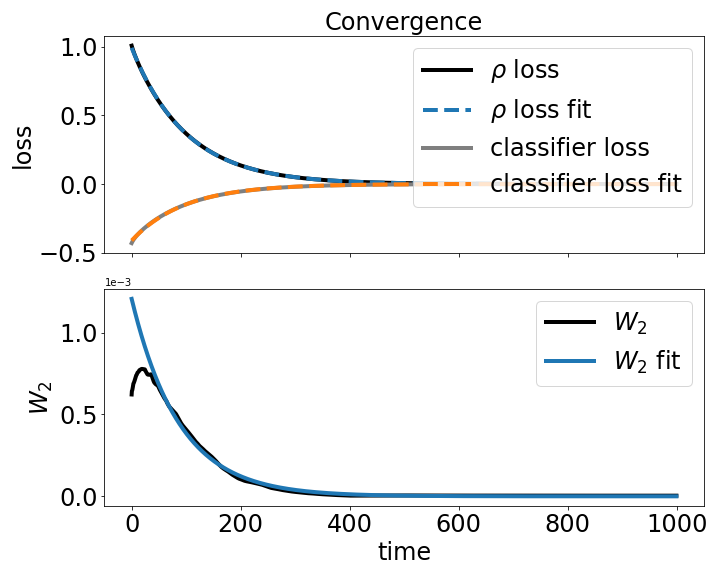}
    \end{subfigure}
    \caption{After the criteria for government aid was released in 1997, local officials misreported income data to increase the number of constituents qualifying for aid. The PDE \eqref{eq:dynamics_competitive} is able to capture the sharp drop at the classifier threshold. The convergence rate for the loss of the population and algorithm are $0.00995$ and $0.0102$; the convergence rate for $\Wass_2(\rho_t,\rho^{(98)})$, where $\rho^{(98)}$ is the steady state distribution, is also $0.0114$ which is similar to the expected rate of $0.01$. The expected rate is computed using convexity properties of the KL term.}
    \label{fig:colombia}
\end{figure}
The distribution of poverty index scores for families whose true poverty index is qualifying them for aid is assumed stationary and given by $\pi$; families whose true poverty index does not qualify them for aid is given by $\rho$, which evolves according to \eqref{eq:dynamics_competitive} and represents a strategic population. The threshold is given by $x(t)=47$ for all time.

The initial condition for the strategic population is set to $\rho_0 =\N(54, 10)$. The stationary distribution of families that should qualify is given by the data distribution in the year 1995, denoted $\rho^{(95)}$, minus the strategic population, which we assume is half of the total population: $\pi = 2\rho^{(95)}-\N(54,10)$. The utility functions are $f_1(z,x)=1-q(z,x)-lz$ and $f_2(z,x)=q(z,x)$, where $q(z,x)=(1+\exp(-a(z-x)))^{-1}$. We use $\alpha=0.1$, $a=2$, $l=0.06$, $\rhot=\rho_0$ and $W=0$. Here, $1-q(z,x)$ is the probability that the classifier assigns a label of "qualified" to a family with attributes $z$ and classifier parameters $x$. Families aim to maximize their probability of such a classification. The $lz$ term models a preference for a lower poverty index score, regardless of the classifier parameters.
Using the notation from Section~\ref{sec:application}, the resulting energy functional is
\begin{align*}
    G(\rho,x) = \int f_1(z,x) \d \rho(z) + \int f_2(z,x) \d \pi(z) - \alpha \int \rho(z) \log \rho(z) \d z\,,
\end{align*}
with $\ddt x(t) = 0$.

The potentials appearing in the functional $G$ are chosen so that the global minimizer is $\rho^{(98)}$. In Figure \ref{fig:colombia} (left), the model distribution plotted is $(\rho + \pi)/2$, with $1e6$ samples drawn to generate the plot. We observe that the model is able to capture the sharp drop on the right side of the qualifying threshold due to the steep classifier $f_1$, as well as the curvature of the distribution close to the threshold. In Figure~\ref{fig:colombia}, we also plot the loss of the classifier and population (top right), and the Wasserstein distance between the data from 1998, denoted $\rho^{(98)}$, and $\rho_t$ (bottom right), and fit exponential functions to estimate the rate of decay. The expected rate of decay is specified by the rate given in Theorem~\ref{thm:PDE_aligned} since the classifier is stationary. Because $f_1$ is such that $\lambda_{f,1}=-\Lambda_1=0$,
the rate is generated by the convexity of $\rhot$, which here we set to $\rho_0$; the expected rate is $0.01$ since $\rho_0 \sim \N(54,10)$.
The convergence rate for the loss of the population and algorithm and for $\Wass_2(\rho_t,\rho^{(98)})$ are close to the theoretical value; see Figure~\ref{fig:colombia} for details.

\subsection{Loan Applications: Feature Modification}\label{subsec:loan_applications}
In settings such as loan applications, agents not eligible for a loan (label 0) aim to be misclassified to receive a more desirable outcome, such as qualifying for a loan (algorithm predicts label 1). In this numerical experiment, we consider real loan application data from \cite{cukierski_give_2011} and allow label-0 agents to manipulate two out of their eleven features. We selected the two features that, pairwise, gave the lowest classification loss, which are ``age" and ``number of times the borrower has been 90 days or more past due." The agents have a penalty for deviating from the initial condition, as enforced by the KL divergence term with $\rhot=\rho_0$, and the potential function is the negative of the probability of being classified as a label-1 agent. While the agents can only manipulate two features, we allow the classifier to update based on all eleven features. The agents with true label 1 do not manipulate any features.
\begin{figure}
    \centering
    \includegraphics[scale=0.3]{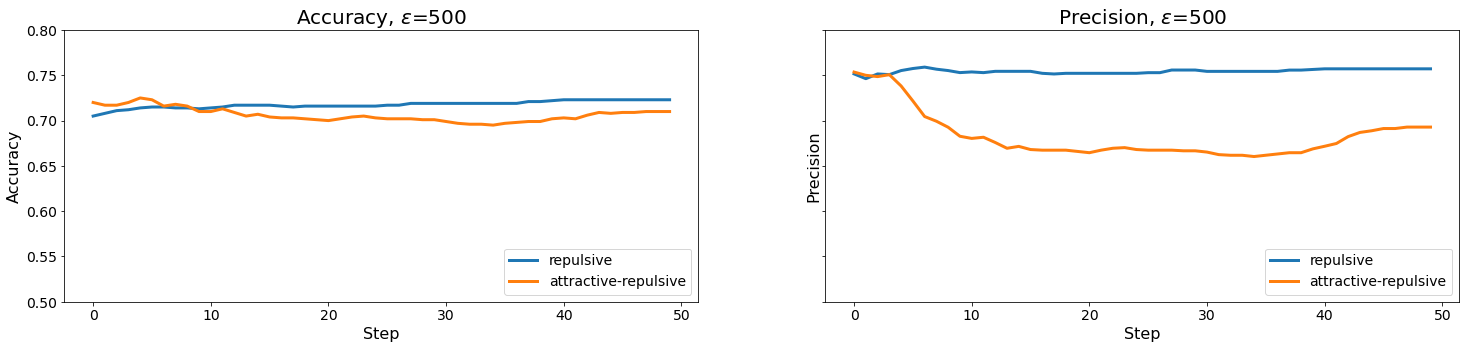}
    \caption{While the accuracy of the classifier is similar under both interaction models, the precision differs; this indicates that understanding the intra-agent interactions is important for understanding how errors impact subpopulations, in this case, those agents with algorithm label ``qualified."}
    \label{fig:kernel_comparison_accuracy_precision}
\end{figure}
\begin{figure}
    \centering
    \begin{subfigure}[b]{0.47\textwidth}
    \includegraphics[trim= 8 5 0 40, clip, width=\textwidth]{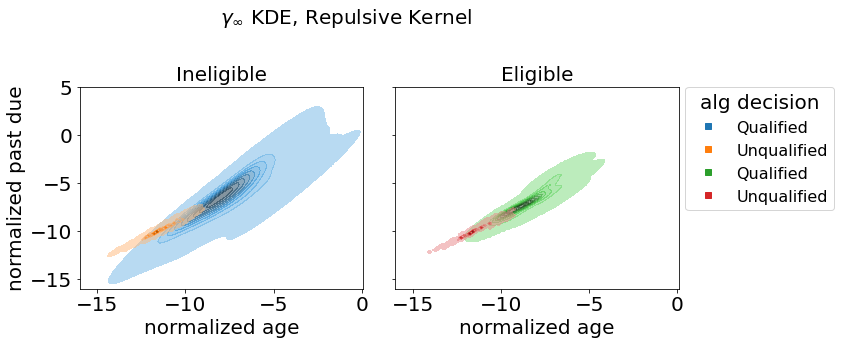}
        \caption{Repulsive Kernel}
    \end{subfigure}
    \begin{subfigure}[b]{0.47\textwidth}
    \includegraphics[trim= 8 5 0 40, clip, width=\textwidth]{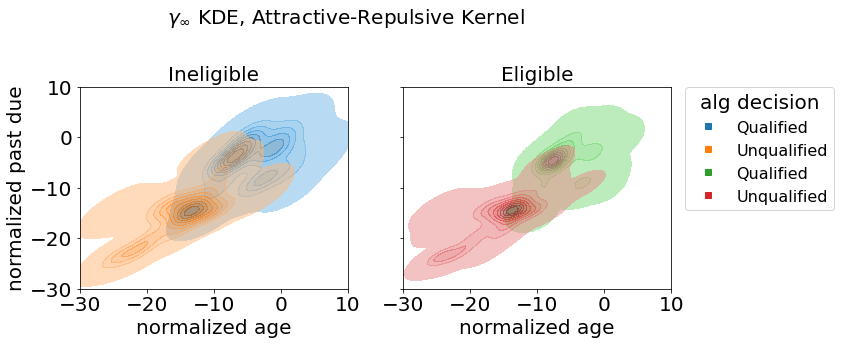}
        \caption{Attractive-repulsive kernel}
    \end{subfigure}
    \caption{Agent densities are split based on their true label (eligible vs ineligible). The color denotes the algorithm output regarding loan qualification (qualified vs unqualified). The repulsive kernel causes the agents to spread apart along attributes 1 (normalized age) and 2 (normalized past due), while the attractive-repulsive kernel causes more swarm-like behavior. This is not evident from classifier performance only, indicating the importance of understanding the population dynamics explicitly.}
    \label{fig:kernel_comparison_density}
\end{figure}
For this application, we consider interactions between agents assuming that people exchange information about their loan applications and application outcomes. We compare the evolution of the population for $N=1000$ agents under two different interaction kernels:
\begin{align*}
    W_r(z) &= -|z| &\quad \text{repulsive}\,, \\
    W_a(z) &= 4e^{-5|z|} - 2e^{-|z|/4} &\quad \text{attractive-repulsive}\,,
\end{align*}
implemented via the underlying particle system given by
\begin{align*}
    \d z_t^{(i)} = -\nabla_z f_1(z_t^{(i)},x)\d t - \sum_{j\ne i} \nabla_z W(|z_t^{(i)}-z_t^{(j)}|) \d t +\alpha\nabla\log \rho_0(z_t^{(i)}) + \sqrt{2 \alpha}\d B_t\,, \quad z_t^{(i)}\in\R^2\,,
\end{align*}
with $W=W_a$ and $W=W_r$ for each setting.
The repulsive interaction kernel encourages agents to move away from their neighbors, while the attractive-repulsive kernel encourages swarm-like behavior by pushing neighboring agents apart while attracting agents that are far from each other. The energy functional which corresponds to this example is
\begin{align*}
    G(\rho,x) &= \int f_1(z,x) \d \rho(z) + \int f_2(z,x) \d \pi(z) - \iint W(z-\tilde z) \d \rho(z) \d \rho(\tilde z) 
    - \alpha KL(\rho\, | \rhot) + \frac{\beta}{2} \norm{x-x_0}^2 \,. 
\end{align*}
We update the agent data and linear classifier using code adapted from \cite{miller2020whynot}, which uses a particle-based gradient descent scheme in which the agents take 500 update steps per 1 update step for the algorithm. In Figure \ref{fig:kernel_comparison_accuracy_precision}, we observe that under the same training scheme, the accuracy for the classifier is about the same for both interaction kernels. However, the precision, which is the number of true positives over the number of all positives, is not the same. We compute the accuracy for a subpopulation of the label-0 population: those agents initially classified as 0 and those initially (mis)classified as 1. Under the repulsive kernel, these agents are classified with 96\% and 20\% accuracy at the steady state; in contrast, the agents are classified with 90\% and 9\% accuracy under the attractive-repulsive kernel. 
Although we do not have convergence guarantees in this setting, this example highlights the importance of modeling intra-species interactions for applications in which non-asymptotic behavior is relevant.

\sloppy This indicates that the initially-mislabeled subpopulation benefits from attractive-repulsive interactions more than repulsive interactions. In Figure \ref{fig:kernel_comparison_density}, we plot the density estimate and observe that the kernels induce different feature distributions, even though the accuracy is similar, indicating that observing only the performance of the classifier fails to indicate important details about specific subpopulations. 

\subsection{Performative Prediction}
A current state-of-the-art model for updating an algorithm in the face of distribution shift is given in \cite{miller_outside_2021}; the algorithm perturbs the $x$ parameter randomly and records the mean of the population distribution after it responds to the perturbation. Then a linear model is fit between the perturbation direction and the distribution mean, and the algorithm is set to the minimizer of the linear model.
\begin{figure}[!ht]
    \centering
    \begin{subfigure}[b]{0.28\textwidth}
        \includegraphics[scale=0.25]{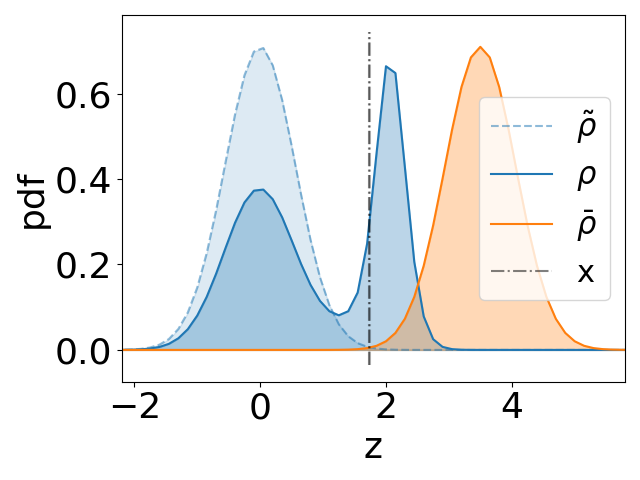}
        \caption{Learning a linear mapping yields a loss of $-1.66$. }
        \label{fig:perf_pred}
    \end{subfigure}
    \hspace{0.8cm}
    \begin{subfigure}[b]{0.28\textwidth}
        \includegraphics[scale=0.25]{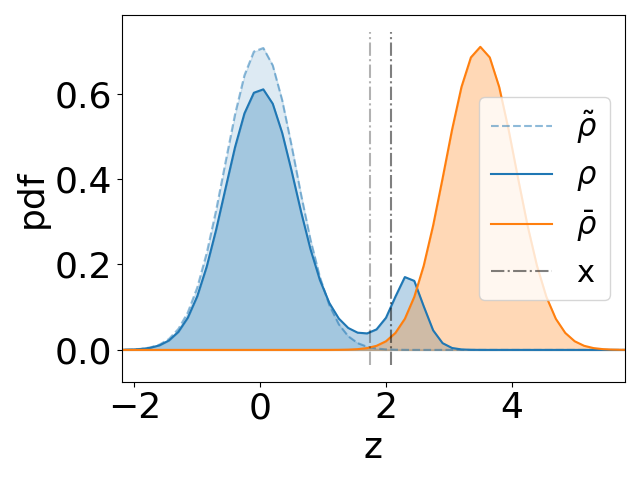}
        \caption{Doing na\"ive gradient descent leads to a loss of $-1.89$.}
        \label{fig:grad_desc}
    \end{subfigure}
    \hspace{0.8cm}
    \begin{subfigure}[b]{0.28\textwidth}
        \includegraphics[scale=0.25]{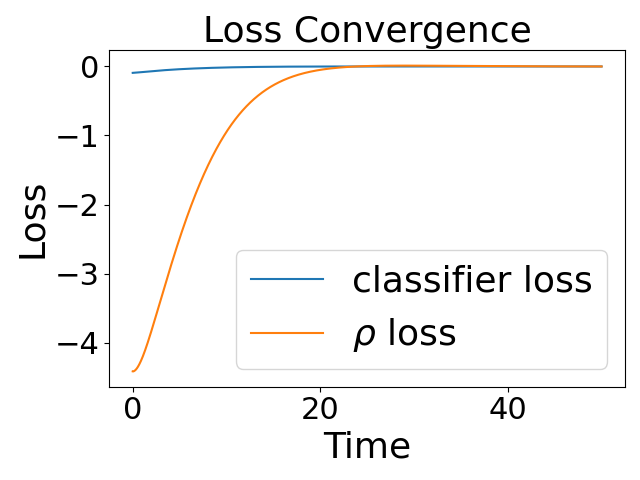}
        \caption{In setting (b), the losses evolve according to expected rates.} \label{fig:perf_pred_rates}
    \end{subfigure}
    \caption{A gradient descent approach outperforms a state-of-the-art technique of learning and using a linear mapping between the classifier parameters and the mean of the strategic distribution, illustrating the importance of having more detailed population models.}
    \label{fig:SOA}
\end{figure}
We simulate the strategic distribution using our PDE model in the competitive setting, comparing classifier performance under two update strategies: (1) the perturbation method described above, and (2) na\"ive gradient descent, specifically  Wasserstein gradient descent in the dynamics \eqref{eq:dynamics_competitive}. The energy used in the example is the same function as in Example~\ref{subsec:loan_applications}, with $W=0$ and rescaled parameters $\alpha$, $\beta$, $a$, $l$, and $\rhot$.
The resulting distribution shifts and algorithm updates are shown in Figure \ref{fig:SOA}.
In Figure \ref{fig:perf_pred}, we observe that the optimization problem from the mean shift model \cite{miller_outside_2021} sets the classifier to $x(t)\approx 0$ for all time, while in Figure \ref{fig:grad_desc}, the classifier moves to the right, resulting in better performance. This difference occurs because the perturbation method detects only a mean shift, not that two modes are appearing, and underestimates the impact of the population mass that splits from the main population. This illustrates that modeling the population with more fine-grained detail than just mean information is critical for selecting algorithm update strategies that are most effective. In Figure~\ref{fig:perf_pred_rates}, the convergence of the loss for the population $\rho$ and classifier $x$ are shown, where we subtracted the value at steady state from the energy functional so that the energy converges to zero. The rates are given by $0.192$ and $0.192$ for the population and classifier (resp.). The theoretical rate from Theorem~\ref{thm:convergence_Wbar_competitive} is $\lambda_c= \min\{0.05,1-1.483 \}=-0.483$, which indicates that the theorem conditions are not satisfied, since $\lambda_c <0$. However, the condition is not tight (for example, when the mass is not concentrated around the worst-case convexity given by $\Lambda_1$), and we observe convergence empirically. The loss for $\rho$ is increasing because $\rho$ is maximizing $G$, and although $x$ is minimizing $G$, the evolution of $\rho$ causes the loss of $x$ to increase slightly.

\section{Cooperative Setting (Proof of Theorem~\ref{thm:PDE_aligned})}\label{sec:proof_thm_aligned}
 Denote
\begin{align*}
    \V(\rho,\mu) &= \iint f(z,x)\d \rho(z) \d \mu(x)+\int V_1(z)\d\rho(z) + \int V_2(x) \d \mu(x)\,,\\
    \W(\rho,\mu) &= \frac{1}{2}\int (W_1 \ast \rho)(z)\,\d \rho(z) + \frac{1}{2}\int (W_2 \ast \mu)(x)\d\mu(x) \,,
\end{align*}
so that the functional $F_a$ is given by 
\begin{align*}
    F_a(\rho,\mu) = \V(\rho,\mu) + \alpha H(\rho) + \beta H(\mu) + \W(\rho,\mu)\,.
\end{align*}

In order to prove the existence of a unique ground state for $F_a$, a natural approach is to consider the corresponding Euler-Lagrange equations
\begin{subequations}\label{eq:EL-Ga}
\begin{align}
    \alpha \log\rho(z) + \int f(z,x)\d\mu(x) + V_1(z) + (W_1 \ast \rho)(z)  &= c_1[\rho,\mu]\quad \text{ for all } z\in\supp(\rho)\,, \label{eq:EL-Ga-rho}\\
    \beta \log \mu(x) + \int f(z,x)\d\rho(z) + V_2(x) + (W_2 \ast \mu)(x) &= c_2[\rho,\mu] \quad \text{ for all } x\in\supp(\mu)\,, \label{eq:EL-Ga-mu}
\end{align}
\end{subequations}
where $c_1,c_2$ are constants that may differ on different connected components of $\supp(\rho)$ and $\supp(\mu)$. These equations are not easy to solve explicitly, and we are therefore using general non-constructive techniques from calculus of variations. We first show continuity and $\lambda_a$-convexity properties for the functional $F_a$ (Lemma~\ref{lem:lsc} and Proposition~\ref{prop:dipl-convexity}), where
$$\lambda_a\coloneqq \lambda_{f}+\min\{\lambda_{V,1},\lambda_{V_2}\}\,;$$
essential properties that will allow us to deduce existence and uniqueness of ground states using the direct method in the calculus of variations (Proposition~\ref{prop:existenceGa}). Using the Euler-Lagrange equation~\eqref{eq:EL-Ga}, we then prove properties on the support of the ground state (Corollary~\ref{cor:ground-states-supp-Ga}).
To obtain convergence results, we apply the HWI method: we first show a general 'interpolation' inequality among the energy, the energy dissipation and the metric (Proposition~\ref{prop:HWI}); this fundamental inequality will then imply a generalized logarithmic Sobolev inequality (Corollary~\ref{cor:logSob}) relating the energy to the energy dissipation, and a generalized Talagrand inequality (Corollary~\ref{cor:Talagrand}) that  translates convergence in energy into convergence in metric. Putting all these ingredients together will then allow us to conclude the statements in Theorem~\ref{thm:PDE_aligned}. Throughout this section, we assume $\lambda_a>0$ and that Assumptions~\ref{assump:f_lower}(i), \ref{assump:V_lower} and \ref{assump:W_lower} hold.

\subsection{Ground States and Steady States}
\begin{lemma}[Lower semi-continuity]\label{lem:lsc}
 The functional $F_a:\P\times\P\to \R$ is lower semi-continuous with respect to the weak topology.
\end{lemma}
\begin{proof}
    We split the energy $F_a$ into three parts: (i) $\alpha \Hc(\rho)+\beta \Hc(\mu)$, (ii) $\W(\rho,\mu)$, and (iii) the joint potential energy $\V(\rho,\mu)$. For (i), Lemma~\ref{lem:lsc_entropy_santambrogio} gives lower-semicontinuity for $\rho\in\P$ with respect to the weak topology. \cite[Proposition 7.2]{Santambrogio} provides lower-semicontinuity for (ii) because $W_i$ is continuous.
    For (iii), note that $f$ is lower semi-continuous and bounded below thanks to Assumptions~\ref{assump:f_lower}(i) and \ref{assump:V_lower}, and so the result follows from \cite[Proposition 7.1]{Santambrogio}.
\end{proof}

\begin{proposition}[Uniform displacement convexity]\label{prop:dipl-convexity}
Fix $\gamma_0, \gamma_1\in\P_2\times\P_2$. Along any geodesic $(\gamma_s)_{s\in[0,1]}\in\P_2\times\P_2$ connecting $\gamma_0$ to $\gamma_1$, we have for all $s\in[0,1]$
    \begin{align}\label{eq:ddssG}
    \ddss F_a(\gamma_s) \geq \lambda_a \Wbar(\gamma_0,\gamma_1)^2 \,.
\end{align}
As a result, the functional $F_a:\P\times\P\to \R$ is uniformly displacement convex with constant $\lambda_a>0$. 
\end{proposition}

\begin{proof}
Let $\gamma_0$ and $\gamma_1$ be two absolutely continuous probability measures with bounded second moments. The general case can be recovered using approximation arguments. Denote by $\phi, \psi:\R^{d_i}\to\R$ the optimal Kantorovich potentials pushing $\rho_0$ onto $\rho_1$, and $\mu_0$ onto $\mu_1$, respectively:
\begin{align*}
    &\rho_1=\nabla \phi_\# \rho_0 \quad \text{ such that } \quad  \Wass_2(\rho_0,\rho_1)^2 = \int_{\R^{d_1}} \|z-\nabla \phi(z)\|^2 \d \rho_0(z)\,,\\
    &\mu_1=\nabla \psi_\# \mu_0 \quad \text{ such that } \quad  \Wass_2(\mu_0,\mu_1)^2 = \int_{\R^{d_2}} \|x-\nabla  \psi(x)\|^2 \d \mu_0(x)\,.
\end{align*}
The now-classical results in \cite{Benamou-Brenier} guarantee that there exist convex functions $\phi, \psi$ that satisfy the conditions above. Then the path $(\gamma_s)_{s\in[0,1]}=(\rho_s,\mu_s)_{s\in[0,1]}$ defined by
\begin{align*}
    \rho_s = [(1-s)\id+ s \nabla \phi]_\#\rho_0 \,,\qquad
    \mu_s  = [(1-s)\id+s\nabla \psi]_\#\mu_0
\end{align*}
is a $\Wbar$-geodesic from $\gamma_0$ to $\gamma_1$.

The first derivative of $\V$ along geodesics in the Wasserstein metric is given by
\begin{align*}
    \dds \V(\gamma_s) 
    = &\dds \left[\iint f((1-s)z +s\nabla\phi(z),(1-s)x +s\nabla\psi(x)) \,\d\rho_0(z)\d\mu_0(x)\right. \\
    &+\left.\int V_1((1-s)z+s\nabla\phi(z))\d \rho_0(x) + \int V_2((1-s)x +s\nabla\psi(x) ) \,\d\mu_0(x)
    \right]\\
    = &\iint \nabla_x f((1-s)z +s\nabla\phi(z),(1-s)x +s\nabla\psi(x))\cdot (\nabla\psi(x)-x) \,\d\rho_0(z)\d\mu_0(x) \\
    &+\iint\nabla_z f((1-s)z +s\nabla\phi(z),(1-s)x +s\nabla\psi(x))\cdot (\nabla\phi(z)-z) \,
    \d\rho_0(z)\d\mu_0(x) \\
    &+ \int \nabla_z V_1((1-s)z+s\nabla\phi(z))\cdot(\nabla \phi(z)-z)\, \d \rho_0(z) \\
    &+  \int \nabla_x V_2((1-s)x +s\nabla\psi(x) )\cdot (\nabla\psi(x)-x) \, \d\mu_0(x)\,,
\end{align*}
and taking another derivative we have
\begin{align*}
    \ddss \V(\gamma_s) &= \iint \begin{bmatrix}(\nabla\psi(x)-x) \\ (\nabla\phi(z)-z) \end{bmatrix}^\top \cdot D_s(z,x) \cdot \begin{bmatrix}(\nabla\psi(x)-x) \\ (\nabla\phi(z)-z) \end{bmatrix} \, \d\rho_0(z) \d\mu_0(x) \\
    &\ \  +\int (\nabla\varphi(z)-z)^\top \cdot \nabla^2_{z}V_1((1-s)z +s\nabla\psi(z) ) \cdot(\nabla\varphi(z)-z)\, \d\rho_0(x) \\
    & \quad +\int (\nabla\psi(x)-x)^\top \cdot \nabla^2_{x}V_2((1-s)x +s\nabla\psi(x) ) \cdot(\nabla\psi(x)-x)\, \d\mu_0(x) \\
       &\geq \lambda_f \Wbar(\gamma_0,\gamma_1)^2 + \lambda_{f,1}\Wass_2(\rho_0,\rho_1)^2 + \lambda_{f,2}\Wass_2(\mu_0,\mu_1)^2 \,,
\end{align*}
where we denoted $D_s(z,x):= \text{Hess}(f)((1-s)z +s\nabla\phi(z),(1-s)x +s\nabla\psi(x))$, and the last inequality follows from Assumptions~\ref{assump:f_lower}(i) and \ref{assump:V_lower} and the optimality of the potentials $\phi$ and $\psi$.

Following \cite{carrillo_kinetic_2003,villani-OTbook-2003} and using Assumption~\ref{assump:W_lower}, the second derivatives of the diffusion and self-interaction terms are given by
\begin{align}\label{eq:classic_convexity}
    &\ddss \Hc(\rho_s) 
    \geq 0\,,\quad
    \ddss \Hc(\mu_s) \ge 0\,, \quad \ddss \W(\gamma_s) \ge 0 \,.
\end{align}
Putting the above estimates together, we obtain \eqref{eq:ddssG}. 
\end{proof}
\begin{remark}
    If the dynamics are such that the center of mass of $\rho_t$ or $\mu_t$ are preserved for all time, then the convexity of $W_i$ contributes to the rate of convergence. This occurs, for example, when $V_1$ and $W_1$ are radially symmetric and the initial condition $\rho_0$ is radially symmetric; then the rate $\lambda_a$ would be $\lambda_a = \lambda_f + \min\{\lambda_{V,1}+\lambda_{W,1},\lambda_{V_2}\}$. For details, see \cite[Theorems 2.2, 2.4, 2.5]{carrillo_kinetic_2003}.
\end{remark}

\begin{lemma}[Lower bound]\label{lem:Fa_non_negative}
 We have $F_a\ge 0$ over $\P(\R^{d_1})\times \P({\R^{d_2}})$.
\end{lemma}
\begin{proof}
    By Assumption~\ref{assump:f_lower}(i), \ref{assump:V_lower} and \ref{assump:W_lower}, $f,V_i,W_i\ge 0$. When $\alpha=\beta=0$, then $F_a\ge 0$.
    If $\alpha>0$ or $\beta>0$, we will write the $\log$ term as a KL divergence to show that $F_a$ is non-negative. Since $\lambda_a>0$, either $f \ne 0$ or $V_1,V_2 \ne 0$. We will show the setting in which $f\ne 0$ and $V_1,V_2=0$ as all other cases follow analogously. If either $\rho$ (with $\alpha>0$) or $\mu$ (with $\beta>0$) is a singular measure, then $F_a=+\infty$ according to the definition of the entropy functional, and so the claim holds trivially true. In all other cases, the functional $F_a$ can be rewritten as
    \begin{align*}
        F_a[\rho,\mu] &= \frac12\int \rho W_1 \ast \rho + \frac12\int \mu W_2 \ast \mu  \\
        &+\alpha \int \left(\int \rho(z) \log \frac{\rho(z)}{\tilde f_1 (z,x)} \d z \right) \d \mu(x) + \beta \int \left( \int \mu(x) \log \frac{\mu(x)}{\tilde f_2(z,x)} \d x \right) \d \rho(z)\,,
    \end{align*}
    where $\tilde f_1(z,x) := \exp\left(-\frac{1}{2\alpha} f(z,x)\right)$ and $\tilde f_2(z,x) := \exp\left(-\frac{1}{2\beta} f(z,x)\right)$. Since $\rho,\mu$ are absolutely continuous with respect to $\tilde f_1$, $\tilde f_2$ respectively, we have $F_a \ge 0$ by Pinsker's inequality and using $W_i\ge 0$.
\end{proof}
\begin{proposition}(Ground state)\label{prop:existenceGa}
 The functional $F_a:\P(\R^{d_1})\times\P(\R^{d_2})\to [0,\infty]$ admits a unique minimizer $\gamma_*=(\rho_*,\mu_*)$ which satisfies $\gamma_*\in\P_2\times\P_2$. Moreover, if $\alpha>0$ it satisfies $\rho_*\in L^1_+(\R^{d_1})$, and if $\beta>0$ it satisfies $\mu_*\in L^1_+(\R^{d_2})$.
\end{proposition}

\begin{proof}
We show existence of a minimizer of $F_a$ using the direct method in the calculus of variations. Denote by $\gamma=(\rho,\mu)\in\P\times \P \subset \M \times \M$ a pair of probability measures as a point in the product space of Radon measures. Since $F_a\ge 0$ on $\P\times \P$ by Lemma~\ref{lem:Fa_non_negative} and not identically $+\infty$ everywhere, there exists a minimizing sequence $(\gamma_n)\in \P\times\P$. Note that $(\gamma_n)$ is in the closed unit ball of the dual space of continuous functions vanishing at infinity $(C_0(\R^{d_1}) \times C_0(\R^{d_2}))^*$ endowed with the dual norm $\norm{\gamma_n}_*=\sup \frac{|\int f\d \rho_n + \int g\d\mu_n|}{\norm{(f,g)}_\infty}$ over $f,g\in C_0(\R^{d_i})$ with $\|(f,g)\|_\infty:=\|f\|_\infty + \|g\|_\infty \neq 0$.
  By the Banach-Alaoglu theorem \cite[Thm 3.15]{rudin_functional_1991}  there exists a limit $\gamma_*= (\rho_*,\mu_*)\in\M\times \M=(C_0\times C_0)^*$ and a convergent subsequence (not relabelled) such that $\gamma_n\weakstar \gamma_*$.
  
  It remains to show that $\int \d\rho_*=\int \d\mu_*=1$ to conclude that $\gamma_*\in\P\times\P$. To this aim, it is sufficient to show tightness of $(\rho_n)$ and $(\mu_n)$, preventing the escape of mass to infinity as we have $\int \d\rho_n=\int \d\mu_n=1$ for all $n\ge 1$. Tightness follows from Markov's inequality \cite{ghosh} if we can establish uniform bounds on 
  the second moments, i.e. we want to show that there exists a constant $C>0$ independent of $n$ such that
  \begin{equation}\label{eq:tightness_second_moment}
      \int \|z\|^2\d\rho_n(z) + \int \|x\|^2\d\mu_n(x) <C\qquad \forall n\in\mathbb{N}\,.
  \end{equation}
To establish \eqref{eq:tightness_second_moment}, observe that thanks to Assumption~\ref{assump:f_lower}, there exist a constant $c_0\in\R$ and vector $c_1\in\R^{d_1+d_2}$ such that $f(z,x) \ge c_0\norm{\begin{bmatrix}
    z,x
\end{bmatrix}+c_1}^2 $ for all $[z,x]\in\R^{d_1 + d_2}$. 
Additionally, from Lemma~\ref{lem:rho_log_rho_lower_bound} applied with $\varepsilon= c_0/\left(2\max\{\alpha,\beta\} (\norm{c_1}+1)\right)$, we have that $\Hc (\rho) \ge -\eps \int \norm{z}^2 \d \rho(z) - c_\eps$ for some $c_\eps\ge 0$, with the analogous bound for $\Hc(\mu)$. Then
\begin{align*}
c_0\iint &\norm{[z,x]+c_1}^2 \d \gamma_n(z,x) \le \iint f(z,x) \d \gamma_n(z,x) \\
&\le F_a(\gamma_n)+ \eps\left(\alpha \int \norm{z}^2 \d \rho_n(z) + \beta \int \norm{x}^2 \d \mu_n(x) \right) + (\alpha + \beta)c_\eps\,.
\end{align*}
Since $\frac{\norm{[z,x]}^2}{\norm{c_1}+1}-\norm{c_1} \le \norm{[z,x]+c_1}^2 $, the estimate can be rearranged to
\begin{align*}
&\frac{c_0}{2(\norm{c_1}+1)} \left(  \int \norm{z}^2 \d\rho_n(z) + \int \norm{x}^2\d \mu_n(x)\right) 
  \le F_a(\gamma_n) + \hat c_\eps 
 \le F_a(\gamma_1) + \hat c_\eps < \infty\,,
\end{align*}
where $\hat c_\eps \coloneqq c_0 \norm{c_1}+(\alpha+\beta)c_\eps$. Hence, the second moment is uniformly bounded.
This concludes the proof that the limit $\gamma_*$ satisfies  $\gamma_*\in\P\times\P$, and indeed $\rho_*\in\P_2(\R^{d_1}),\mu_*\in\P_2(\R^{d_2})$ as well. 
Further, note that the above second moment bound implies that $(\gamma_n)$ also converges weakly according to Definition~\ref{def:weakcv}. Finally, $\gamma_*$ is a minimizer of $F_a$ thanks to weak lower-semicontinuity of $F_a$ following Lemma~\ref{lem:lsc}.

If $\alpha>0$, the ground state $\rho_*$ satisfies
  \begin{align*}
      \rho_*(z) = c_3 \exp\left(-\frac{1}{\alpha}\left(\int f(z,x)\d\mu_*(x) + (W_1\ast \rho_*)(z)+ V_1(z)\right) \right)\quad \text{on }\supp \rho_*\,.
  \end{align*}
  and so $\rho_*\in L_+^1(\R^{d_1})$.
  Similarly, if $\beta>0$, the ground state $\mu_*$ satisfies
  \begin{align*}
      \mu_*(x) = c_4 \exp\left(-\frac{1}{\beta}\left(\int f(z,x)\d\rho_*(z) + (W_2\ast \mu_*)(x)+ V_2(x)\right) \right)\quad \text{on }\supp \mu_*\,,
  \end{align*}
  and we have that $\mu_*\in L_+^1(\R^{d_2})$.

Next we show uniqueness using a contradiction argument. 
Suppose $\gamma_*=(\rho_*,\mu_*)$ and $\gamma_*'=(\rho_*',\mu_*')$ are minimizers of $F_a$. For $s\in[0,1]$, define 
$\gamma_s := ((1-s)\id+sT,(1-s)\id+sS)_\#\gamma_*$, where $T,S:\R^{d_i}\to\R^{d_i}$ are the optimal transport maps such that $\rho_*' = T_\#\rho_*$ and $\mu_*' = S_\#\mu_*$.  By Proposition~\ref{prop:dipl-convexity} the energy $F_a$ is uniformly displacement convex, and so we have 
\begin{align*}
    F_a(\gamma_s)\leq (1-s)F_a(\gamma_*) + s F_a(\gamma_*') = F_a(\gamma_*).
\end{align*}
If $\gamma_*\neq \gamma_*'$ and $s\in (0,1)$, then strict inequality holds by applying similar arguments as in \cite[Proposition 1.2]{mccann_convexity_1997}.
 However, the strict inequality $ F_a(\gamma_s) < F_a(\gamma_*)$ for $\gamma_*\neq \gamma_*'$ is a contradiction to the minimality of $\gamma_*$. Hence, the minimizer is unique.
\end{proof}

\begin{corollary}\label{cor:ground-states-supp-Ga}
 Any minimizer $\gamma_*=(\rho_*,\mu_*)$ of $F_a$ is a steady state for equation~\eqref{eq:dynamics_aligned} according to Definition~\ref{def:sstate_Ga}. If $\alpha>0$, then $\supp(\rho_*)=\R^{d_1}$ and $\rho_*\in C^2(\R^{d_1})$. If $\beta>0$, then $\supp(\mu_*)=\R^{d_2}$ and $\mu_* \in C^2(\R^{d_2})$.
\end{corollary}

  \begin{proof}
  By Proposition~\ref{prop:existenceGa}, we have $\rho_*, \mu_*\in \P_2$. As $\gamma^*$ is a minimizer, it is in particular a critical point, and therefore satisfies equations~\eqref{eq:EL-Ga}. In order to show that $\gamma_*$ is a steady state for equation~\eqref{eq:dynamics_aligned}, we need to show that $\nabla W_1 \ast \rho_* \in L_{loc}^1$ and $\nabla W_2 \ast \mu_* \in L_{loc}^1$, and that if $\alpha>0$, then $\rho_*\in W^{1,2}_{loc}\cap L^1_+\cap L^\infty_{loc}$ and if $\beta>0$, then $\mu_* \in W^{1,2}_{loc} \cap L^1_+\cap L^\infty_{loc}$.
  
  We claim that $\int f(z,\tilde x)\d\mu_*(\tilde x)<\infty$,  $\int f(\tilde z,x)\d\rho_*(\tilde z)<\infty$, $W_1\ast\rho_*(z) <\infty$ and $W_2\ast\mu_*(x) <\infty$ for any fixed $(z,x)\in \R^{d_1}\times \R^{d_2}$. Indeed, each term in the energy without the interaction and coupling potentials can be lower-bounded by a constant separately using positivity of the potentials and Lemma~\ref{lem:rho_log_rho_lower_bound} together with the second moment bound for $\gamma_*$. Hence, $\iint f(z,x) \d \mu_*(x) \d \rho_*(z) < \infty $,  $\int (W_1\ast\rho_*)(z) \d \rho_*(z)<\infty$ and $\int (W_2\ast\mu_*)(x) \d \mu_*(x)<\infty$ as $\gamma_*$ is a minimizer. This concludes the bounds, and so we also obtain $\nabla W_1 \ast \rho_* \in L_{loc}^1$ and $\nabla W_2 \ast \mu_* \in L_{loc}^1$ since $W_1,W_2\in C^2$.

If $\alpha=\beta=0$, we can differentiate \eqref{eq:EL-Ga} directly to obtain \eqref{eq:sstates}.
Now, consider the case $\alpha>0, \beta>0$. Rearranging \eqref{eq:EL-Ga}, we obtain (for possibly different constants $c_1[\rho_*, \mu_*],c_2[\rho_*,\mu_*] \neq 0$) that
\begin{equation}
    \begin{aligned}\label{eq:EL-Ga-rewritten}
        \rho_*(z) &= c_1[\rho_*,\mu_*] \exp{\left[-\frac{1}{\alpha}\left( \int f(z,x)\d\mu_*(x) + W_1\ast\rho_*(z) + V_1(z) \right)\right]}\quad \text{ on } \supp(\rho_*)\,,\\
        \mu_*(x) &= c_2[\rho_*,\mu_*] \exp{\left[-\frac{1}{\beta}\left( \int f(z,x)\d\rho_*(z) + W_2\ast\mu_*(x) + V_2(x) \right)\right]}\quad \text{ on } \supp(\mu_*)\,,
    \end{aligned}
\end{equation}
and so $\rho_*,\mu_*\in L^1_+$.
Then for any compact set $K\subset \R^{d_1}$,
    \begin{align*}
        \sup_{z\in K} \rho_*(z) \leq c_1[\rho_*,\mu_*]  \sup_{z \in K} \exp\left(-\frac{1}{\alpha}\left( \int f(z,x)\d\mu_*(x) + V_1(z) \right)\right) \sup_{z \in K} \exp \left( -\frac{1}{\alpha} W_1 \ast \rho_* \right).
    \end{align*}
    As $f,V_1, W_1\ge 0$, the exponential terms on the right-hand side are finite. Therefore $\rho_* \in L_{loc}^\infty$.
    To show that $\rho_*\in W_{loc}^{1,2}$, note that for any compact set $K\subset \R^{d_1}$, we have $\int_K |\rho_*(z)|^2 \d z < \infty $ as a consequence of $\rho_* \in L_{loc}^\infty$.
Moreover, defining $T[\gamma](z)\coloneqq  -\frac{1}{\alpha}\left( \int f(z,x)\d\mu(x) + W_1\ast\rho(z)+V_1(z)\right)\le 0$, we have 
    \begin{align*}
        \int_K |\nabla \rho_*|^2 \d z = c_1[\rho_*,\mu_*]^2 \int_K |\nabla T[\gamma_*]|^2\exp (2T[\gamma_*]) \d z\,,
    \end{align*}
    \sloppy which is bounded noting that $\exp (2T[\gamma_*]) \le 1$ and that $\nabla T[\gamma_*](\cdot)$ is in $L^\infty_{loc}$, since $f(\cdot,x), W_1(\cdot), V_1(\cdot)\in C^1(\R^{d_1})$ by Assumptions~\ref{assump:f_lower}-\ref{assump:W_lower}. We conclude that $\rho_*\in W^{1,2}_{loc}$, and use an identical argument for the case when $\beta>0$, and indeed $(\rho_*,\mu_*)$ solves \eqref{def:sstate_Ga} in the sense of distributions as a consequence of \eqref{eq:EL-Ga}.

Next, we show that if $\alpha>0$ then $\supp(\rho_*)=\R^{d_1}$ using again the relation~\eqref{eq:EL-Ga-rewritten}.
Indeed, $\exp{\left[-\frac{1}{\alpha}\left( \int f(z,x)\d\mu_*(x) + W_1\ast\rho_*(z)+V_1(z)\right)\right]}>0$ for all $z\in \R^{d_1}$ since $\int f(z, x)\d\mu_*( x)<\infty$,  $W_1\ast\rho_*(z) <\infty$ and $V_1(z)<\infty$.  
 Then, thanks to continuity of $f\in C^2$, $V_1\in C^2$, and $W_1\ast\rho_*(z) \in C^2$, we conclude 
$\rho_* \in C^2(\R^{d_1})$. The same argument is applied to $\mu_*$ when $\beta>0$ to obtain $\supp \mu_*=\R^{d_2}$ and $\mu_*\in C^2(\R^{d_2})$.
  \end{proof}

\subsection{Functional Inequalities}

The following inequality is referred to as HWI inequality and represents the key result to obtain convergence to equilibrium.

\begin{proposition}[HWI inequality]\label{prop:HWI}
Define the dissipation functional
\begin{align*}
    D_a(\gamma):= \iint \norm{\begin{bmatrix}
        \nabla_z \delta_\rho F_a[\rho,\mu](z) \\ \nabla_x \delta_\mu F_a[\rho,\mu](x)
    \end{bmatrix}}^2\d\gamma(z,x)\,.
\end{align*}
Let $\gamma_0,\gamma_1\in\P_2\times\P_2$ such that ${F_a(\gamma_0), F_a(\gamma_1)\allowbreak, D_a(\gamma_0)<\infty}$, and let Assumptions~\ref{assump:f_lower}(i), \ref{assump:V_lower} and \ref{assump:W_lower} hold with $\lambda_a > 0$. Then
    \begin{align}\label{eq:HWI_ineq}
        F_a(\gamma_0)-F_a(\gamma_1)\le \Wbar(\gamma_0,\gamma_1)\sqrt{D_a(\gamma_0)} - \frac{\lambda_a}{2} \,\Wbar(\gamma_0,\gamma_1)^2\,.
    \end{align}
\end{proposition}
\begin{proof}
    For simplicity, consider $\gamma_0,\gamma_1$ that have smooth Lebesgue densities of compact support. The general case can be recovered using approximation arguments. Let $(\gamma_s)_{s\in[0,1]}$ denote a $\Wbar$-geodesic between $\gamma_0,\gamma_1$. Following similar arguments as in \cite{carrillo_kinetic_2003} and \cite[Section 5]{otto_generalization_2000} and making use of the calculations in the proof of Proposition~\ref{prop:dipl-convexity}, we have
    \begin{align*}
    \left.\dds F_a(\gamma_s)\right|_{s=0} 
    &= \iint 
     \begin{bmatrix} \xi_1(z) \\ \xi_2(x)\end{bmatrix} \cdot
        \begin{bmatrix} (\nabla\phi(z)-z) \\ (\nabla\psi(x)-x)\end{bmatrix}\,\d\gamma_0(z,x)\,,
    \end{align*}
    where
    \begin{align*}
        \xi_1[\gamma_0](z)&:=  \int\nabla_z f(z,x)\d\mu_0(x) + \alpha \nabla_z\log\rho_0(z) + \nabla W_1 \ast \rho_0(z) + \nabla_z V_1(z)=\nabla_z \delta_\rho F_a[\gamma_0](z)\,,\\
        \xi_2[\gamma_0](x)&:= \int\nabla_x f(z,x)\d\rho_0(z) + \beta \nabla_x\log\mu_0(x) + \nabla W_2 \ast \mu_0(x) + \nabla_x V_2(x) =\nabla_x \delta_\mu F_a[\gamma_0](x)\,.
    \end{align*}
    Note that the dissipation functional can then be written as 
    $$
    D_a(\gamma_0)=\iint \left(\norm{\xi_1(z)}^2 + \norm{\xi_2(x)}^2\right)\d\gamma_0(z,x)\,.
    $$
    Using the double integral Cauchy-Schwarz inequality \cite{steele_cauchy-schwarz_2004}, we obtain
\begin{align*}
    \left.\dds F_a(\gamma_s)\right|_{s=0} 
    &\ge
 -\left(\sqrt{ \iint \norm{\begin{bmatrix}
        \xi_1 \\ \xi_2
    \end{bmatrix}}^2 \d \gamma_0}\right)\left(\sqrt{ \iint \norm{\begin{bmatrix}
        \nabla \phi(z)-z \\ \nabla\psi(x)-x 
    \end{bmatrix}}^2\d \gamma_0} \right)\\
    &= -\sqrt{D_a(\gamma_0)}\,\sqrt{\int \|\nabla\phi(z)-z\|^2 \d \rho_0 + \int \|\nabla\psi(x)-x\|^2 \d \mu_0  } \\
    &= -\sqrt{D_a(\gamma_0)} \,\Wbar(\gamma_0,\gamma_1)\,.
\end{align*}
Next, we compute a Taylor expansion of $F_a(\gamma_s)$ when considered as a function in $s$ and use the bound on $\ddss F_a$ from \eqref{eq:ddssG}:
\begin{align*}
    F_a(\gamma_1) &= F_a(\gamma_0) + \left.\dds F_a(\gamma_s)\right|_{s=0} + 
    \int_0^1 (1-t) \left.\left(\ddss F_a(\gamma_s)\right)\right|_{s=t}\,\d t\\
    &\geq F_a(\gamma_0) - \sqrt{D_a(\gamma_0)}\, \Wbar(\gamma_0,\gamma_1) + \frac{\lambda_a}{2} \Wbar(\gamma_0,\gamma_1)^2\,.
\end{align*}
\end{proof}

\begin{remark}
    The HWI inequality in Proposition~\ref{prop:HWI} immediately implies uniqueness of minimizers for $F_a$ in the set $\left\{\gamma\in\P\times\P\,:\, D_a(\gamma)<+\infty\right\}$. Indeed, if $\gamma_0$ is such that $D_a(\gamma_0)=0$, then for any other minimizer $\gamma_1$ in the above set we have $F_a(\gamma_0)\le F_a(\gamma_1)$ with equality if and only if $\Wbar(\gamma_0,\gamma_1)=0$.
\end{remark}

\begin{corollary}[Generalized Log-Sobolev inequality]\label{cor:logSob}
    Denote by $\gamma_*$ the unique minimizer of $F_a$. With Assumptions~\ref{assump:f_lower}(i), \ref{assump:V_lower} and \ref{assump:W_lower}, any $\gamma\in\P_2\times\P_2$ such that $F_a(\gamma), D_a(\gamma)<\infty$ satisfies
    \begin{equation}\label{eq:logSob}
        D_a(\gamma)\ge 2\lambda_a \,F_a(\gamma\,|\,\gamma_*)\,.
    \end{equation}
\end{corollary}
\begin{proof}
    This statement follows immediately from Proposition~\ref{prop:HWI}. Indeed, let $\gamma_1=\gamma_*$ and $\gamma_0=\gamma$ in \eqref{eq:HWI_ineq}. Then
    \begin{align*}
        F_a(\gamma\,|\,\gamma_*)&\le \Wbar(\gamma,\gamma_*)\sqrt{D_a(\gamma)} - \frac{\lambda_a}{2} \,\Wbar(\gamma,\gamma_*)^2\\
        &\le \max_{t\ge 0} \left(\sqrt{D_a(\gamma)} t - \frac{\lambda_a}{2} \,t^2\right) = \frac{D_a(\gamma)}{2\lambda_a}\,.
    \end{align*}
\end{proof}

\begin{corollary}[Talagrand inequality]\label{cor:Talagrand} 
Denote by $\gamma_*$ the unique minimizer of $F_a$. With Assumptions~\ref{assump:f_lower}(i), \ref{assump:V_lower} and \ref{assump:W_lower} and $\lambda_a > 0$, it holds
        $$\Wbar(\gamma,\gamma_*)^2 \leq \frac{2}{\lambda_a} F_a(\gamma\,|\,\gamma_*)$$
for any $\gamma\in\P_2\times\P_2$ such that $F_a(\gamma)<\infty$.
\end{corollary}

\begin{proof}
    This is also a direct consequence of Proposition~\ref{prop:HWI} by setting $\gamma_0=\gamma_*$ and $\gamma_1=\gamma$. Then $F_a(\gamma_*)<\infty$ and $D_a(\gamma_*)=0$, and the result follows.
\end{proof}

\begin{proof}[Proof of Theorem~\ref{thm:PDE_aligned}]
The entropy terms $\Hc(\rho)$ and $\Hc(\mu)$ produce diffusion in $\rho$ and $\mu$ for the corresponding PDEs in \eqref{eq:dynamics_aligned}. As a consequence, if $\alpha>0$ (resp. $\beta>0$) a solution $\rho_t$ (resp. $\mu_t$) to \eqref{eq:dynamics_aligned} and minimizer $\rho^*$ (resp. $\mu^*$) for $F_a$ has to be an $L^1$ function.
Result (a) corresponds to the statements in Proposition~\ref{prop:existenceGa} and Corollary~\ref{cor:ground-states-supp-Ga}. To obtain (b), we differentiate the energy $F_a$ along solutions $\gamma_t$ to the equation \eqref{eq:dynamics_aligned}:
\begin{align*}
    \frac{\d}{\d t} F_a(\gamma_t)
    &= \int \delta_\rho F_a [\gamma_t](z)\partial_t\rho_t\d z
    + \int \delta_\mu F_a [\gamma_t](x)\partial_t\mu_t\d x\\
    &= -\int \left\|\nabla_z\delta_\rho F_a [\gamma_t](z)\right\|^2\d\rho_t(z)
    -\int \left\|\nabla_x\delta_\mu F_a [\gamma_t](x)\right\|^2\d\mu_t(x)\\
    &= -D_a(\gamma_t) 
    \le -2\lambda_a F_a(\gamma_t\,|\,\gamma_*)\,,
\end{align*}
where the last bound follows from Corollary~\ref{cor:logSob}. Applying Gr\"onwall's inequality, we immediately obtain decay in energy,
\begin{equation*}
    F_a(\gamma_t\,|\,\gamma_*)\le e^{-2\lambda_a t} F_a(\gamma_0\,|\,\gamma_*)\,.
\end{equation*}
Finally, applying Talagrand's inequality (Corollary~\ref{cor:Talagrand}), the decay in energy implies decay in the product Wasserstein metric,
\begin{equation*}
    \Wbar(\gamma_t,\gamma_*) \le c e^{-\lambda_a t}\,,
\end{equation*}
where $c>0$ is a constant only depending on $\gamma_0$, $\gamma_*$ and the parameter $\lambda_a$. It follows immediately that $\gamma_*$ is a steady state; see Corollary~\ref{cor:ground-states-supp-Ga}.
\end{proof}

\section{Competitive Setting (Proof of Theorem~\ref{thm:convergence_Wbar_competitive})}\label{sec:competitive_proof}
In the competitive setting, the dynamics are parameterized by the energy functional
\begin{align*}
    F_c(\rho,\mu) = \iint f(z,x) \d \rho(z) \d \mu(x) - \cR(\rho)+\U(\mu)
\end{align*}
where $\cR(\rho) =  \alpha \Hc(\rho) + \frac{1}{2} \int (W_1 \ast \rho)(z)\,\d\rho(z) + \int V_1(z) \d \rho(z) $ and $
    \U(\mu) = \beta \Hc(\mu) + \frac{1}{2} \int (W_2 \ast \mu)(x)\,\d\mu(x) + \int V_2(x) \d \mu(x)$. For all results in this section, let Assumptions~\ref{assump:f_lower}(ii), \ref{assump:V_lower}, and \ref{assump:W_lower} hold. We define two convexity coefficients, 
    $$\lambda_{c,1}\coloneqq \lambda_{f,1}+\lambda_{V,1}\,,\qquad \lambda_{c,2}\coloneqq \lambda_{f,2}+\lambda_{V,2}$$ 
    where $\lambda_{c,1}$ is the displacement concavity coefficient of $F_c$ with respect to $\rho$, and $\lambda_{c,2}$ is the displacement convexity coefficient of $F_c$ with respect to $\mu$. We assume that $\lambda_{c,1},\lambda_{c,2}>0$ in order to obtain convergence. The rate of convergence depends on the species with weaker convexity; the convergence rate is given by
    \begin{align*}
        \lambda_c = \min\{\lambda_{c,1},\lambda_{c,2}\}>0\,.
    \end{align*}
\begin{lemma}[Concavity-Convexity of $F_c$]\label{lem:Gc_convex_concave} 
    The functional $F_c$ is uniformly displacement $\lambda_{c,2}$-convex in $\mu$ for any fixed $\rho\in\P_2$ and uniformly displacement $\lambda_{c,1}$-concave in $\rho$ for any fixed $\mu\in\P_2$. 
\end{lemma}
\begin{proof}
    For a fixed $\mu$, the $\lambda_{c,1}$ concavity of the functional $F_c(\rho,\mu)$ can be computed as in Proposition \ref{prop:dipl-convexity}.
    For a fixed $\rho$, the $\lambda_{c,2}$ convexity of the functional $F_c(\rho,\mu)$ can be computed as in Proposition~\ref{prop:dipl-convexity}.
\end{proof}

To show contraction, we  apply \cite[Theorem 23.9]{Villani07}, which provides an expression for the time derivative of $\Wbar(\gamma_t,\gamma_t')^2$. This theorem requires that the velocities of the trajectories are in $L^2$, which we show in the following lemma using the dissipation functional $D_c: \tilde \P \times \tilde \P \to \R \cup \{+\infty\}$,
\begin{align*}
    D_c(\gamma):= \iint \norm{\begin{bmatrix}
        \nabla_z \delta_\rho F_c[\gamma](z) \\ \nabla_x \delta_\mu F_c[\gamma](x)
    \end{bmatrix}}^2\d\gamma(z,x)\,.
\end{align*}
\begin{lemma}\label{lem:velocities_in_L2}
    Let $\gamma_t$ be a solution of the dynamics \eqref{eq:dynamics_competitive}, with initial condition $\gamma_0\in\P_2(\R^{d_1})\times\P_2(\R^{d_2})$ such that $D_c(\gamma_0)<\infty$. 
    Then
    \begin{align*}
        D_c(\gamma_t)\le e^{-2\lambda_c t} D_c(\gamma_0)\quad \forall\ t\ge 0\,.
    \end{align*}
\end{lemma}
\begin{proof}
Denote
\begin{align*}
    h(t)&:=\int\norm{\nabla_z \delta_\rho F_c[\rho_t,\mu_t](z)}^2 \d \rho_t(z) + \int \norm{\nabla_x \delta_\mu F_c[\rho_t,\mu_t](x) }^2 \d \mu_t(x) = D_c(\gamma_t) \,.
\end{align*}
By direct differentiation, we can write $h(t)$ as the difference of two dissipations,
\begin{align*}
    h(t)&= \left[ \frac{\d }{\d t} F_c[\rho_t,\mu_\tau] - \frac{\d}{\d\tau} F_c[\rho_\tau,\mu_t] \right]\bigg|_{\tau = t}\,.
\end{align*}
We  use that the coupling term cancels when differentiating $h(t)$ so that
\begin{align}\label{eq:coupling-cancels}
    \frac{\d}{\d t} h(t) = \left[ \frac{\d^2}{\d t^2} F_c[\rho_t,\mu_\tau]\right]\bigg|_{\tau=t} - \left[\frac{\d^2}{\d\tau^2} F_c[\rho_t,\mu_\tau]\right] \bigg|_{\tau=t} \,.
\end{align}
To show \eqref{eq:coupling-cancels}, recall from the definition of the energy that
    $$-\nabla_z \delta_\rho F_c[\rho_t,\mu_t](z) = \nabla_z \left(-\int f(z,x)\mu_t(x)\d x + \alpha\log\rho_t(z) + V_1(z) + W_1 \ast \rho_t(z)\right).$$
    Computing the time derivative of the weighted $L^2$ norm of this velocity gives
    \begin{align*}
        &\ddt \int \norm{\nabla_z\delta_\rho F_c[\rho_t,\mu_t](z)}^2 \d \rho_t(z) =\int \norm{\nabla_z \delta_\rho F_c[\rho_t,\mu_t](z)}^2 \partial_t \rho_t \d z 
        \\
        &\qquad -2\iint \< \nabla_z \delta_\rho F_c[\rho_t,\mu_t](z),\nabla W_1(z-z')\partial_t \rho_t(z') + \alpha \nabla(\partial_t \rho_t/\rho_t) > \d \rho_t(z) \d z' \\
        &\qquad + 2\iint \< \nabla_z \delta_\rho F_c[\rho_t,\mu_t](z),\nabla_z f(z,x) \partial_t \mu_t(x)> \d \rho_t(z) \d x \,. 
    \end{align*}
    For the diffusion term,
    \begin{align*}
        -2\alpha\int \< \nabla_z \delta_\rho F_c[\rho_t,\mu_t](z),\alpha \nabla (\partial_t \rho_t / \rho_t)>\d \rho_t(z) &= 2\alpha\int \div{\rho_t \nabla_z \delta_\rho F_c[\rho_t,\mu_t](z)} \partial_t \rho(z)  \\
        &= -2 \alpha \int \left| \div{\rho_t \nabla_z \delta_\rho F_c[\rho_t,\mu_t](z)} \right|^2 \d z\,.
    \end{align*}
    Considering each of the three remaining terms individually, we start with
    \begin{align*}
        &\int \norm{\nabla_z \delta_\rho F_c[\rho_t,\mu_t](z)}^2 \partial_t \rho_t \d z  = \int \nabla_z \norm{\nabla_z \delta_\rho F_c[\rho_t,\mu_t](z)}^2 \nabla_z \delta_\rho F_c[\rho_t,\mu_t](z) \d \rho_t(z) \\
        &= 2\int\<\nabla_z \delta_\rho F_c[\rho_t,\mu_t](z),\nabla^2_z \delta_\rho F_c[\rho_t,\mu_t](z) \cdot \nabla_z \delta_\rho F_c[\rho_t,\mu_t](z)> \d \rho_t(z)\,.
    \end{align*}
    For the second term, we obtain
    \begin{align*}
        &-2\iint \< \nabla_z \delta_\rho F_c[\rho_t,\mu_t](z),\nabla W_1(z-z')  \partial_t \rho_t(z') > \d \rho_t(z) \d z' \\
        &= -2\iint \big \langle \nabla_z \delta_\rho F_c[\rho_t,\mu_t](z),\nabla^2_{zz'} W_1(z-z') \cdot \nabla_z \delta_\rho F_c[\rho_t,\mu_t](z') \big\rangle \d \rho_t(z) \d \rho_t(z') \, .
    \end{align*}
    The third term is 
    \begin{align*}
        &\iint \< \nabla_z \delta_\rho F_c[\rho_t,\mu_t](z),\nabla_z f(z,x) \partial_t \mu_t(x)> \d \rho_t(z) \d x \\
        &\quad = -\iint \<\nabla_z \delta_\rho F_c[\rho_t,\mu_t](z),\nabla_{zx}^2 f(z,x) \cdot \nabla_x \delta_\mu F_c[\rho_t,\mu_t](x)> \d \rho_t(z) \d \mu_t(x) \,.
    \end{align*}
    Likewise, we compute the time derivative of the term $\int \norm{\nabla_x \delta_\mu F_c[\rho_t,\mu_t](x)}^2 \d \mu_t(x)$ which is nearly identical to that of the $\rho_t$ velocity term, where the diffusion term again can be bounded above by zero. Further, note that the expression for the coupling term is exactly the same as for $\rho_t$, just with the opposite sign. Due to the zero-sum structure, when both velocities are summed, this term cancels. Thus, we obtain
    \begin{align}
        &\ddt h(t)
        = 2 \int\<\nabla_z \delta_\rho F_c[\rho_t,\mu_t](z),\nabla^2_z \delta_\rho F_c[\rho_t,\mu_t](z) \cdot \nabla_z \delta_\rho F_c[\rho_t,\mu_t](z)> \d \rho_t(z) \notag\\
        &\quad -2 \int\< \nabla_x \delta_\mu F_c[\rho_t,\mu_t](x), \nabla_x^2 \delta_\mu F_c[\rho_t,\mu_t](x) \cdot  \nabla_x \delta_\mu F_c[\rho_t,\mu_t](x)> \d \mu_t(x) \notag\\
        &\quad -2\iint \< \nabla_z \delta_\rho F_c[\rho_t,\mu_t](z),\nabla^2_{zz'} W_1(z-z') \cdot \nabla_z \delta_\rho F_c[\rho_t,\mu_t](z') > \d \rho_t(z) \d \rho_t(z') \notag\\
        & \quad -2\iint \big\langle \nabla_x \delta_\mu F_c[\rho_t,\mu_t](x),\nabla^2_{xx'} W_2(x-x')  \cdot \nabla_x \delta_\mu F_c[\rho_t,\mu_t](x') \big\rangle \d \mu_t(x) \d \mu_t(x') \notag \\
        & \quad -2 \alpha \int \left| \div{\rho_t \nabla_z \delta_\rho F_c[\rho_t,\mu_t](z)} \right|^2 \d z -2 \beta \int \left| \div{\mu_t \nabla_x \delta_\mu F_c[\rho_t,\mu_t](x)} \right|^2 \d x\,. \label{eq:h-dot}
    \end{align}
    This expression is equivalent to 
        \begin{align*}%
 \left[\ddt \int \norm{\nabla_z\delta_\rho F_c[\rho_t,\mu_\tau](z)}^2 \d \rho_t(z) + \frac{\d}{\d\tau}\int \norm{\nabla_x\delta_\mu F_c[\rho_t,\mu_\tau](x)}^2 \d \mu_\tau(x) \right]\bigg|_{\tau=t}\,,
    \end{align*}
    which proves \eqref{eq:coupling-cancels}.
    
    By Lemma~\ref{lem:Gc_convex_concave}, we have $F_c[\rho,\mu]$ with fixed $\mu$ is $\lambda_{c,1}$-concave in $\rho$ and $F_c[\rho,\mu]$ with fixed $\rho$ is $\lambda_{c,2}$-displacement convex in $\mu$. Therefore, we have for all $t,\tau\ge 0$,
    \begin{align*}
        \frac{\d^2}{\d t^2 } F_c[\rho_t,\mu_\tau] \le -2\lambda_{c,1} \ddt F_c[\rho_t,\mu_\tau] \,, \qquad \frac{\d^2}{\d \tau^2} F_c[\rho_t,\mu_\tau] \ge -2\lambda_{c,2} \frac{\d }{\d \tau} F_c[\rho_t,\mu_\tau]\,.
    \end{align*}
    This allows us to use a Bakry-Emry type approach for deriving a decay estimate for $h(t)$. In particular, using \eqref{eq:coupling-cancels}, the fact that $\ddt F_c[\rho_t,\mu_\tau]\ge 0$, $\frac{\d }{\d \tau} F_c[\rho_t,\mu_\tau]\le 0$ and $\lambda_c:=\min\{\lambda_{c,1},\lambda_{c,2}\}$ we have
    \begin{align*}
    \frac{\d}{\d t} h(t) &= \left[ \frac{\d^2}{\d t^2} F_c[\rho_t,\mu_\tau]\right]\bigg|_{\tau=t} - \left[\frac{\d^2}{\d\tau^2} F_c[\rho_t,\mu_\tau]\right] \bigg|_{\tau=t} \\
    &\le \left[ - 2\lambda_{c,1}\frac{\d }{\d t} F_c[\rho_t,\mu_\tau] + 2\lambda_{c,2}\frac{\d}{\d\tau} F_c[\rho_t,\mu_\tau] \right]\bigg|_{\tau = t} \le -2\lambda_c h(t)\,.
    \end{align*}
    We conclude using Gr\"onwall's estimate. 
\end{proof}
\begin{proposition}[Contraction]\label{prop:contraction}
    Fix $T>0$.
    Let $\gamma_t$ and $\gamma_t'$ be any two solutions of the dynamics \eqref{eq:dynamics_competitive}, with initial conditions $\gamma_0,\gamma_0'\in\P_2^{ac}(\R^{d_1})\times\P_2^{ac}(\R^{d_2})$ such that $D_c(\gamma_0)<\infty$ and $D_c(\gamma_0')<\infty$. Assume $\gamma_t,\gamma_t'\in\P_2^{ac}(\R^{d_1})\times\P_2^{ac}(\R^{d_2})$ for all $t\in [0,T)$ and  $\nabla_z\delta_\rho F_c[\gamma_t](z), \nabla_z\delta_\rho F_c[\gamma_t'](z),\nabla_x\delta_\mu F_c[\gamma_t](x),\nabla_x\delta_\mu F_c[\gamma_t'](x)  $ are locally Lipschitz in $z,x$ for all $t\in [0,T)$.
    Then $\gamma_t$ and $\gamma_t'$ satisfy
    \begin{align*}
        \Wbar(\gamma_t,\gamma_t') \le e^{-\lambda_c t} \Wbar(\gamma_0,\gamma_0') \quad \text{ for all } t\in [0,T)\,.
    \end{align*}
\end{proposition}
\begin{remark}\label{remark:P_ac}
    Although the contraction theorem is stated for measures that are absolutely continuous, the result \cite[Theorem 23.9]{Villani07}  on the time derivative of $\Wbar$ can be generalized to the case where $\rho_t=\delta_{z(t)},\rho_t'=\delta_{z'(t)}$ or $\mu_t=\delta_{x(t)},\mu_t'=\delta_{x'(t)}$ for all times $t\ge 0$. For details, see Lemma~\ref{lem:propagating_dirac_contraction} for the setting in which $\beta=0$ with $\mu_0=\delta_{x(0)}$,  and $\alpha>0$ with $\rho_0 \in\P_2^{ac}$.
\end{remark}
\begin{proof}
    Define $\nabla \varphi_t(z)$ and $\nabla\psi_t(x)$ so that $\rho_t={\nabla \varphi_t}_\# \rho_t'$ and $\mu_t = {\nabla \psi_t}_\# \mu_t'$. Due 
    to Lemma~\ref{lem:velocities_in_L2},
    \begin{align*}
        \int \norm{\nabla_z\delta_\rho F_c[\rho_t,\mu_t](z)}^2 \d \rho_t(z) + \int \norm{\nabla_x \delta_\mu F_c[\rho_t,\mu_t](x)}^2 \d \mu_t(x) < \infty \quad \forall\ t\ge 0\,,
    \end{align*}
    with the same holding for $\gamma'$.
    From \cite[Theorem 23.9]{Villani07}, the time derivative of the joint metric along solutions to \eqref{eq:dynamics_competitive} can be controlled by 
    \begin{align*}
        \ddt \Wbar(\gamma_t,\gamma_t')^2 &= 2\iint \begin{bmatrix}
            \nabla \varphi_t^{-1}(z) - z  \\
            \nabla \psi_t^{-1}(x)- x
        \end{bmatrix}
        \cdot
        \begin{bmatrix}
            -\nabla \delta_\rho F_c[\gamma_t](z) \\
            \nabla \delta_\mu F_c[\gamma_t](x)
        \end{bmatrix} \d \gamma_t(z,x)\\
        & \quad +
        2 \iint \begin{bmatrix}
            \nabla \varphi_t(z) - z  \\
            \nabla \psi_t(x) - x
        \end{bmatrix}
        \cdot
        \begin{bmatrix}
            -\nabla \delta_\rho F_c[\gamma_t'](z) \\
            \nabla \delta_\mu F_c[\gamma_t'](x)
        \end{bmatrix} \d \gamma_t'(z,x) \\
        & = 2\iint \begin{bmatrix}
            \nabla \varphi_t(z) - z  \\
            \nabla \psi_t(x) - x
        \end{bmatrix}\cdot
        \begin{bmatrix}
            \nabla \delta_\rho F_c[\gamma_t](\nabla \varphi_t(z))-\nabla \delta_\rho F_c[\gamma_t'](z) \\
            \nabla \delta_\mu F_c[\gamma_t'](x) - \nabla \delta_\mu F_c[\gamma_t](\nabla \psi(x))
        \end{bmatrix} \d \gamma_t'(z,x) \\
        &\le -2\lambda_c \Wbar(\gamma_t,\gamma_t')^2 \quad \forall\ t\in (0,T)\,,
    \end{align*}
    where the last inequality follows from Lemma~\ref{lem:Wbar_upper_bd}. By Gr\"onwall's lemma, exponential convergence follows. For time $t=0$, the inequality holds by definition since $e^{-\lambda_c t}|_{t=0}=1$.
\end{proof}
Next, our goal is to show that the semigroup for the dynamics \eqref{eq:dynamics_competitive} maps to $\P_2 \times \P_2$. The key ingredient is control of second moments along the evolution. More precisely, we will show that the second moments converge exponentially to a ball and remain in that ball for all time. This result will then allow us to apply a contractive inequality to prove the existence of a steady state for the dynamics \eqref{eq:dynamics_competitive}.

\begin{proposition}[Uniformly Bounded Second Moments]\label{prop:2ndmoment} 
Let $\gamma_t$ be a solution to \eqref{eq:dynamics_competitive} with $\gamma_0\in\P_2\times \P_2$ such that $D_c(\gamma_0)<\infty$. If $\alpha=0$, assume $\rho_0=\delta_{z_0}$ for some $z_0\in\R^{d_1}$. If $\beta=0$, assume $\mu_0=\delta_{x_0}$ for some $x_0\in\R^{d_2}$. Then $\Wbar(\gamma_t,\bar \delta)^2$ satisfies
    \begin{align*}
        \ddt \Wbar(\gamma_t,\bar \delta)^2 \le -\lambda_c \Wbar(\gamma_t,\bar \delta)^2 + 2\hat c\,,
    \end{align*}
    for some $\hat c \ge 0$, where $\bar\delta(z,x)=\delta_{(0,0)}(z,x)$.
    For any time $t\ge 0$, it holds that 
    \begin{align*}
        \int \norm{z}^2 \d \rho_t(z) + \int \norm{x}^2 \d \mu_t(x)\le K\coloneqq\max\left\{\Wbar(\gamma_0,\bar\delta),\frac{2\hat c}{\lambda_c} \right\}\,.
    \end{align*}
\end{proposition}
\begin{proof}
If $\alpha=0$ ($\beta=0$), then $\rho_t$ ($\mu_t$) remains a Dirac Delta for all times, and so its second moment vanishes. Let us assume $\alpha,\beta>0$. Thanks to the diffusion, $\gamma_t\in\P^{ac}\times\P^{ac}$ for all $t>0$, and thanks to Lemma~\ref{lem:velocities_in_L2} we have $D_c(\gamma_t)<\infty$ for all $t\ge 0$. 
The sum of the squared second moments can be written as $\Wbar(\gamma_t,\bar \delta)^2$.
Directly differentiating the second moments along solutions to \eqref{eq:dynamics_competitive}, 
we have
\begin{align}\label{eq:ddt_W2}
    \ddt \Wbar(\gamma_t,\bar \delta)^2 \le 2 \iint \begin{bmatrix}
        z \\ x
    \end{bmatrix} \cdot
    \begin{bmatrix}
        \nabla_z \delta_\rho F_c[\gamma_t](z) \\ -\nabla_x \delta_\mu F_c[\gamma_t](x)
    \end{bmatrix} \d \rho_t(z) \d \mu_t(x) \,.
\end{align}
Our goal is to upper-bound the right-hand side in terms of $\Wbar(\gamma_t,\bar \delta)^2$. We will compute the terms in $F_c$ separately, starting with the entropy terms. We have 
\begin{align*}
    -\int z \cdot \nabla_z \delta_\rho \Hc(\rho_t) \d \rho_t(z) &= -\int z \cdot \nabla \log \rho_t \d \rho_t(z) = -\int z \cdot \nabla_z \rho_t(z) \d z= d_1\int \d \rho_t(z)  = d_1\,,\\
    -\int x \cdot \nabla_x \delta_\mu \Hc(\mu_t) \d \mu_t(x) &= -\int x \cdot \nabla \log \mu_t \d \mu_t(x) = -\int x \cdot \nabla_x \mu_t(x) \d x= d_2\int \d \mu_t(x) = d_2.
\end{align*}
For the remaining terms, we will use the convexity inequality
\begin{align*}
    f(y) \ge f(y') + \nabla f(y')\cdot (y-y') + \frac{\lambda}{2}\norm{y-y'}^2 \quad \forall y,y'\in\R^d\,.
\end{align*}
Applying this to $W_1$, we use a change of variables on half of the integral and use the symmetry of $W_1$, in particular that $\nabla W_1(z-z')=-\nabla W_1(z'-z)$, to rewrite 
\begin{align*}
    -\iint z \cdot \nabla W_1(z-z') \d \rho_t(z) \d \rho_t(z') 
    &= -\frac{1}{2} \int (z-z') \cdot \nabla W_1(z-z')  \d \rho_t(z) \d \rho_t(z')\,.
\end{align*}
Selecting $y'=z-z'$ and $y=0$ results in
\begin{align*}
    -\frac{1}{2} \iint (z-z')& \cdot \nabla W_1(z-z')  \d \rho_t(z) \d \rho_t(z') \le \iint -\frac{\lambda_{W,1}}{4} \norm{z-z'}^2 \d \rho_t(z) \d\rho_t(z')  \\
    &\quad +\iint [ W_1(0) - W_1(z-z')  ] \d \rho_t(z) \d\rho_t(z') \le 0 \,,
\end{align*}
since $W_1(0)\le W_1(z)$ for all $z$ by symmetry and convexity of $W_1$. A similar estimate can be computed for $W_2$. For the cross-term and potential terms, we will use the convexity inequality for each species, with $y=0$ and $y'=z$ for $z\cdot \nabla_z f(z,x)$ and $y=0$ and $y'=x$ for $-x\cdot \nabla_x f(z,x)$,
\begin{align*}
    \iint & \begin{bmatrix}
        z \\ x
    \end{bmatrix}\cdot
    \begin{bmatrix}
        \nabla_z ( f(z,x)-V_1(z)) \\ -\nabla_x (f(z,x)-V_2(x))
    \end{bmatrix} \d \rho_t(z) \d \mu_t(x) \le -\int \frac{\lambda_{c,1}}{2}\norm{z}^2 \d \rho_t(z) -\int \frac{\lambda_{c,2}}{2}\norm{x}^2 \d \mu_t(x) \\
    &+ \iint (-f(0,x) + f(z,0) + V_1(0)-V_1(z) + f(z,x)-f(z,x) + V_2(0)-V_2(x) )\d \gamma_t(z,x)\,.
\end{align*} 
Next, we use that $f(z,0)-V_1(z)$ is $\lambda_{c,1}$ concave in $z$, $f(0,x)+V_2(x)$ is $\lambda_{c,2}$ convex in $x$, and continuity to define
\begin{align*}
    c := V_1(0)+V_2(0) + \max_{z\in\R^{d_1}} f(z,0)-V_1(z) + \max_{x\in\R^{d_2}} -f(0,x)-V_2(x) < \infty\,,
\end{align*}
which gives an upper-bound for the cross term
\begin{align*}
    \iint \begin{bmatrix}
        z \\ x
    \end{bmatrix} \cdot
    \begin{bmatrix}
        \nabla_z f(z,x) \\ -\nabla_x f(z,x)
    \end{bmatrix} &\d \rho_t(z) \d \mu_t(x) \le -\int \frac{\Lambda_{f,1}}{2}\norm{z}^2 \d \rho_t(z) -\int \frac{\lambda_{f,2}}{2}\norm{x}^2 \d \mu_t(x) + c\,.
\end{align*}
Note that $c\ge 0$ because $\max_{z\in\R^{d_1}} f(z,0)-V_1(z) \ge f(0,0)-V_1(0)$ and $\max_{x\in\R^{d_2}}-f(0,x)-V_2(x) \ge -f(0,0)-V_2(0)$.
Combining all terms gives
\begin{align*}
    \frac{1}{2}\ddt \Wbar(\gamma_t,\bar \delta)^2 \le -\frac{\lambda_c}{2} \Wbar(\gamma_t,\bar\delta)^2 + \hat c\,,
\end{align*}
where $\hat c = c+ \alpha d_1 + \beta d_2$. It holds that $\hat c\ge 0$ because all terms are non-negative. 
The solution to this ODE satisfies 
\begin{align*}
    \Wbar(\gamma_t,\bar \delta)^2 \le \max\left\{\Wbar(\gamma_0,\bar\delta)^2,\frac{2\hat c}{\lambda_c} \right \} \quad \forall t\ge 0\,.
\end{align*}
\end{proof}

\begin{proposition}[Existence and Uniqueness of Steady States]\label{prop:competitive_ground_state}
  There exists a unique steady state $\gamma_\infty=(\rho_\infty,\mu_\infty)\in \tP_2 \times \tP_2$ of equation~\eqref{eq:dynamics_competitive} according to Definition \ref{def:sstates} with $\int \rho_\infty = 1,\int \mu_\infty=1$. This steady state is a Nash equilibrium for $F_c$. Additionally,
    \begin{itemize}
        \item if $\alpha>0$, the steady state $\gamma_\infty$ satisfies $\rho_\infty\in L_+^1(\R^{d_1}) \cap C^2(\R^{d_1})$ with $\norm{\rho_\infty}_1 = 1$ and $\supp(\rho_\infty)=\R^{d_1}$;
        \item if $\beta>0$, the steady state $\gamma_\infty$ satisfies $\mu_\infty\in L_+^1(\R^{d_2})\cap C^2(\R^{d_2})$ with $\norm{\mu_\infty}_1 = 1$ and $\supp(\mu_\infty)=\R^{d_2}$.
    \end{itemize}   

\end{proposition}
\begin{proof}
We split the proof into three steps.
First, (i) we prove there exists $\gamma_\infty \in\P_2\times \P_2$ satisfying $\Wbar(\gamma_\infty,\gamma(t))=0$ for all $t\ge0$ with $\gamma(t)$ the solution to \eqref{eq:dynamics_competitive} with initial condition $\gamma_\infty$. Then (ii) we will show that $\gamma_\infty$ is also a critical point; specifically it is a Nash equilibrium. 
Lastly, (iii) we show $\gamma_\infty$ has the required regularity properties to satisfy Definition \ref{def:sstates}. 

(i) Fix $\gamma_0\in\P_2 \times \P_2$ such that $D_c(\gamma_0)<\infty$. If $\alpha=0$, assume $\rho_0=\delta_{z_0}$ for some $z_0\in\R^{d_1}$. If $\beta=0$, assume $\mu_0=\delta_{x_0}$ for some $x_0\in\R^{d_2}$. By Proposition~\ref{prop:2ndmoment}, we know that $\gamma_t\in\P_2\times\P_2$ for all $t\ge 0$. If $\alpha,\beta>0$, we have $\gamma(t)\in \P^{ac}_2 \times \P^{ac}_2$ thanks to the diffusion.
Contraction in $\P^{ac}_2 \times \P^{ac}_2$ follows from Proposition~\ref{prop:contraction}. If $\alpha=0$ or $\beta=0$, contraction is in $\P_2$ instead of $\P_2^{ac}$ (see Remark~\ref{remark:P_ac}).  
Let $T(t):\P_2 \times \P_2 \to \P_2 \times \P_2$ be the semigroup for \eqref{eq:dynamics_competitive}. Using a contraction result, \cite[Lemma 7.3]{carrillo_contractive_2007}, there exists a unique steady state $\gamma_\infty\in \P_2^{ac} \times \P_2^{ac}$ (or $\P_2 \times \P_2$ resp.), that is,
\begin{align*}
    \Wbar(\gamma_\infty, T(t)\gamma_\infty) = 0 \,.
\end{align*}
(ii) To show that $\gamma_\infty$ is a critical point, we consider the optimization problems:
\begin{align*}
    c_1 = \sup_{\rho\in\P_2} F_c(\rho,\mu_\infty)\,, \qquad 
    c_2 = \inf_{\mu\in\P_2} F_c(\rho_\infty,\mu)\,.
\end{align*}
First, note that the supremum $c_1$ and infimum $c_2$ are in fact attained by a unique maximizer $\rho_\dagger\in\P_2$ and unique minimizer $\mu_\dagger\in\P_2$. This can be shown using similar arguments as those in Proposition \ref{prop:existenceGa}; concavity of $F_c(\cdot,\mu_\infty)$ and convexity of $F_c(\rho_\infty,\cdot)$ follow from Lemma~\ref{lem:Gc_convex_concave}. The functional $F_c(\cdot,\mu_\infty)$ has an upper-bound due to concavity via a similar argument as in Lemma~\ref{lem:Fa_non_negative} and Lemma~\ref{lem:upper_lower_bounds_f}, and $F_c(\rho_\infty,\cdot)$ has a lower-bound using the same argument with convexity. Upper and lower-semicontinuity complete the set of required ingredients for the proof of Proposition \ref{prop:existenceGa}. We can then write the well-defined optimization problem
\begin{align}\label{eq:EL_steady_state_arg}
    \rho_\dagger = \amax_{\rho\in\P_2} F_c(\rho,\mu_\infty)\,, \qquad 
    \mu_\dagger = \amin_{\mu\in\P_2} F_c(\rho_\infty,\mu)\,.
\end{align}
As a maximizer and minimizer respectively, they satisfy the Euler-Lagrange (EL) conditions
\begin{align*}
    \delta_\rho F_c[\rho_\dagger,\mu_\infty](z)=c_1\, \forall z\in\supp \rho_\dagger\,, \quad \delta_\mu F_c[\rho_\infty,\mu_\dagger](x)=c_2\, \forall x\in\supp\mu_\dagger \,.
\end{align*}
The EL condition for $\rho_\dagger$ is
\begin{align*}
    \delta_\rho F_c[\rho_\dagger,\mu_\infty](z) =-\alpha\log \rho_\dagger(z)+\int f(z,x)\d \mu_\infty(x) -V_1(z) -  (W_1 \ast \rho_\dagger)(z) =\tilde c_1\,\,\forall z\in\supp \rho_\dagger \,.
\end{align*}
Note that the left-hand side of the EL condition is finite for any fixed $z$.
Rearranging, in the case when $\alpha>0$, we have 
\begin{align}\label{eq:EL-Gc-rewritten}
    \rho_\dagger(z) = \exp\left( \frac{1}{\alpha}\left[-\tilde c_1+\int f(z,x)\d \mu_\infty(x)-W_1\ast\rho_\dagger(z) -V_1(z) \right]\right)\,.
\end{align}
We have $W_1\ast\rho_\dagger \in C^2$ due to $W_1\in C^2$, $\int f(z,x)\d\mu_\infty(x) \in C^2$ because $f(\cdot,x)\in C^2$, and $V_1\in C^2$ from Assumptions~\ref{assump:f_lower}(ii), \ref{assump:V_lower}, and \ref{assump:W_lower}. Therefore $\rho_\dagger\in C^2$. If $\alpha=0$, then $\delta_\rho F_c[\rho_\dagger,\mu_\infty](\cdot)\in C^2$ follows immediately. We conclude that $\delta_\rho F_c[\rho_\dagger,\mu_\infty](\cdot)\in C^2$ and $\div{\rho_\dagger \nabla \delta_\rho F_c[\rho_\dagger,\mu_\infty](z)} = 0$ pointwise a.e. Similarly, $\div{\mu_\dagger \nabla \delta_\mu F_c[\rho_\infty,\mu_\dagger](x)} = 0$ pointwise a.e., and if $\beta>0$, then $\mu_\dagger\in C^2$.

Define $S_1(t):\P_2\to\P_2$ and $S_2(t):\P_2\to\P_2$ as the semigroups for the uncoupled dynamics given by
\begin{align}\label{eq:uncoupled_dynamics}
    \partial_t \hat \rho = - \div{\hat \rho \nabla \delta_\rho F_c[\hat \rho,\mu_\infty](z)}\,, \quad \partial_t \hat \mu = \div{\hat \mu \nabla \delta_\mu F_c[\rho_\infty,\hat \mu](x)}\,,
\end{align}
respectively. Note that $S_i(t)$ maps to $\P_2$ by the same argument as in Proposition~\ref{prop:2ndmoment}. Due to the regularity of $\rho_\dagger$ and $\mu_\dagger$ and the fact that they are critical points, they are also steady states of \eqref{eq:uncoupled_dynamics}
\begin{align*}
    \Wass_2(\rho_\dagger,S_1(t)\rho_\dagger) = 0\,, \qquad \Wass_2(\mu_\dagger,S_2(t) \mu_\dagger) = 0\,.
\end{align*}
Since $\gamma_\infty$ is a steady state of \eqref{eq:dynamics_competitive},
\begin{align*}
    \Wass_2(\rho_\infty,S_1(t) \rho_\infty)=0\,, \qquad \Wass_2(\mu_\infty ,S_2(t) \mu_\infty) =0\,.
\end{align*}
Next, we again apply the contraction theorem \cite[Lemma 7.3]{carrillo_contractive_2007} to \eqref{eq:uncoupled_dynamics}, which states that there is a unique steady state for $S_1(t)$ and for $S_2(t)$. Therefore, 
\begin{align*}
    \Wbar(\gamma_\dagger,\gamma_\infty) = 0\,.
\end{align*}

(iii) If $\alpha, \beta>0$, then $\gamma_\dagger\in C^2 \times C^2$. Since equality of $\gamma_\dagger$ with $\gamma_\infty$ holds almost everywhere with respect to the Lebesgue measure, it holds that $\gamma_\infty$ is a weak solution to \eqref{eq:dynamics_competitive}, such that $\gamma_\infty \in (W_{loc}^{1,2} \cap L_{loc}^\infty ) \times (W_{loc}^{1,2} \cap L_{loc}^\infty)$.
When $\alpha>0$, the argument for $\supp(\rho_\infty)=\R^{d_1}$ and $\rho_\infty \in C^2(\R^{d_1})$ follows as in Corollary~\ref{cor:ground-states-supp-Ga}, where instead of showing that $\int f(z,x)\d\mu_\infty(x)+V_1(z) < \infty$ for any fixed $z\in\R^{d_1}$, showing $-\int f(z,x)\d\mu_\infty(x) + V_1(z) < \infty$ for any fixed $z\in\R^{d_1}$ gives the same result. This holds because $-\int f(z,x)\d\mu_\infty(x) + V_1(z)$ is continuous in $z$, just as $\int f(z,x)\d\mu_\infty(x)+V_1(z)$ is continuous in $z$ in the setting of Corollary~\ref{cor:ground-states-supp-Ga}. The setting when $\beta>0$ follows similarly to $\alpha>0$. 
\end{proof}

\begin{proposition}[Uniqueness of Nash Equilibrium]\label{prop:unique_Nash}
    There exists at most one Nash equilibrium in $\tP_2\times \tP_2$ for the energy functional $F_c$.
\end{proposition}
\begin{proof}
    Suppose there exists two equilibria, $\gamma_*\in\tP_2\times\tP_2$ and $\gamma_\infty\in\tP_2\times\tP_2$ such that $\Wbar(\gamma_\infty,\gamma_*)>0$. In the weak sense, the following hold: 
 \begin{align*}
     \nabla_z\delta_\rho F_c[\gamma_\infty](z)&=0 \quad \forall\ z\in\supp\rho_\infty\,, \qquad \nabla_z \delta_\rho F_c[\gamma_*](z) = 0 \quad\, \forall z\in\supp \rho_*\,, \\
     \nabla_x\delta_\mu F_c[\gamma_\infty](x)&=0 \quad \forall\ x\in\supp\mu_\infty\,, \qquad \nabla_x \delta_\mu F_c[\gamma_*](x) = 0 \quad \forall\, x\in\supp \mu_*\,.
 \end{align*}
 We plug $\gamma_*$ and $\gamma_\infty$ into the inequality proved in Lemma~\ref{lem:Wbar_upper_bd}, with $\rho_*=\nabla \varphi_\# \rho_\infty$ and $\mu_*=\nabla \psi_\# \mu_\infty$. Note that both $\nabla \varphi$ and $\nabla \psi$ exist and are invertible (on the support of $\rho_\infty$ and $\mu_\infty$ resp.) because the measures are either absolutely continuous or propagating as Diracs, and so the left-hand side in Lemma~\ref{lem:Wbar_upper_bd} is well-defined. Hence,
 \begin{align*}
        \iint &\begin{bmatrix}
              z - \nabla \varphi(z) \\
              x - \nabla \psi(x)
          \end{bmatrix} \cdot \begin{bmatrix}
              \nabla_z\delta_\rho F_c[\gamma_*](\nabla \varphi(z) )  \\
               - \nabla_x \delta_\mu F_c[\gamma_*](\nabla \psi(x))
          \end{bmatrix} \d \rho_\infty(z) \d \mu_\infty(x) \\
          &= \iint \begin{bmatrix}
              (\nabla \varphi)^{-1}(z) - z \\
              (\nabla \psi)^{-1}(x) - x
          \end{bmatrix} \cdot \begin{bmatrix}
              \nabla_z\delta_\rho F_c[\gamma_*](z )  \\
               - \nabla_x \delta_\mu F_c[\gamma_*](x)
          \end{bmatrix} \d \rho_*(z) \d \mu_*(x)= 0 \ge \lambda_c\Wbar(\gamma_\infty,\gamma_*)^2 \,,
    \end{align*}
    which is a contradiction and therefore the Nash equilibrium is unique. 
\end{proof}

Now we present the proof of Theorem~\ref{thm:convergence_Wbar_competitive}.
 \begin{proof}[Proof of Theorem~\ref{thm:convergence_Wbar_competitive}]
 From Proposition~\ref{prop:competitive_ground_state} we have existence and uniqueness of steady states of \eqref{eq:dynamics_competitive} in $\tP_2\times\tP_2$,
 and that the steady state is a Nash equilibrium for $F_c$. This Nash equilibrium is unique from Proposition~\ref{prop:unique_Nash}. This concludes the proof of (a).

    For (b), the uniform bound on the second moment follows from Proposition~\ref{prop:2ndmoment}. Further, let $(\rho_0,\mu_0)\in\P_2^{ac}(\R^{d_1})\times \P_2^{ac}(\R^{d_2})$, $\rho_\infty = \nabla \varphi_{\#} \rho_t$ and $\mu_\infty=\nabla \psi_{\#} \mu_t$. We compute the time derivative of the joint Wasserstein-2 metric using \cite[Theorem 23.9]{Villani07}.
    We have
     \begin{align*}
          \ddt& \Wbar(\gamma_t,\gamma_\infty)^2 
          = 2\iint \begin{bmatrix}
              \nabla \varphi(z) - z \\
              \nabla \psi(x) - x
          \end{bmatrix} \cdot \begin{bmatrix}
              -\nabla_z \delta_\rho F_c[\rho_t,\mu_t](z) \\
              \nabla_x \delta_\mu F_c[\rho_t,\mu_t](x)
          \end{bmatrix} \d \gamma_t(z,x) \\
          &= 2\iint \begin{bmatrix}
              \nabla \varphi(z) - z \\
              \nabla \psi(x) - x
          \end{bmatrix} \cdot \begin{bmatrix}
              \nabla_z \delta_\rho F_c[\rho_\infty,\mu_\infty](\nabla \varphi(z))-\nabla_z \delta_\rho F_c[\rho_t,\mu_t](z) \\
              \nabla_x \delta_\mu F_c[\rho_t,\mu_t](x)-\nabla_x \delta_\mu F_c[\rho_\infty,\mu_\infty](\nabla \psi(x))
          \end{bmatrix} \d \gamma_t(z,x)\,,
     \end{align*}
     where we can add the second term because $\gamma_\infty$ is a steady state, so $$\int\nabla_z \delta_\rho F_c[\rho_\infty,\mu_\infty](\nabla \varphi(z))\d\rho_t(z)=0\,.$$ This can be shown by rewriting the support of the steady state using the pushforward from $\rho_\infty$ to $\rho_t$. A similar argument applies to the first variation in $\mu$, and in the setting in which the measure propagating is a Dirac.
    Now we can apply Lemma~\ref{lem:Wbar_upper_bd}
    to obtain $\ddt \Wbar(\gamma_t,\gamma_\infty)^2 \le -2\lambda_c \Wbar(\gamma_t,\gamma_\infty)^2$ which gives exponential convergence with rate $\lambda_c$ using Gr\"onwall's Lemma, completing the proof of (b).
 \end{proof}

\appendix

\section*{Appendix}

The appendices contain proofs in the timescale-separated settings as well as supporting results.
The first two appendices focus on the competitive setting: Appendix~\ref{sec:fast_x_proof} treats the case of a fast algorithm (proof of Theorem~\ref{thm:convergence_fast_mu}), and Appendix~\ref{sec:fast_rho_proof} of a fast population (proof of Theorem~\ref{thm:convergence_fast_rho}). In both cases, the corresponding objective functional is given by 
$G(\rho,x):\P(\R^d) \times \R^d \to [-\infty,\infty]$ as follows:
\begin{align*}
    G(\rho,x) &= \int f_1(z,x) \d \rho(z) + \int f_2(z,x) \d \pi(z)  +\frac{\kappa}{2} \norm{x-x_0}^2 -\cR(\rho)\,,
\end{align*}
where
\begin{align*}
 \cR(\rho)&:=
 \begin{cases}
     \alpha KL(\rho\,|\,\rhot) + \frac{1}{2}\int  W \ast \rho\,\d\rho\,, &\text{ if } \rho\ll \rhot\,,\\
     +\infty  &\text{ else }\,,
 \end{cases}
\end{align*}
where $\alpha>0$ and the reference measure $\rhot$ satisfies $\rhot\in\P(\R^d)\cap L^1(\R^d)$, $\log\rhot\in C^2(\R^d)$ and $\supp\rhot=\R^d$.  
Throughout Appendices~\ref{sec:fast_x_proof} and \ref{sec:fast_rho_proof}, we will assume that Assumptions~\ref{assump:f_lower}(ii), \ref{assump:V_lower} and \ref{assump:W_lower} hold such that
\begin{align*}
    \lambda_b:=\alpha\lambdat -\Lambda_1 >0\qquad \text{ and }
\qquad    \lambda_d:=\lambda_1+\lambda_2+\kappa = \lambda_{V,2}+\lambda_{f,2}>0\,.
\end{align*}

Appendix~\ref{sec:extra_lemmas} contains auxiliary lemmas used throughout the paper.

\section{Competitive Objective, Fast Algorithm (Proof of Theorem~\ref{thm:convergence_fast_mu})}\label{sec:fast_x_proof}
Define
$$G_b(\rho) \coloneqq G (\rho,\bx(\rho))\,,$$
where the best response $\bx(\rho)$ of the classifier for any $\rho\in\P(\R^d)$ with finite energy is given by 
\begin{equation}\label{eq:bx}
    \bx(\rho) \coloneqq \amin_{\bar{x}\in\R^d} G(\rho,\bar{x})\,.
\end{equation}
Since for such fixed $\rho\in\P(\R^d)$, the energy $G(\rho,\cdot)$ is strictly $\lambda_d$-convex in $x$, it has a unique minimizer, and so the best response is well defined. If $\rho\in\P(\R^d)$ is such that $G(\rho,x_1)=\pm\infty$ for some $x_1\in\R^d$, then $G(\rho,x)=\pm\infty$ for any $x\in\R^d$ and we define $b(\rho):=0$ in that case.
We begin with auxiliary results computing the first variations of the best response $\bx$ and then the different terms in the first variation of $G_b(\rho)$. The following result can be deduced directly from Definition~\ref{def:sstates}.

\begin{lemma}[Steady states for \eqref{eq:dynamics_competitive_x_fast}]\label{lem:sstate_x_fast}
    Let $\rho_\infty\in L^1_+(\R^d)\cap L^\infty_{loc}(\R^d)$ with $\|\rho_\infty\|_1=1$. Then $\rho_\infty$ is a steady state for the system \eqref{eq:dynamics_competitive_x_fast} if $\rho_\infty\in W^{1,2}_{loc}(\R^d)$, $\nabla W\ast \rho_\infty \in L^1_{loc}(\R^d)$, $\rho_\infty$ is absolutely continuous with respect to $\rhot$, and $\rho_\infty$ satisfies 
\begin{align}\label{eq:ss-rho-xfast}
   &\nabla_z \left(f_1(z,b(\rho_\infty)) - \alpha\log\left(\frac{\rho_\infty(z)}{\rhot(z)}\right) - W\ast\rho_\infty(z) \right) = 0 \,,\qquad \text{ on }\supp(\rho_\infty)\,,
\end{align}
in the sense of distributions, where $\bx(\rho_\infty) \coloneqq \amin_{x} G(\rho_\infty,x)$.
\end{lemma}

\begin{lemma}[First variation of the best response]\label{lem:first_variation_b}
    The first variation of the best response of the classifier at $\rho$ (if it exists) is 
    \begin{align*}
        \delta_\rho \bx[\rho](z) = -Q(\rho)^{-1} \nabla_x f_1(z,\bx(\rho)) \quad \text{ for almost every } z\in \R^d\,,
    \end{align*}
    where $Q(\rho)\succeq (\kappa+\lambda_1+\lambda_2)\Id_d$ is a symmetric matrix, constant in $z$ and $x$, defined as $$Q(\rho)\coloneqq \kappa \Id_d+\int \nabla_x^2 f_1(z,\bx(\rho)) \d \rho(z)  +\int \nabla_x^2 f_2(z,\bx(\rho))\d\pi(z)\,.$$
    In particular, we then have for any $\psi\in C_c^\infty(\R^d)$ with $\int\psi\,\d z=0$ that
    \begin{equation*}
        \lim_{\eps\to 0} \frac{1}{\eps} \left\| \bx[\rho+\eps\psi] - \bx[\rho]-\eps \int \delta_\rho \bx[\rho](z)\psi(z)\d z\right\| = 0\,.
    \end{equation*}
\end{lemma}
\begin{proof}
Let $\psi\in C_c^\infty(\R^d)$ with $\int\psi\,\d z=0$ and fix $\eps>0$. Any minimizer of $G(\rho+\eps\psi,x)$ for fixed $\rho$ must satisfy 
\begin{align*}
    \nabla_x G(\rho+\eps \psi,b(\rho+\eps \psi)) = 0\,.
\end{align*} 
Differentiating in $\eps$, we obtain
\begin{align}\label{eq:depsGc}
   \int \delta_\rho \nabla_x G[\rho+\eps\psi,b(\rho+\eps\psi)]\psi(z)\,\d z 
   +  \nabla_x^2 G(\rho+\eps\psi,b(\rho+\eps\psi)) \int \delta_\rho b[\rho+\eps\psi](z) \psi(z)\,\d z= 0\,.
\end{align} 
Next, we explicitly compute all terms involved in \eqref{eq:depsGc}.
Computing the derivatives yields
\begin{align*}
     \nabla_x G(\rho,x) &= \int \nabla_x f_1(z,x) \d \rho(z)  +\int \nabla_x f_2(z,x) \d \pi(z) + \kappa( x- x_0) \\
    \delta_\rho \nabla_x G [\rho,x] (z) &= \nabla_x f_1(z,x) \\
    \nabla_x^2 G(\rho,x) &=\int \nabla_x^2 f_1(z,x) \d \rho(z)  + \int \nabla_x^2 f_2(z,x) \d \pi(z) + \kappa \Id_d.
\end{align*}
Note that $\nabla_x^2 G$ is invertible
because $\lambda_1 + \lambda_2 + \kappa > 0$.
Inverting this term and substituting these expressions into \eqref{eq:depsGc} for $\eps=0$ gives 
\begin{align*}
    \int \delta_\rho b[\rho](z)\psi(z)\,\d z&= -\bigg[\kappa \Id_d+\int \nabla_x^2 f_1(z,\bx(\rho)) \d \rho(z) 
    + \int \nabla_x^2 f_2(z,\bx(\rho)) \d \pi(z) \bigg]^{-1}  \int \nabla_x f_1(z,\bx(\rho))\psi(z)\,\d z \\
    &=-\int Q(\rho)^{-1}  \nabla_x f_1(z,\bx(\rho))\psi(z)\,\d z \,.
\end{align*}
Finally, the lower bound on $Q(\rho)$ follows thanks to Assumptions~\ref{assump:f_lower}(ii) and \ref{assump:V_lower}.
\end{proof}

\begin{lemma}[First variation of $G_b$]\label{lem:first_variations}
The first variation of $G_b$ is given by
\begin{equation*}
    \delta_\rho G_b[\rho](z) = h_1(z)+h_2(z)+\kappa h_3(z) - \delta_\rho \cR[\rho](z)\,,
\end{equation*}
where
    \begin{align*}
        &h_1(z):=\frac{\delta}{\delta \rho} \left(\int f_1(\tilde z,\bx(\rho)) \d \rho(\tilde z)\right)(z)
        = \< \int \nabla_x f_1(\tilde z,\bx(\rho))\d \rho(\tilde z) , \frac{\delta b}{\delta \rho}[\rho](z)> + f_1(z,\bx(\rho))\,,\\
        &h_2(z):=\frac{\delta}{\delta \rho} \left(\int f_2(\tilde z,\bx(\rho)) \d \pi(\tilde z)\right)(z)
        =\<\int \nabla_x f_2(\tilde z,\bx(\rho)) \d \pi(\tilde z), \frac{\delta b}{\delta \rho}[\rho](z)>\,,\\
        &h_3(z):=\frac{1}{2}\frac{\delta}{\delta \rho} \|\bx(\rho)-x_0\|^2
        =\<\bx(\rho)-x_0,\frac{\delta b}{\delta \rho}[\rho](z)>\,,
    \end{align*}
    and 
    \begin{align*}
        \delta_\rho \cR[\rho](z) = \alpha\log (\rho(z)/\rhot(z)) + (W \ast \rho)(z)\,.
    \end{align*}
\end{lemma}

\begin{proof}
We begin with general expressions for the Taylor expansions of $\bx:\P(\R^d)\to\R^d$ and $f_i(z,\bx(\cdot)):\P(\R^d)\to\R$ for $i=1, 2$ around $\rho$. 
Let $\psi \in \mathcal{T}$ with $\mathcal{T}=\{\psi:\int \psi(z) \d z =0\}$. Then
\begin{align}\label{eq:Taylor1}
    b(\rho+\eps \psi) = \bx(\rho) + \eps \int \frac{\delta b}{\delta \rho}[\rho](z')\psi(z') \d z' + O(\eps^2)
\end{align}
and 
\begin{align}\label{eq:Taylor2}
    f_i(z,b(\rho+\eps \psi)) = f_i(z,\bx(\rho)) + \eps \<\nabla_x f_i(z,\bx(\rho)),\int \frac{\delta b}{\delta \rho}[\rho](z')\psi(z') \d z'> + O(\eps^2)\,.
\end{align}
We compute explicitly each of the first variations:
\begin{enumerate}
    \item[(i)] Using \eqref{eq:Taylor2}, we have
    \begin{align*}
    \int \psi(z)& h_1(z) \d z = \lim_{\eps \rightarrow 0} \frac{1}{\eps} \bigg[ \int f_1(z, b(\rho + \eps \psi))(\rho(z) + \eps \psi(z))\d z-\int f_1(z,\bx(\rho))\rho(z) \d z \bigg] \\
    &=  \< \int \nabla_x f_1 (z,\bx(\rho)) \d \rho(z) ,\int \frac{\delta \bx(\rho)}{\delta \rho}[\rho](z')\psi(z')\d z' > + \int f_1(z,\bx(\rho)) \psi(z) \d z \\
    &= \int\< \int \nabla_x f_1 (z,\bx(\rho)) \d \rho(z) , \frac{\delta \bx(\rho)}{\delta \rho}[\rho](z') > \psi(z')\d z' + \int f_1(z,\bx(\rho)) \psi(z) \d z \\
    \Rightarrow
    h_1(z)&= \< \int \nabla_x f_1(\tilde z,\bx(\rho))\d \rho(\tilde z) , \frac{\delta b}{\delta \rho}[\rho](z)> + f_1(z,\bx(\rho))\,.
\end{align*}
\item[(ii)] Similarly, using again \eqref{eq:Taylor2},
\begin{align*}
     \int \psi(z) h_2(z) \d z 
     &= \lim_{\eps \rightarrow 0} \frac{1}{\eps} \bigg[ \int f_2(z, b(\rho + \eps \psi))\d\pi(z)-\int f_2(z,\bx(\rho))\pi(z) \d z \bigg] \\
     &= \int \<\int \nabla_x f_2(\tilde z,\bx(\rho)) \d \pi(\tilde z), \frac{\delta b}{\delta \rho}[\rho](z)> \psi(z)\d z\\
     \Rightarrow h_2(z)&= \<\int \nabla_x f_2(\tilde z,\bx(\rho)) \d \pi(\tilde z), \frac{\delta b}{\delta \rho}[\rho](z)>\,.
\end{align*}
\item[(iii)] Finally, from \eqref{eq:Taylor1} it follows that
\begin{align*}
    \int \psi(z) h_3(z) \d z &= \lim_{\eps \rightarrow 0} \frac{1}{2\eps} \bigg[  \<b(\rho+\eps \psi)-x_0,b(\rho+\eps \psi)-x_0> -  \<\bx(\rho)-x_0,\bx(\rho)-x_0> \bigg] \\
    &= \int  \<\bx(\rho)-x_0,\frac{\delta b}{\delta \rho}[\rho](z)> \psi(z)\d z \\
    \Rightarrow h_3(z)&= \<\bx(\rho)-x_0,\frac{\delta b}{\delta \rho}[\rho](z)>\,.
\end{align*}
\end{enumerate}
Finally, the expression for $\delta_\rho \cR[\rho]$ follows by direct computation \eqref{eq:bx}.
\end{proof}

\begin{proposition}[Danskin-Type Result]\label{prop:first_variation_Gb}
\sloppy Denote $G_b(\rho) \coloneqq G (\rho,\bx(\rho))$ with $\bx(\rho)$ given by \eqref{eq:bx}. Then $$\delta_\rho G_b[\rho] = \left.\delta_\rho G[\rho] \right|_{x=\bx(\rho)}.$$
\end{proposition}\label{proof_lemma1}
\begin{proof}
    We start by computing $\delta_\rho G(\cdot,x)[\rho](z)$ for any $z,x\in\R^d$:
\begin{align}\label{eq:first-variation-Gc}
    \delta_\rho G(\cdot,x)[\rho](z) 
    &= f_1(z,x) - \delta_\rho \cR[\rho](z).
\end{align}
Next, we compute $\delta_\rho G_b$. Using Lemma~\ref{lem:first_variations}, the first variation of $G_b$ is given by
\begin{align*}
    \delta_\rho G_b[\rho](z) &= h_1(z)+h_2(z)+\kappa h_3(z) - \delta_\rho \cR[\rho](z) \\
    &=  -\<\left[\int\nabla_x f_1(\tilde z,\bx(\rho))\d \rho(\tilde z)+ \int \nabla_x f_2(\tilde z,\bx(\rho)) \d \pi(\tilde z)+\kappa(\bx(\rho)-x_0)\right],\delta_\rho b[\rho](z)>\\&\quad + f_1(z,\bx(\rho)) - \delta_\rho \cR[\rho](z) \,.
\end{align*}
Note that 
\begin{align}\label{eq:gradx-Gc}
    \nabla_x G(\rho,x) = \int\nabla_x f_1(\tilde z,x)\d \rho(\tilde z)+ \int \nabla_x f_2(\tilde z,x) \d \pi(\tilde z)+\kappa(x-x_0)\,,
\end{align}
and by the definition of the best response $\bx(\rho)$, we have $\nabla_x G(\rho,x)|_{x=\bx(\rho)}=0$. Substituting into the expression for $\delta_\rho G_b$ and using \eqref{eq:first-variation-Gc}, we obtain
\begin{align*}
    \delta_\rho G_b[\rho](z) 
    = f_1(z,\bx(\rho)) - \delta_\rho \cR[\rho](z)
    = \delta_\rho G(\cdot,x)[\rho](z) \bigg|_{x=\bx(\rho)}\,. %
\end{align*}
This concludes the proof.
\end{proof}

\begin{proposition}[Displacement concavity]\label{prop:Gb_concave}
Fix $\rho_0,\rho_1\in\P_2(\R^d)$. Along any geodesic $(\rho_s)_{s\in[0,1]}\in\P_2(\R^d)$ connecting $\rho_0$ to $\rho_1$, we have for all $s\in[0,1]$
    \begin{align}\label{eq:ddssGb}
    \ddss G_b(\rho_s) \le - \lambda_b \Wass_2(\rho_0,\rho_1)^2 \,, \qquad 
   \lambda_b:= \alpha\tilde\lambda - \Lambda_1\,.
\end{align}
As a result, the functional $G_b:\P_2(\R^d)\to [-\infty,+\infty]$ is uniformly displacement concave with constant $\lambda_b> 0$. 
\end{proposition}

\begin{proof}
Consider any $\rho_0,\rho_1\in\P_2(\R^d)$.
Then any $\Wass_2$-geodesic $(\rho_s)_{s\in[0,1]}$ connecting $\rho_0$ with $\rho_1$ solves the following system of geodesic equations:
\begin{equation}\label{eq:geodesics}
    \begin{cases}
        \partial_s \rho_s + \div{\rho_s v_s}=0\,,\\
        \partial_s(\rho_s v_s) + \div{\rho_s v_s \otimes v_s}=0\,,
    \end{cases}
\end{equation}
where $\rho_s:\R^d\to\R$ and $v_s:\R^d \to \R^d$ .
The first derivative of $G_b$ along geodesics can be computed explicitly as
\begin{align*}
    \dds G_b(\rho_s) &= \int \nabla_z f_1(z,b(\rho_s)) \cdot v_s(z) \rho_s(z) \d z - \dds \cR(\rho_s) \\
    &+ \<\bigg[\int \nabla_x f_1(z,x) \d \rho_s(z) + \int \nabla_x f_2(z,x) \d \pi(z) + \kappa(x-x_0)\bigg]\bigg|_{x=b(\rho_s)},\dds b(\rho_s) >.
\end{align*}
For the last term, the left-hand side of the inner product is zero by definition of the best response $\bx(\rho_s)$ to $\rho_s$, see \eqref{eq:gradx-Gc}. Therefore
\begin{align*}
    \dds G_b(\rho_s) = \int \nabla_z f_1(z,b(\rho_s)) \cdot v_s(z) \rho_s(z) \d z - \dds \cR(\rho_s)\,.
\end{align*}
Differentiating a second time, using \eqref{eq:geodesics} and integration by parts, we obtain
\begin{align*}
    \ddss G_b(\rho_s) = L_1(\rho_s) + L_2(\rho_s) - \ddss \cR(\rho_s)\,,
\end{align*}
where
\begin{align*}
    L_1(\rho_s) &:=  \int \nabla_z^2 f_1(z,b(\rho_s)) \cdot (v_s \otimes v_s)\, \rho_s \d z = \int \< v_s , \nabla_z^2 f_1(z,b(\rho_s)) \cdot v_s >\, \rho_s \d z\,,\\
    L_2(\rho_s) &:= \int \dds b(\rho_s) \cdot \nabla_{xz}^2 f_1(z,b(\rho_s)) \cdot v_s(z) \,\rho_s(z)\d z\,.
\end{align*}
From \eqref{eq:classic_convexity}, we have that 
$$\ddss \cR (\rho_s) \geq \alpha\tilde\lambda \, \Wass_2(\rho_0,\rho_1)^2\,,$$
and thanks to Assumption~\ref{assump:f_lower}(ii) it follows that
\begin{align*}
   L_1(s)
    &\leq \Lambda_1 \Wass_2(\rho_0,\rho_1)^2 .
\end{align*}
This leaves $L_2$ to bound; we first consider the term $\dds b(\rho_s)$:
\begin{align*}
    \dds b(\rho_s) &= \int \delta_\rho b[\rho_s](\tilde z)  \partial_s \rho_s(\tilde z ) \d \tilde z
    = -\int \delta_\rho b[\rho_s](\tilde z) \div{\rho_s v_s}(\tilde z) \d \tilde z \\
    &= \int \nabla_z \delta_\rho b[\rho_s] (\tilde z) \cdot v_s(\tilde z) \d \rho_s (\tilde z).
\end{align*}
Defining $u(\rho_s)\in\R^d$ by
\begin{align*}
    u(\rho_s) \coloneqq \int \nabla_{xz}^2 f_1(z,b(\rho_s)) \cdot v_s(z) \d \rho_s(z)\,,
\end{align*}
using the results from Lemma~\ref{lem:first_variation_b} for $\nabla_z \delta_\rho b[\rho_s]$
and the fact that $Q(\rho)$ is constant in $z$ and $x$, we have 
\begin{align*}
    L_2(\rho_s) &= -\iint \left[Q(\rho_s)^{-1} \nabla_{xz}^2 f_1(\tilde z,b(\rho_s)) \cdot v_s(\tilde z)\right] \cdot \nabla_{xz}^2 f_1(z,b(\rho_s)) \cdot v_s(z) \,\d\rho_s(z) \d\rho_s(\tilde z)\\
    &= -\<u(\rho_s),Q(\rho_s)^{-1} u(\rho_s) >  \le 0
\end{align*}

Combining all terms together, we obtain
\begin{align*}
    \ddss G_b(\rho_s) \leq -\left(\alpha\tilde\lambda  - \Lambda_1 \right) \Wass_2(\rho_0,\rho_1)^2\,.
\end{align*}
\end{proof}

\begin{remark}\label{rem:f1_f2_upper_bound}
Under some additional assumptions on the functions $f_1$ and $f_2$, we can obtain an improved convergence rate. In particular, assume that for all $z,x\in\R^d$,
\begin{itemize}
\item there exists a constant $\Lambda_{1u} \ge \lambda_1$ such that $\nabla_x^2 f_1(z,x) \preceq \Lambda_{1u} \Id_d$;
    \item there exists a constant $\Lambda_{2u}\ge \lambda_2$ such that $\nabla^2_x f_2(z,x) \preceq \Lambda_{2u} \Id_d$;
    \item there exists a constant $\sigma\ge 0$ such that
$\norm{\nabla_{xz}^2 f_1(z,x)}_2\ge \sigma$.
\end{itemize}
Then we have
$ -Q(\rho_s)^{-1}\preceq -1/(\kappa+\Lambda_{1u}+\Lambda_{2u})\Id_d$. 
Using Lemma~\ref{lem:first_variation_b}, we then obtain a stronger bound on $L_2$ as follows: 
\begin{align*}
    L_2(\rho_s)  &\leq -\frac{1}{\kappa+\Lambda_{1u}+\Lambda_{2u}} \norm{u(\rho_s)}^2 
    \\
    &\leq -\frac{1}{\kappa+\Lambda_{1u}+\Lambda_{2u}}\int \norm{\nabla_{xz}^2  f_1(z,b(\rho_s))}_2^2 \d \rho_s(z) \int \norm{v_s(z)}^2 \d \rho_s(z) \\
    &\leq -\frac{\sigma^2}{\kappa+\Lambda_{1u}+\Lambda_{2u}} \Wass_2(\rho_0,\rho_1)^2.
\end{align*}
This means we can improve the convergence rate in~\eqref{eq:ddssGb} to $\lambda_b:= \alpha\tilde\lambda +\frac{\sigma^2}{\kappa+\Lambda_{1u}+\Lambda_{2u}}-\Lambda_1 $.
\end{remark}

\begin{lemma}[Uniform boundedness of the best response]\label{lemma:x_bounded}
Let Assumption~\ref{assump:f_loose_upper} hold. Then
    for any $\rho\in\P(\R^d)$, we have
    \begin{equation*}
        \norm{\bx(\rho)}^2\leq \norm{x_0}^2 + \frac{2(a_1+a_2)}{\kappa}\,.
    \end{equation*}
\end{lemma}
\begin{proof}
The bound trivially holds if $\rho\in\P(\R^d)$ is such that $G(\rho,x)=\pm\infty$ for some $x\in\R^d$ since then we have $\bx(\rho)=0$. Else,
by definition of the best response $\bx(\rho)$, we have
    \begin{align*}
        \int \nabla_x f_1(z,\bx(\rho)) \d \rho(z) + \int \nabla_x f_2(z,\bx(\rho))\d \pi(z) + \kappa(\bx(\rho) - x_0) = 0\,.
    \end{align*}
    To show that that $\bx(\rho)$ is uniformly bounded, we take the inner product of the above expression with $\bx(\rho)$ itself
    \begin{align*}
        \kappa \|\bx(\rho)\|^2 = \kappa x_0 \cdot \bx(\rho) -  \int \nabla_x f_1(z,\bx(\rho)) \cdot \bx(\rho) \d \rho(z) - \int \nabla_x f_2(z,\bx(\rho)) \cdot \bx(\rho) \d \pi(z)\,.
    \end{align*}
    Using Assumption~\ref{assump:f_loose_upper} to bound the two integrals, together with using Young's inequality to bound the first term on the right-hand side, we obtain
    \begin{align*}
        \kappa \|\bx(\rho)\|^2 \leq \frac{\kappa}{2} \norm{x_0}^2 + \frac{\kappa}{2}\norm{\bx(\rho)}^2 + a_1+a_2\,.
    \end{align*}
\end{proof}

\begin{lemma}[Upper semi-continuity]\label{lem:continuity_fast_x}
Let Assumption
\ref{assump:f_loose_upper} hold. The functional $G:\P(\R^d)\times \R^d \to [-\infty,+\infty]$ is upper semi-continuous when $\P(\R^d)\times \R^d$ is endowed with the product topology of the weak topology and the Euclidean topology. Moreover, the functional $G_b:\P(\R^d)\to [-\infty,+\infty]$ is upper semi-continuous with respect to the weak topology. 
\end{lemma}
\begin{proof}
The functional $G:\P(\R^d)\times \R^d \to [-\infty,+\infty]$ is continuous in the second variable thanks to Assumption~\ref{assump:f_lower}(ii). Similarly, $\int f_1(z,x)\d\rho(z) + \int f_2(z,x)\d\pi(z)$ is continuous in $\rho$ thanks to \cite[Proposition 7.1]{Santambrogio} using the continuity of $f_1$ and $f_2$. Further, $-\cR$ is upper semi-continuous using Lemma~\ref{lem:lsc_entropy_santambrogio} and \cite[Proposition 7.2]{Santambrogio} thanks to Assumptions~\ref{assump:V_lower} and \ref{assump:W_lower}. This concludes the continuity properties for $G$.

The upper semi-continuity of $G_b$ then follows from a direct application of a version of Berge's maximum theorem \cite[Lemma 16.30]{aliprantis_correspondences_1999}. Let $R:= \norm{x_0}^2 + \frac{2(a_1+a_2)}{\kappa}>0$. We define $\varphi:(\P(\R^d),\Wass_2) \twoheadrightarrow \R^d$ as the correspondence that maps any $\rho\in\P(\R^d)$ to the closed ball $\overline{B_R(0)}\subset \R^d$. Then the graph of $\varphi$ is $\Gr \varphi = \P(\R^d)\times \{\overline{B_R(0)}\}$. With this definition of $\varphi$, the range of $\varphi$ is compact and $\varphi$ is continuous with respect to weak convergence, and so it is in particular upper hemicontinuous. Thanks to Lemma~\ref{lemma:x_bounded}, the best response function $\bx(\rho)$ is always contained in $\overline{B_R(0)}$ for any choice of $\rho\in\P(\R^d)$. As a result, maximizing $-G(\rho,x)$ in $x$ over $\R^d$ for a fixed $\rho\in\P(\R^d)$ reduces to maximizing it over $\overline{B_R(0)}$. Using the notation introduced above, we can restrict $G$ to $G:\Gr \varphi\rightarrow \R$ and write
    \begin{align*}
    G_b(\rho) \coloneqq \max_{\hat{x}\in\varphi(\rho)}-G(\rho,\hat{x}).
    \end{align*}
    Because $G(\rho,x)$ is upper semi-continuous when $\P(\R^d)\times \R^d$ is endowed with the product topology of the weak topology and the Euclidean topology, \cite[Lemma 16.30]{aliprantis_correspondences_1999} guarantees that $G_b(\cdot)$ is upper semi-continuous in the weak topology.
\end{proof}

\begin{proposition}[Ground state]\label{prop:existenceGb}
Let Assumption~\ref{assump:f_loose_upper} hold.
There exists a unique maximizer $\rho_*$ for the functional $G_b$ over $\P(\R^d)$, it satisfies $\rho_*\in\P_2(\R^d)\cap L^1(\R^d)$ and is absolutely continuous with respect to $\rhot$.
\end{proposition}
\begin{proof}
    Uniqueness of the maximizer (if it exists) is guaranteed by the uniform concavity provided by Lemma~\ref{prop:Gb_concave}.
To show existence of a maximizer, we use the direct method in the calculus of variations, requiring the following key properties for $G_b$: (1) boundedness from above, (2) upper semi-continuity, and (3) tightness of any maximizing sequence.
   To show (1), note that $\nabla_z^2 (f_1(z,x)+\alpha \log\rhot(z))\preceq -(\alpha\tilde\lambda-\Lambda_1)\Id_d$ for all $z,x\in \R^d\times\R^d$ by Assumptions~\ref{assump:f_lower}(ii) and \ref{assump:V_lower}, and so
   \begin{equation}\label{eq:largerhob}
       f_1(z,x)+\alpha \log\rhot(z) \leq c_0(x) - \frac{(\alpha\tilde\lambda-\Lambda_1)}{4} \|z\|^2 \qquad \forall (z,x)\in\R^d\times\R^d
   \end{equation}
   with $c_0(x):= f_1(0,x)+\alpha \log\rhot(0) + \frac{1}{\alpha\tilde\lambda-\Lambda_1} \|\nabla_z\left[f_1(0,x)+\alpha \log\rhot(0) \right]\|^2$. Therefore,
   \begin{align*}
       G_b(\rho) 
       &= \int \left[f_1(z,\bx(\rho))+\alpha \log\rhot(z)\right] \d\rho(z) +\int f_2(z,\bx(\rho))\d\pi(z)
       +\frac{\kappa}{2}\|\bx(\rho)-x_0\|^2\\
       &\qquad 
       -\alpha \int\rho\log\rho - \frac{1}{2}\int \rho W\ast \rho\\
        &\le c_0(\bx(\rho)) 
        +\int f_2(z,\bx(\rho))\d\pi(z)
       +\frac{\kappa}{2}\|\bx(\rho)-x_0\|^2\,.
   \end{align*}
   To estimate each of the remaining terms on the right-hand side, denote $R:= \norm{x_0}^2 + \frac{2(a_1+a_2)}{\kappa}$ and recall that $\|\bx(\rho)\|\le R$ for any $\rho\in\P(\R^d)$ thanks to Lemma~\ref{lemma:x_bounded}. By continuity of $f_1$ and $\log\rhot$, there exists a constant $c_1\in \R$ such that
   \begin{align}\label{eq:c0bound}
      \sup_{x\in B_R(0)} c_0(x)=\sup_{x\in B_R(0)} \left[f_1(0,x)+\alpha \log\rhot(0) + \frac{1}{\alpha\tilde\lambda-\Lambda_1} \left\|\nabla_z\left(f_1(0,x)+\alpha \log\rhot(0) \right)\right\|^2\right]\le c_1\,.
   \end{align}
   The term $\int f_2(z,b(\rho))\d \pi(z)$ is controlled by $c_2$ thanks to
  \eqref{eq:V_upper-b}.
   The third term can be bounded directly to obtain
      \begin{align*}
       G_b(\rho) 
        &\le c_1+c_2 + \kappa (R^2 + \|x_0\|^2)\,.
   \end{align*}
   This concludes the proof of (1). Statement (2) was shown in Lemma~\ref{lem:continuity_fast_x} using Assumption~\ref{assump:f_loose_upper}. Then we obtain a maximizing sequence $(\rho_n)\in \P(\R^d)$ which is in the closed unit ball of $C_0(\R^d)^*$ and so the Banach-Anaoglu theorem \cite[Theorem 3.15]{rudin_functional_1991} there exists a limit $\rho_*$ in the Radon measures and a subsequence (not relabeled) such that $\rho_n \weakstar \rho_*$. In fact, $\rho_*$ is absolutely continuous with respect to $\rhot$ as otherwise $G_b(\rho_*)=-\infty$, which contradicts that $G_b(\cdot)>-\infty$ somewhere. We conclude that $\rho_*\in L^1(\R^d)$ since $\rhot\in L^1(\R^d)$. To ensure $\rho_*\in\P(\R^d)$, we require (3) tightness of the minimizing sequence $(\rho_n)$. By Markov's inequality \cite{ghosh} it is sufficient to establish a uniform bound on the second moments:
 \begin{equation}\label{eq:tightness-Gb}
      \int \|z\|^2\d\rho_n(z) <C\qquad \forall n\in\mathrm{N}\,.
  \end{equation}
To see this we proceed in a similar way as in the proof of Proposition~\ref{prop:existenceGa}. Defining
    \begin{align*}
        K(\rho) := -\int \left[f_1(z,\bx(\rho))+\alpha \log \rhot(z)\right] \d \rho(z) + \alpha \int \rho \log \rho \, \d z + \frac{1}{2}\int \rho W \ast \rho \, \d z\,,
    \end{align*}
    we have $K(\rho) = -G_b(\rho) + \int f_2(z,\bx(\rho))\d\pi(z) + \frac{\kappa}{2}\norm{\bx(\rho)-x_0}^2$. Then using again the bound on $\bx(\rho)$ from Lemma~\ref{lemma:x_bounded},
    \begin{align*}
        K(\rho) &\le -G_b(\rho) + \sup_{x\in B_R(0)} \int f_2(z,x)\d\pi(z) + \kappa \left(R^2 + \norm{x_0}^2\right)\\
        &\le -G_b(\rho) + c_2 + \kappa \left(R^2 + \norm{x_0}^2\right)\,,
    \end{align*}
 where the last inequality is thanks to  \eqref{eq:V_upper-b}. 
 Hence, using the estimates \eqref{eq:largerhob} and \eqref{eq:c0bound} from above,  
   \begin{align*}
         \frac{(\alpha\tilde\lambda-\Lambda_1)}{4} \int \norm{z}^2 \d \rho_n(z) 
         &\le c_1 -\int \left[f_1(z,\bx(\rho_n))+\alpha \log \rhot(z)\right] \d \rho_n(z) \,.
    \end{align*}
Applying Lemma~\ref{lem:rho_log_rho_lower_bound} with $\eps= \frac{\alpha \lambdat-\Lambda_1}{8\alpha} $ in the same fashion as in the proof of Proposition~\ref{prop:existenceGa} and noting that the sequence $(\rho_n)$ is minimizing $(-G_b)$,
    \begin{align*}
        \frac{(\alpha\tilde\lambda-\Lambda_1)}{4} \int \norm{z}^2 \d \rho_n(z)  & \le c_1 + K(\rho_n) + \alpha\eps \int \norm{z}^2 \d \rho_n(z) + \alpha c_\eps
         \\
        \Rightarrow \frac{(\alpha\tilde\lambda-\Lambda_1)}{8} \int \norm{z}^2 \d \rho_n(z)  &\le c_1 -G_b(\rho_n) + c_2 + \kappa \left(R^2 + \norm{x_0}^2\right) + \alpha c_\eps \\
        &\le c_1 -G_b(\rho_1) + c_2 + \kappa \left(R^2 + \norm{x_0}^2\right)  + \alpha c_\eps <\infty\,,
    \end{align*}
    which uniformly bounds the second moments of $(\rho_n)$. This concludes the proof for the estimate \eqref{eq:tightness-Gb} and also ensures that $\rho_*\in\P_2(\R^d)$.
\end{proof}

    \begin{corollary}\label{cor:ground-states-supp-Gb}
        Any maximizer $\rho_*$ of $G_b$ is a steady state for equation~\eqref{eq:dynamics_competitive_x_fast}, and satisfies $\supp(\rho_*)=\supp(\rhot)=\R^d$ and $\rho_*\in C^2(\R^d)$.
    \end{corollary}
    \begin{proof}
To show that $\rho_*$ is a steady state according to Lemma~\ref{lem:sstate_x_fast} we can follow exactly the same argument as in the proof of Proposition~\ref{prop:competitive_ground_state}, just replacing $\int f_1(z,x)\,\mu_*(x)$ with $f_1(z,\bx(\rho_*))$. It remains to show that $\supp(\rho_*)=\supp(\rhot)$.
        As $\rho^*$ is a maximizer, it is in particular a critical point, and therefore satisfies that $\delta_\rho G_b[\rho_*](z)$ is constant on all connected components of $\supp(\rho_*)$. Thanks to Proposition~\ref{prop:first_variation_Gb}, this means there exists a constant $c[\rho_*]$ (which may be different on different components of $\supp(\rho_*)$) such that
\begin{align*}
    f_1(z,b(\rho_*)) - \alpha\log\left(\frac{\rho_*(z)}{\rhot(z)}\right) - W\ast\rho_*(z) = c[\rho_*] \qquad \text{ on } \supp(\rho_*)\,.
\end{align*}
Rearranging, we obtain (for a possibly different constant $c[\rho_*] \neq 0$)
\begin{align}\label{eq:EL-Gb-rewritten}
    \rho_*(z) = c[\rho_*] \rhot(z) \exp{\left[\frac{1}{\alpha}\left( f_1(z,b(\rho_*)) - W\ast\rho_*(z)\right)\right]}\qquad \text{ on } \supp(\rho_*)\,.
\end{align}
Firstly, $\supp(\rho_*)
\subset \supp(\rhot)$ since $\rho_*$ is absolutely continuous with respect to $\rhot$. Secondly, note that $\exp{ \frac{1}{\alpha}f_1(z,b(\rho_*))} >0$ for all $z\in\R^d$ since $f_1\in\R$.
Finally, we claim that $\exp{\left(-\frac{1}{\alpha} W\ast\rho_*(z)\right)}>0$ for $\rho_*$-a.e. $z\in \R^d$. In other words, we claim that $W\ast\rho_*(z) <\infty$ for $\rho_*$-a.e. $z\in \R^d$. 
 Else, if $W\ast\rho_*(z) =\infty$ on a set of non-zero measure, then $\int \rho_* W \ast \rho_* =+\infty$. Since all other terms are upper bounded due to the $x$-player minimization, this implies that $G_b(\rho_*)=-\infty$. Since $\rho^*$ is a maximizer of $G_b$, this would mean $G_b\equiv -\infty$, which is a contradiction.
We conclude that $\supp(\rho_*)=\supp(\rhot)$. 
Finally,  $\rho_*\in C^2(\R^d)$ thanks to $\rhot, f_1, W\in C^2(\R^d)$.
\end{proof}

\begin{remark}
Note that $\rhot \in L^\infty(\R^d)$ by Assumption~\ref{assump:V_lower}.
    If we have in addition that $f_1(\cdot,x)\in L^\infty(\R^d)$ for all $x\in \R^d$, then the maximizer $\rho_*$ of $G_b$ is in $L^\infty(\R^d)$ as well. This follows directly by bounding the right-hand side of \eqref{eq:EL-Gb-rewritten}.
\end{remark}

With the notion of steady state given in Lemma~\ref{lem:sstate_x_fast}, we can obtain improved regularity for any steady state $\rho_\infty$ under mild additional assumptions.

\begin{lemma}\label{lem:steady_states_in_C}
    Any steady state $\rho_\infty$ for equation \eqref{eq:dynamics_competitive_x_fast} has continuous Lebesgue density, $\rho_\infty\in C(\R^d)$.
\end{lemma}
\begin{proof}
    Thanks to our assumptions, we have $f_1(\cdot,b(\rhoinf)) + \alpha \log \rhot(\cdot) \in C^1$, which implies that $\nabla_z (f_1(\cdot,b(\rhoinf)) + \alpha \log \rhot(\cdot)) \in L_{loc}^\infty$. By the definition of a steady state, $\rho_\infty \in L^1 \cap L_{loc}^\infty$ and $\nabla W \ast \rhoinf \in L_{loc}^\infty$.
    Let 
    \begin{align*}
        h(z) \coloneqq  \rhoinf(z)\nabla_z \left[ f_1(z,\bx(\rhoinf))+ \alpha \log \rhot(z)-(W\ast\rhoinf)(z)\right]\,.
    \end{align*}
    Then by the aforementioned regularity, we obtain $h\in L_{loc}^1 \cap L_{loc}^\infty$.
    By interpolation, it follows that $h\in L_{loc}^p$ for all $1<p<\infty$. This implies that $\div{ h} \in W_{loc}^{-1,p}$. Since $\rhoinf$ is a weak $W_{loc}^{1,2}$-solution of \eqref{eq:ss-rho-xfast}, we have
    \begin{align*}
       \Delta \rhoinf= \div{h}\,,
    \end{align*}
and so by classic elliptic regularity theory we conclude $\rhoinf \in W_{loc}^{1,p}$. Finally, applying Morrey's inequality, we have $\rhoinf \in C^{0,k}$ where $k=\frac{p-d}{p}$ for any $d < p < \infty$. Therefore $\rhoinf \in C(\R^d)$ (after possibly being redefined on a set of measure zero).  
\end{proof}

 With the above preliminary results, we can now show the HWI inequality, which implies again a Talagrand-type inequality and a generalized logarithmic Sobolev inequality.

 \begin{proposition}[HWI inequalities]\label{prop:HWI-Gb}
Define the dissipation functional
\begin{align*}
    D_b(\gamma):= \iint \| \nabla_z \delta_\rho G_b[\rho](z) \|^2\d\rho(z)\,.
\end{align*}
Let Assumption~\ref{assump:f_loose_upper} hold.
Denote by $\rho_*$ the unique maximizer of $G_b$.
   \item[\textbf{(HWI)}] Let $\rho_0,\rho_1\in\P_2(\R^d)$ such that $G_b(\rho_0), G_b(\rho_1),D_b(\rho_0)<\infty$. Then 
    \begin{align}\label{eq:HWI_ineq-Gb}
        G_b(\rho_0)-G_b(\rho_1)\le \Wass_2(\rho_0,\rho_1)\sqrt{D_b(\rho_0)} - \frac{\lambda_b}{2} \,\Wass_2(\rho_0,\rho_1)^2
    \end{align}
    \item[\textbf{(LogSobolev)}] Any $\rho\in\P_2(\R^d)$ such that $G_b(\rho), D_b(\rho)<\infty$ satisfies
    \begin{equation}\label{eq:logSob-Gb}
        D_b(\rho)\ge 2\lambda_b \,G_b(\rho\, |\, \rho_*)\,.
    \end{equation}
    \item[\textbf{(Talagrand)}] For any $\rho\in\P_2(\R^d)$ such that $G_b(\rho)<\infty$, we have
    \begin{equation}\label{eq:Talagrand-Gb}
    \Wass_2(\rho,\rho_*)^2 \leq \frac{2}{\lambda_b} G_b(\rho\,|\,\rho_*)\,.
    \end{equation}
\end{proposition}
\begin{proof}
    The proof for this result follows analogously to the arguments presented in the proofs of Proposition~\ref{prop:HWI}, Corollary~\ref{cor:logSob} and Corollary~\ref{cor:Talagrand}, using the preliminary results established in Proposition~\ref{prop:Gb_concave} and Proposition~\ref{prop:existenceGb}.
\end{proof}

\begin{proof}[Proof of Theorem~\ref{thm:convergence_fast_mu}]
Following the same approach as in the proof of Theorem~\ref{thm:PDE_aligned}, the results in Theorem~\ref{thm:convergence_fast_mu} immediately follow by combining Proposition~\ref{prop:existenceGb}, Corollary~\ref{cor:ground-states-supp-Gb} and Proposition~\ref{prop:HWI-Gb} applied to solutions of the PDE \eqref{eq:dynamics_competitive_x_fast}.
\end{proof}

\section{Competitive Objective, Fast Population (Proof of Theorem~\ref{thm:convergence_fast_rho})}\label{sec:fast_rho_proof}
In this section, let Assumptions~\ref{assump:f_lower}(ii), \ref{assump:V_lower} and \ref{assump:W_lower} hold throughout with $\alpha \lambdat > \Lambda_1$. The proof for Theorem~\ref{thm:convergence_fast_rho} uses similar strategies as that of Theorem \ref{thm:convergence_fast_mu}, but considers the evolution of an ODE rather than a PDE. 
Recall that for any $x\in\R^d$ the best response $\br(x)(\cdot)\in\P(\R^d)$ in \eqref{eq:dynamics_competitive_rho_fast} is defined as
$$
\br(x) \coloneqq \amax_{\hat{\rho}\in\P}G(\hat{\rho},x)\,,
$$
where the energy $G(\rho,x):\P(\R^d) \times \R^d \to [-\infty,\infty]$ is given as in Appendix~\ref{sec:fast_x_proof} by
\begin{align*}
    G(\rho,x) &= \int f_1(z,x) \d \rho(z) + \int f_2(z,x) \d \pi(z)  +\frac{\kappa}{2} \norm{x-x_0}^2 -\alpha KL(\rho\,|\,\rhot) - \frac{1}{2}\int (W \ast \rho)(z)\d \rho(z)\,.
\end{align*}
We start by showing that the best response $\br(x)$ exists and is uniquely defined for each $x\in\R^d$ (Proposition~\ref{lem:existence_of_maximizer}). In order to show convergence to equilibrium for the algorithm dynamics $x(t)$ (Theorem~\ref{thm:convergence_fast_rho}), we need to differentiate the energy 
$G_d(x) := \left.G(\rho,x)\right|_{\rho=\br(x)}$ in $x$, which is a non-trivial task given that it is not obvious that the best response $\br(x)$ is even differentiable in $x$. Our goal is therefore to achieve a Danskin-like result of the form
    \begin{align*}
        \nabla_x G_d(x) = \left.\left(\nabla_x G(\rho,x)\right)\right|_{\rho=\br(x)}
    \end{align*}
(Proposition~\ref{prop:danskin-br}). Danskin's theorem, also known as the envelope theorem, is a classical result in (Euclidean) game theory. In zero-sum games, the derivative of a cost function through the best response of the other player is equal to the derivative with the best response plugged in after differentiating, and is used to prove convexity through the implicitly-defined best response function; see \cite{basar_appendix_2008} for details.
To achieve this result, we will need to use that the best response $\br(x)$ is continuous in $x$ in $\Wass_2$ (Corollary~\ref{cor:br-cont}). For the rest of this section, we use the following notation: denote $\m(z,x) := -f_1(z,x) - \alpha \log \rhot(z)$ and 
    \begin{align*}
        F(\rho,x) :=  \alpha\int \rho(z) \log \rho(z) \,\d z+ \int \m(z,x) \d \rho(z) + \frac{1}{2}\int (W \ast \rho)(z) \d \rho(z)\,.
    \end{align*}
Then $G(\rho,x) =\int f_2(z,x) \d \pi(z)  +\frac{\kappa}{2} \norm{x-x_0}^2 - F(\rho,x)$, and so maximizing $G(\rho,x)$ in $\rho$ corresponds to minimizing $F(\rho,x)$ in $\rho$. Showing that the best response $\br(x)$ is continuous means that for any sequence $(x_n)_{n\in\nat}$ in $\R^d$ converging to $\bar x$ as $n\rightarrow \infty$, we have $\Wass_2(\br(x_n),\br(\bar x))\to 0$. 
For a sequence $x_n\to\bar x$, define
\begin{align}\label{eq:cv-defs}
    F_n(\rho) \coloneqq F(\rho,x_n)\,,\quad \bar F(\rho)\coloneqq F(\rho,\bar x)\,, \quad  \rho_n \coloneqq \amin_{\rho\in\P_2(\R^d)} F_n(\rho)\,, \quad m_n(z) \coloneqq m(z,x_n)\,.
\end{align}
To obtain the continuity of $\br(x)$, we first show that the sequence $(\rho_n)_{n\in\nat}$ has uniformly bounded second moments (obtained from convergence of the second moments Proposition~\ref{prop:cv_second_moments}) and that $F_n \xrightarrow{\Gamma} \bar F$ in $\Wass_2$ (Proposition~\ref{prop:gamma_cv}). For these results, we will need an extension of Fatou's Lemma (Theorem~\ref{thm:fatou_weak_cv_meas}) which gives conditions under which limits and integrals can be exchanged when not only the integrand but also the measure of integration depend on $n\to\infty$. The notion of \textit{asymptotically uniformly integrable (a.u.i)} (Definition~\ref{def:aui}) relaxes the notion of lower-boundedness of the integrand, appears as a condition in the extension of Fatou's Lemma (Theorem~\ref{thm:fatou_weak_cv_meas}), and is a key ingredient to show the $\Gamma$-convergence result in Proposition~\ref{prop:gamma_cv}.
\begin{proposition}\label{lem:existence_of_maximizer} 
For each $x\in\R^d$  there exists a unique maximizer $\rho_*:=\br(x)$ solving $\amax_{\hat{\rho}\in\P} G(\hat{\rho},x)$. Further, $\br(x)\in L^1(\R^d) \cap \P_2(\R^d)$, $\supp (\br(x)) =\supp(\rhot)$, and there exists a function $c:\R^d\to \R$ such that the best response $\rho_*(z)=\br(x)(z)$ solves the Euler-Lagrange equation
\begin{equation}\label{eq:EL-fastrho}
   \delta_\rho G[\rho_*,x](z):= \alpha \log\rho_*(z) +\m(z,x) + (W\ast\rho_*)(z)= c(x)
\end{equation}
 for all $(z,x)\in\supp(\rhot)\times\R^d$.
\end{proposition}
\begin{proof}
    Equivalently, consider the minimization problem for $F(\rho,x)$ for some fixed $x$. 
    Note that $\m(z,x)$ is strictly convex in $z$ for fixed $x$ by Assumptions~\ref{assump:f_lower}(ii) and \ref{assump:V_lower}. Together with Assumption~\ref{assump:W_lower}, we can directly apply the uniqueness and existence result from \cite[Theorem 2.1 (i)]{carrillo_kinetic_2003}.

    The result on the support of $\br(x)$ and the expression for the Euler-Lagrange equation follows by the same arguments as in Corollary~\ref{cor:ground-states-supp-Ga} and Corollary~\ref{cor:ground-states-supp-Gb}, using Assumption~\ref{assump:f_lower}(ii) and that $\alpha \lambdat > \Lambda_1$. The proof that $r(x)\in L_1(\R^d) \cap \P_2(\R^d)$ uses the same arguments as those in Proposition~\ref{prop:existenceGb}, where instead of using $x=b(\rho)$, $x$ is some fixed value.
\end{proof}

\begin{definition}[asymptotically uniformly integrable, \cite{feinberg_fatous_2019}]\label{def:aui}
    A sequence of measurable $(\R \cup \{\pm\infty\})$-valued functions $(f_n)_{n\in\nat}$ is called \emph{asymptotically uniformly integrable (a.u.i) with respect to a sequence of measures $( \mu_n)_{n\in\nat} \subset \P(\R^d)$} if 
    \begin{align*}
        \lim_{R\rightarrow+\infty} \limsup_{n\in\nat} \int_{\R^d} |f_n(z)| \mathbb{I}\{ z\in\R^d \ : \ |f_n(z) | \ge R \} \d\mu_n(z) = 0\,.
    \end{align*}
\end{definition}

\begin{theorem}[Theorem 2.4 in \cite{feinberg_fatous_2019}, Fatou's Lemma for Weakly Converging Measures]\label{thm:fatou_weak_cv_meas}
\sloppy    Let $\mathbf{S}$ be a metric space, $(\mu_n)_{n\in \nat}$ be a sequence of measures on $\mathbf{S}$ that converges narrowly to $\mu\in\M(\mathbf{S})$, and let $(f_n)_{n\in\nat}$ be a sequence of measurable $(\R\cup\{\pm\infty\})$-valued functions on $\mathbf{S}$ such that $\lim_{n\rightarrow\infty,s'\rightarrow s}f_n(s')$ exists for $\mu$-a.e. $s\in\mathbf{S}$. Denote $f_n^-(z)=-\min\{ f_n(z),0 \}$. Let $(f_n^-)_{n\in\nat}$ be a.u.i with respect to $(\mu_n )_{n\in\nat}$. Then
    \begin{align*}
        \liminf_{n\rightarrow\infty} \int_{\mathbf{S}} f_n(s) \mu_n(s) (\d s) \le \int_{\mathbf{S}} \liminf_{n \rightarrow \infty,s'\rightarrow s} f_n(s') \mu(\d s)\,.
    \end{align*}
\end{theorem}
\begin{remark}
    If instead the condition on $f_n$ is strengthened to $(f_n)_{n\in\nat}$ being a.u.i, then the limit bound holds with equality \cite[Corollary 2.8]{feinberg_fatous_2019}; this result will also be utilized.
\end{remark}

\begin{proposition}[$\Gamma$-Convergence of Energy]\label{prop:gamma_cv}
Let Assumption~\ref{assump:f_upper}(b)
hold.
For any sequence $(x_n)$ converging to some limit $\bar x$ in $\R^d$, we have $F_n \xrightarrow{\Gamma} \bar F$ as $n\rightarrow \infty$ in $\Wass_2$ for sequences $(\mu_n)\in\P_2^{ac}$. %
\end{proposition}
\begin{proof}
    We write
    \begin{align*}
        F_n(\rho) = H_0(\rho) + H_n(\rho)
    \end{align*}
    where we define 
    \begin{align*}
        H_0(\rho) &= \alpha \int \rho \log\rho + \frac{1}{2}\int \rho W \ast \rho\,,  \qquad
        H_n(\rho) 
        = \int m_n(z)\d\rho(z)\,. 
    \end{align*}
    where $m_n(z)=\m(z,x_n)$. Convergence in $\Wass_2$ is equivalent to narrow convergence $\mu_n\weakstar \bar \mu$ in $\P(\R^d)$ together with convergence of second moments $\int \norm{z}^2 \d \mu_n(z) \rightarrow \int \norm{z}^2 \d \bar \mu(z)$. This is equivalent to $\int \varphi(z) \d \mu_n(z) \rightarrow \int \varphi(z) \d\bar \mu(z)$ for all $\varphi$ with at most quadratic growth, and implies that $\mu_n$ also converges weakly. First, we will show that for every
    $(\mu_n)\in\P_2^{ac}$ converging to $\bar \mu$ in $\Wass_2$,
    we have
    \begin{align*}
        \bar F (\bar\mu) \le \li F_n (\mu_n)\,.
    \end{align*}
    We accomplish this by considering individual terms, noting that 
    \begin{align*}
        \li F_n (\mu_n) \geq \li H_0 (\mu_n) + \li H_n(\mu_n)\,,
    \end{align*}
    see \cite[Chapter 1, (1.2)]{braides_gamma-convergence_2002}, and thus lower bounds for each term sufficiently lower-bound $\li F_n(\mu_n)$. By Lemma~\ref{lem:lsc_entropy_santambrogio}, 
    \begin{align*}
        \li \int \mu_n \log \mu_n \ge \int \bar \mu \log \bar \mu\,.
    \end{align*}

    The interaction kernel term can be written as $\iint W(z-z') \d (\mu_n \times \mu_n) (z,z')$. Since $W\ge 0$ and is lower-semicontinuous due to Assumption~\ref{assump:W_lower}, applying \cite[Portmanteau Theorem 1.3.4 (iv)]{van_der_vaart_weak_2023} gives that
    \begin{align*}
        \li \int W(z-z') \d (\mu_n \times \mu_n)(z,z') \ge \int W(z-z') \d (\bar\mu \times \bar\mu)(z,z').
    \end{align*}
    This concludes the proof that for any sequence $(\mu_n)\in\P_2^{ac}$ such that $\Wass_2(\mu_n,\bar \mu) \rightarrow 0$, 
    we have
 \begin{align*}
    \li  H_0(\mu_n) \ge H_0(\bar \mu)\,.
 \end{align*}
For the second term $H_n(\mu_n)$, define 
\begin{align*}
    \overline{H}(\bar\mu) &\coloneqq -\int \m(z,\bar x) \d\bar\mu(z) \,.
\end{align*}
We now show that $\li H_n(\mu_n) \ge \overline{H}(\bar\mu)$ via Fatou's Lemma for narrowly converging measures, Theorem~\ref{thm:fatou_weak_cv_meas}. Note that $\nabla_z^2 \m(z,x) \succeq \lambda_b \Id_d$ for all $x,z\in\R^d$, which implies $\norm{\nabla_z^2 \m(z,x)}_2 \ge \lambda_b$ for all $x,z\in\R^d$.
We also have from Assumption~\ref{assump:f_upper}(b) that $\norm{\nabla_{xz}^2 \m(z,x)}_2 \le L$. Hence it follows from Lemma~\ref{lem:f_aui} that $\{m_n^-\}_{n\in\nat}$ is asymptotically uniformly integrable with respect to $(\mu_n)$.
This allows us to apply Theorem~\ref{thm:fatou_weak_cv_meas}, which gives us that
    \begin{align*}
        \li H_n(\mu_n) = \li \int m_n(z) \d \mu_n(z) &\ge \int\liminf_{n\rightarrow\infty,z'\rightarrow z} m_n(z') \bar\mu(z) \d z = \overline{H} (\bar\mu)\,.
    \end{align*}

    Secondly, we must show that for any measure $\bar\mu$ there exists a recovery sequence $(\mu_n)\in\P_2^{ac}$ with $\mu_n$ converging in $\Wass_2$ to $\bar \mu$, such that
    \begin{align*}
        \bar F(\bar\mu) \ge \ls F_n(\mu_n)\,.
    \end{align*}
    We select the constant sequence $\mu_n=\bar\mu$, and evaluate
    \begin{align*}
        \ls F_n(\bar\mu) &= \ls \left(H_0(\bar\mu) + H_n(\bar\mu) \right)\le 
         H_0(\bar\mu) +  \int (\ls m_n(z))\d \bar\mu(z) 
    \end{align*}
    using Fatou's Lemma with the uniform upper bound on $m_n(z)$ from Lemma~\ref{lem:upper_lower_bounds_f}.
    We conclude that
    \begin{align*}
        \ls F_n(\bar\mu) \le H_0(\bar\mu) + \overline{H}(\bar\mu) \,.
    \end{align*}
\end{proof}
  
\begin{proposition}[Uniform Second Moment Bound]\label{prop:rho_n_bdd_moment}
    Given a sequence $x_n\rightarrow \bar x$ in $\R^d$, the sequence $(\rho_n)$ as defined in \eqref{eq:cv-defs} has uniformly bounded second moments.
\end{proposition}
\begin{proof}
Applying Lemma~\ref{lem:upper_lower_bounds_f} (1) with $m(z,x):=-f_1(z,x)-\alpha\log\rhot(z)$, we have $m_n(z):=m(z,x_n) \ge -\hat c_1 + \hat c_2 \norm{z}^2$ for all $n\in\nat$ with $\hat c_2>0$ since $\lambda_b>0$.
Integrating the inequality with respect to $\rho_n$ results in 
\begin{align*}
    \hat c_2 \int &\norm{z}^2 \d\rho_n(z) \le  \hat c_1 + \int m_n(z) \d \rho_n(z)  \\
    & \le \hat c_1 +   \int m_n(z) \d \rho_n(z) +  \alpha\int \rho_n \log \rho_n  +  \frac{\hat c_2}{2}\int \norm{z}^2 \d \rho_n(z) + \hat c_3 + \frac{1}{2} \int \rho_n W \ast \rho_n 
\end{align*}
using Lemma~\ref{lem:rho_log_rho_lower_bound} with $\eps=\hat c_2/(2\alpha)$, denoting $\hat c_3=\alpha c_\eps$.
Using \eqref{eq:V_upper-b}, 
there exists $N>0$ and $c_2>0$
such that $\int f_2(z,x_n)\d\pi(z) \le c_2$ for all $n\ge N$. 
Then for all $n\ge N$,
\begin{align*}
    \frac{\hat c_2}{2}\int \norm{z}^2 \d\rho_n(z) & \le C + \int m_n(z)\d \rho_n(z) + \alpha\int \rho_n \log \rho_n  + \frac{1}{2}\int \rho_n W \ast \rho_n \\
    &\qquad - \frac{\kappa}{2 }\norm{x_n-x_0}^2 - \int f_2(z,x_n) \d \pi(z) \\
    &= C - G(\rho_n,x_n)
\end{align*}
where $C \coloneqq  \hat c_1 + \hat c_3+ c_2 + \max_{n\ge 1}\frac{\kappa}{2}\norm{x_n-x_0}^2$. 
Since $\rho_n$ maximizes $G$, we have $G(\rho_n,x_n)\ge G(\rhot,x_n)$ for any choice of $n\ge 0$. Then for all $n\ge N$,
\begin{align*}
    \frac{\hat c_2}{2} \int \norm{z}^2 \d \rho_n(z) &\le C - G(\rhot,x_n) 
    \le C +\max_{x} (-G(\rhot,x))\,.
\end{align*}
The right-hand side is finite and bounded by some constant independent of $n$ since $G(\rhot,\cdot)$ is twice continuously differentiable and strongly $\lambda_d$-convex over $\R^d$ (see proof of Lemma~\ref{lem:first_variation_b}).
\end{proof}

\begin{proposition}[Convergence of Second Moments]\label{prop:cv_second_moments}
    Let Assumptions
    \ref{assump:f_upper}-\ref{assump:W_upper} hold. If $x_n\rightarrow \bar x$ in $\R^d$, then the sequence $(\rho_n)$ as defined in \eqref{eq:cv-defs} satisfies $\int \norm{z}^2 \d\rho_n(z) \rightarrow \int \norm{z}^2 \d \rhob(z)$.
\end{proposition}
\begin{proof}
Note that $(\rho_n)\in\P_2^{ac}(\R^d)$ by Proposition~\ref{lem:existence_of_maximizer}. We will show convergence of the second moments by showing that $(\norm{z}^2\rho_n(z))$ is equi-integrable, that is,
\begin{align*}
    \lim_{K\rightarrow \infty} \sup_n \int \norm{z}^2 \rho_n(z) \mathbbm 1 \{ z : \norm{z}^2 \rho_n(z) \ge K \}\,\d z = 0\,.
\end{align*}
    Since $\rho_n$ is a minimizer of $F(\cdot,x_n)$, $\rho_n$ satisfies the re-arranged EL equation \eqref{eq:EL-fastrho}
    \begin{align*}
        \rho_n (z)= \tilde c(x_n) \exp\left(-\frac{1}{\alpha} \m(z,x_n) - \frac{1}{\alpha}W\ast\rho_n \right) \le \tilde c(x_n) \exp\left(- \frac{1}{\alpha}\m(z,x_n)  \right)\,,
    \end{align*}
    where we used that $\exp(- W\ast \rho) \le 1$ since $W \ast \rho\ge 0$. Using (1) from Lemma~\ref{lem:upper_lower_bounds_f} which provides a lower bound on $\m$, we obtain
    \begin{align*}
        \rho_n(z) \le \tilde c(x_n) \exp\left(\frac{\hat c_1 - \hat c_2 \norm{z}^2}{\alpha}\right)
    \end{align*}
    with $\hat c_2>0$ since $\lambda_b>0$.
    In order to obtain an upper-bound for $\rho_n$ independent of $x_n$, we show that $\sup_n \tilde c(x_n) < \infty$. Starting again from the Euler-Lagrange equation, using that $\int \rho_n=1$
    \begin{align*}
        &\tilde c(x_n) \int \exp\left(-\frac{1}{\alpha}\m(z,x_n)-\frac{1}{\alpha}W\ast \rho_n\right) \d z =1 \\ &\quad \Rightarrow  \ \tilde c(x_n) = \left( \int \exp(-\frac{1}{\alpha}\m(z,x_n)-\frac{1}{\alpha}W\ast\rho_n)\d z\right)^{-1}\,.
    \end{align*}
    Next, we upper-bound $\tilde c(x_n)$. Using the upper estimate (2) for $\m(z,x_n)$ from Lemma~\ref{lem:upper_lower_bounds_f}, which requires Assumptions~\ref{assump:f_upper} and \ref{assump:V_upper}, we obtain $\m(z,x_n)\le c_1+c_2\|z\|^2$ with $c_2>0$.
    To control the term with $W \ast \rho_n$, we use Assumption~\ref{assump:W_upper} to obtain a growth condition for $W$,
    \begin{align*}
        W(z) &= W(0) + \nabla W(0) \cdot z + \int_0^1 z^\top \nabla^2 W(sz) z \d s \le W(0) + \norm{\nabla W(0)}\norm{z} + \Lambda_{W} \norm{z}^2 \\
        &\le \tilde c_1 + \tilde c_2 \norm{z}^2 \,,
    \end{align*}
    and so
    \begin{align*}
        W\ast \rho_n(z) &= \int W(z-z')\rho_n(z') \d z' \le \tilde c_1 + \tilde c_2 \int \norm{z-z'}^2 \d \rho_n(z')\\
        &\le \tilde c_1 + 2\tilde c_2\int \norm{z'}^2 \d \rho_n(z') + 2 \tilde c_2 \norm{z}^2 \le \tilde C_1 + \tilde C_2 \norm{z}^2
    \end{align*}
    with $\tilde C_1, \tilde C_2>0$,
    using that the second moments of $\rho_n$ are uniformly bounded by Proposition~\ref{prop:rho_n_bdd_moment}. Therefore,
    \begin{align*}
         \tilde c(x_n) \le\left(\int\exp\left(-\frac{1}{\alpha}\left[c_1+ \tilde C_1+(c_2+\tilde C_2)\|z\|^2\right]\right)\d z\right)^{-1}=:\tilde c_0
    \end{align*} 
    The upper-bound on $\rho_n$ is therefore
    \begin{align}\label{eq:rho_n_upper_bd}
        \rho_n(z) \le \tilde c_0 \exp\left (\hat c_1 - \hat c_2\norm{z}^2\right) \quad \text{for a.e. } z\in\R^d\,, 
    \end{align}
    which shows that $\norm{z}^2\rho_n$ is equi-integrable since the Gaussian-like shape of the upper bound in \eqref{eq:rho_n_upper_bd} has a finite second moment. Applying \cite[Theorem 4.5.4]{bogachev_measure_2007} results in the convergence of the second moment.
\end{proof}

\begin{corollary}\label{cor:br-cont}
 Let Assumptions~\ref{assump:f_upper}-\ref{assump:W_upper} hold. The best response $\br(x)\in \P_2(\R^d)$ is $\Wass_2$-continuous in $x\in\R^d$, and $G_d:\R^d \to \R$ is continuous in $\R^d$.
\end{corollary}
\begin{proof}
    Given any sequence $(x_n)_{n\in\nat}$ converging to $\bar x$ in $\R^d$ as $n\to \infty$, we have that $F_n \xrightarrow{\Gamma} \bar F$ in $\Wass_2$ for sequences such that $(\mu_n)\in\P_2^{ac}(\R^d)$ by Proposition~\ref{prop:gamma_cv}. 
    From Lemma~\ref{lem:existence_of_maximizer}, $\br(x_n)=\argmin_{\rho\in\P_2} F_n(\rho)$ exists and is unique for every $x_n$; similarly, $\br(\bar x)$ is the unique minimizer of $\bar F$. Uniform boundedness of second moments for $(\br(x_n))_{n\in\nat})$ follows from Proposition~\ref{prop:rho_n_bdd_moment},
    and tightness of $(\br(x_n))_{n\in\nat}$ follows using \cite[Proposition 7.1.5]{ags}. Then precompactness of $(\br(x_n))_{n\in\nat}$ in the narrow topology follows from \cite[Prokorov's Theorem 5.1.3]{ags}. Together with convergence of second moments (Proposition~\ref{prop:cv_second_moments}), we conclude that $\br(x_n)$ converges to $\br(\bar x)$ in $\Wass_2$ \cite[Theorem 2.10]{braides_handbook_2006}. In other words, the best response $\br(x)$ is continuous in $x$ in $\Wass_2$. 

    Recall that $G_d(x) \coloneqq \max_{\rho\in\P} G(\rho,x)$. For any fixed $\rho$, $G(x,\rho)$ is continuous in $x$ due to our assumptions. Because the maximum over continuous functions is lower-semicontinuous, $G_d(x)$ is lower-semicontinuous in $x$ with respect to $\R^d$ . From Proposition~\ref{prop:gamma_cv}, it holds that $F_n$ is lower-semicontinuous:
    \begin{align*}
        \liminf_{n\rightarrow\infty} F_n(\mu_n) \ge \bar F(\bar \mu)\,.
    \end{align*}
    Since $G_d(x_n) = G(\rho_n,x_n) = \frac{\kappa}{2}\norm{x-x_0}^2 + \int f_2(z,x) \d \pi(z) - F_n(\rho_n)$, $F_n$ is lower-semicontinuous, and all other terms in $G_d$ are continuous in $x_n$, we have
    \begin{align*}
        \limsup_{n\rightarrow \infty} G_d(x_n) = -\liminf_{n \rightarrow \infty } F_n(\rho_n) + \frac{\kappa}{2}\norm{\bar x-x_0}^2 + \int f_2(z,\bar x) \d \pi(z) \le G_d (\bar x)\,,
    \end{align*}    
    which means that $G_d$ is upper-semicontinuous in $x$ with respect to $\R^d$. Since both lower- and upper-semicontinuity have been shown, $G_d$ is continuous in $x$.
\end{proof}

\begin{lemma}\label{lem:limsup_partial_f1_dxi}
   Let Assumptions~\ref{assump:f_upper}-\ref{assump:W_upper} hold. Fix $i\in\{1,\cdots,d\}$ and let $x^{(h)}=\bar x +h e^{(i)}$ and $\rho_h = \amax_{\rho\in\P} G(\rho,x^{(h)})$, where $e^{(i)}$ denotes the $i$th standard unit vector in $\R^d$. Then for any sequence $y^{(h)}\rightarrow \bar x$ as $h\to 0$, we have
    \begin{align}\label{eq:lim_partial_G}
        \lim_{h \rightarrow 0} \partial_{x_i} G(\rho_h,y^{(h)}) = \partial_{x_i} G(\rhob,\bar x)\,, 
    \end{align}
    with $x_i$ denoting the $i$th coordinate of $x\in\R^d$ and $\bar\rho = \amax_{\rho\in\P} G(\rho,\bar x)$. Moreover for any sequence $x_n\rightarrow \bar x$,
    \begin{align}\label{eq:lim_f1_br}
        \lim_{n\rightarrow\infty} \int f_1(z,x_n)\d \br(x_n)(z) = \int f_1(z,\bar x) \d\br(\bar x)(z)\,.
    \end{align}
\end{lemma}
\begin{proof}
To show that \eqref{eq:lim_partial_G} holds, we want to compute
    \begin{align*}
        \lim_{h\rightarrow 0}& \left( \int \partial_{x_i} f_1(z,y^{(h)}) \d \rho_h + \int \partial_{x_i} f_2(z,y^{(h)}) \d \pi(z) + \kappa (y_i^{(h)}-(x_0)_i) \right) \\
        &=
       \lim_{h \rightarrow 0} \left(\int \partial_{x_i} f_1(z,y^{(h)}) \d \rho_h(z) \right)+ \int \partial_{x_i} f_2(z,\bar x) \d \pi(z) + \kappa (\bar x_i-(x_0)_i)\,,
    \end{align*}
    where the limiting value of the last two terms follows immediately from the assumption that $f_2\in C^2$.
    From a corollary of Fatou's lemma for weakly converging measures \cite[Corollary 2.8]{feinberg_fatous_2019}, it is sufficient to show that $(\partial_{x_i} f_1(z,y^{(h)}))_{h\ge 0}$ is a.u.i with respect to $(\rho_h)_{h\ge 0}$. 
    Using the Taylor expansion around $z=0$ with remainder gives
       \begin{align*}
         \partial_{x_i} f_1(z,y^{(h)}) 
        &= \partial_{x_i} f_1(0,y^{(h)}) + \int_0^1 \nabla_z \partial_{x_i} f_1(tz,y^{(h)}) \cdot z\, \d t \,.
    \end{align*}
    From Assumption~\ref{assump:f_upper}(b), we obtain
    \begin{align*}
       \left| \int_0^1 \nabla_z \partial_{x_i} f_1(t z,y^{(h)})\cdot z \,\d t \right| \le L  \norm{z}\,,
    \end{align*}
 and since $\pxi f_1$ is continuous in $x$, we have
 \begin{align*}
     |g_h(z)| \le L\norm{z} + C_0
 \end{align*}
 for $g_h(z)\coloneqq\pxi f_1(z,y^{(h)})$ and for some constant $C_0>0$ depending only on $\bar x$.
For $K>C_0$, the key term for the a.u.i condition can be bounded by
\begin{align*}
    \int |g_h(z)| \mathbbm{1}\{z\, :\, |g_h(z)|\ge K\} \d \rho_h(z) \le \int_{B_{\frac{K-C_0}{L}}^c} (L\norm{z} + C_0)\d \rho_h(z)\,,
\end{align*}
where $B_{\frac{K-C_0}{L}}$ denotes the open ball of radius $\frac{K-C_0}{L}$ centered at zero. Taking the limit as $h\rightarrow 0$ results in
\begin{align*}
    \lim_{h\rightarrow 0} \int |g_h(z)| \mathbbm{1}\{z\, :\, |g_h(z)|\ge K\} \d \rho_h(z) &\le \lim_{h\rightarrow 0} \int_{B_{\frac{K-C_0}{L}}^c} (L\norm{z} + C_0)\d \rho_h(z) \\&=  \int_{B_{\frac{K-C_0}{L}}^c} (L\norm{z}+C_0) \d \rhob(z)\,,
\end{align*}
    where convergence to $\rhob$ is due to $\rho_h \rightarrow \rhob$ in $\Wass_2$ by Corollary~\ref{cor:br-cont}. Now taking $K$ to infinity,
    \begin{align*}
        \lim_{K\rightarrow \infty} \lim_{h\rightarrow 0} &\int |g_h(z)| \mathbbm{1}\{z\, :\, |g_h(z)|\ge K\} \d \rho_h(z) \\ &\le \lim_{K\rightarrow \infty} L\int_{B_{\frac{K-C_0}{L}}^c}\norm{z} \rhob(z)\d z + \lim_{K\rightarrow \infty} C_0 \int_{B_{\frac{K-C_0
        }{L}}^c} \rhob(z)  = 0\,,
    \end{align*}
    since $\int \|z\|\rhob(z)\d z <\infty$ and $\int \rhob(z)\d z = 1$.
    Therefore $\partial_{x_i} f_1(z,y^{(h)})$ is a.u.i. with respect to $(\rho_h)$. We conclude that $$\lim_{h\rightarrow 0} \int g_h(z) \d \rho_h(z) = \int \left(\lim_{h\rightarrow 0 }g_h(z) \right)\d \rhob(z) =\int \pxi f_1(z,\bar x) \d \rhob(x)$$ by applying Corollary~\cite[Corollary 2.8]{feinberg_fatous_2019} and the limit $\lim_{h \rightarrow 0} \partial_{x_i} f_1(z,y^{(h)})$ exists because $f_1$ is $C^2$ in $x$.

    To show that \eqref{eq:lim_f1_br} holds, it is sufficient to show that $(f_1(z,x_n))_{n\in\nat}$ is a.u.i with respect to $(\rho_n)$. Note that $\rho_h$ denotes the best response to $x^{(h)}=\bar x + h e_i$, whereas $\rho_n$ is the best response to \textit{any} sequence $x_n\rightarrow \bar x$. Again using the Taylor expansion around $z=0$,
    \begin{align*}
        f_1(z,x_n) & = f_1(0,x_n) + \nabla_z f_1(0,x_n) \cdot z + \int_0^1 z^\top \nabla^2_z f_1(tz,x_n) z \,\d t\,.
    \end{align*}
    Using that $f_1(0,\cdot)\in C^2$ and applying Assumptions~\ref{assump:f_lower}(ii) and \ref{assump:f_upper}(a) results in
    \begin{align*}
        |f_1(z,x_n)| \le c_1 + c_2\norm{z}^2\,, \quad \forall \ z\in\R^d\,,\, n\in\nat\,,
    \end{align*}
    for some $c_1,c_2\in\R_+$.
    The a.u.i condition, with $\tilde B^c\coloneqq B^c_{\sqrt{\frac{K-c_1}{c_2}}}$ for $K>c_1$, is given by
    \begin{align*}
        \int |f_1(z,x_n)|\mathbbm{1}\{z\, :\, |f_1(z,x_n)| \ge K\}\d \rho_n(z) &\le \int |f_1(z,x_n)|\mathbbm{1}\{z\,:\, c_1 + c_2\|z\|^2 \ge K\} \d \rho_n(z) \\
        & \le  \int_{\tilde B^c} (c_1 + c_2\|z\|^2) \d \rho_n(z)\,.
    \end{align*}
    Taking the limit as $n \rightarrow\infty$ using the convergence of second moments given by Proposition~\ref{prop:cv_second_moments} results in
    \begin{align*}
        \lim_{n\rightarrow \infty} \left( c_1\int_{\tilde B^c} \d \rho_n(z) + c_2 \int_{\tilde B^c} \|z\|^2 \d \rho_n(z) \right) = c_1 \int_{\tilde B^c}\d \rhob(z) + c_2 \int_{\tilde B^c} \|z\|^2 \d \rhob(z)\,.
    \end{align*}
    Now taking the limit as $K$ goes to $\infty$,
    \begin{align*}
        \lim_{K\rightarrow\infty} \left( c_1 \int_{\tilde B^c}\d \rhob(z) + c_2 \int_{\tilde B^c} \|z\|^2 \d \rhob(z)\right) = 0\,,
    \end{align*}
    since $\int \rhob(z)=1$ and $\int \|z\|^2 \d \rhob(z)<\infty$. Therefore $f_1$ is a.u.i with respect to $\rho_n$, and applying \cite[Corollary 2.8]{feinberg_fatous_2019} results in
    \begin{align*}
        \lim_{n\rightarrow\infty} \int f_1(z,x_n) \d \br(x_n)(z) =\int f_1(z,\bar x) \d \br(\bar x)(z)\,.
    \end{align*}
    \end{proof}

\begin{proposition}[Version of Danskin's Theorem]\label{prop:danskin-br}
    Let Assumptions~\ref{assump:f_upper}-\ref{assump:W_upper} hold. Let $\br(x)$ be as defined in \eqref{eq:dynamics_competitive_rho_fast}. Then $G_d\in C^1(\R^d)$ and
    \begin{align*}
        \nabla_x G_d(x) = \left.\left(\nabla_x G(\rho,x)\right)\right|_{\rho=\br(x)}\,.
    \end{align*}
\end{proposition}
\begin{proof}
    We will show that $G_d(x)$ is differentiable by showing that component-wise left and right derivatives coincide. More precisely, for any function $g:\R^d \to \R$ (not necessarily differentiable), we denote
    \begin{align*}
        \pxi^{\sup +}g(x)=\limsup_{h \downarrow 0}\frac{g(x+h e^{(i)})-g(x)}{h}\,, \qquad \pxi^{\inf +}g(x)=\liminf_{h \downarrow 0}\frac{g(x+h e^{(i)})-g(x)}{h}\,, \\
        \pxi^{\sup -}g(x)=\limsup_{h \uparrow 0}\frac{g(x+h e^{(i)})-g(x)}{h}\,,  \qquad \pxi^{\inf -}g(x)=\liminf_{h \uparrow 0}\frac{g(x+h e^{(i)})-g(x)}{h}\,.
    \end{align*}
    Fix $\tilde x \in \R^d$ and recall that $G_d(\tilde x)=G(\br(\tilde x),\tilde x)$. We begin with the $\liminf$ for $G$. Because $r(\tilde x)$ cannot make $G(\cdot,x)$ larger than $G(r(x),x)$ by definition of the best response, we have that  
    \begin{align*}
        G_d(x) \ge G(r(\tilde x),x) \quad \forall \, x\in\R^d\,.
    \end{align*}
    Then for any $h > 0$ and standard normal vector $e^{(i)}\in\R^d$,
    \begin{align*}
        \frac{G_d(\tilde x+he^{(i)})-G_d(\tilde x)}{h} \ge \frac{G(r(\tilde x),\tilde x+h e^{(i)}) - G(r(\tilde x),\tilde x)}{h}\,.
    \end{align*}
    Taking the $\liminf$ in $h>0$ on both sides, we have
    \begin{align*}
        \pxi^{\inf +} G_d(\tilde x) = \liminf_{h \downarrow 0} \frac{G_d(\tilde x+h e^{(i)}) - G_d(\tilde x)}{h} \ge \pxi^{\inf +} G(
        \rho,\tilde x)|_{\rho=r(\tilde x)}= \pxi G(\rho,\tilde x)|_{\rho=r(\tilde x)}\,,
    \end{align*}
    since $G(\rho,\cdot)\in C^2(\R^d)$ for any $\rho\in\P(\R^d)$.
    Likewise, for $h< 0$, 
    \begin{align*}
        \frac{G_d(\tilde x+he^{(i)})-G_d(\tilde x)}{h} \le \frac{G(r(\tilde x),\tilde x+h e^{(i)}) - G(r(\tilde x),\tilde x)}{h}\,,
    \end{align*}
    which, after taking $\limsup$ on both sides, results in
    \begin{align*}
        \pxi^{\sup -}G_d(\tilde x) \le  \pxi G(\rho,\tilde x) |_{\rho=r(\tilde x)}\,.
    \end{align*}
    Next, consider the inequality
    \begin{align*}
        G_d(\tilde x) \ge G(r(x),\tilde x) \quad \forall x\in\R^d
    \end{align*}
    and again setting $x=\tilde x + h e^{(i)}$, we have for any $h>0$, 
    \begin{align}\label{eq:lower_bound_Gd}
        \frac{G_d(\tilde x+h e^{(i)}) - G_d(\tilde x)}{h} \le \frac{G(r(\tilde x+h e^{(i)}),\tilde x + h e^{(i)}) - G(r(\tilde x+h e^{(i)}),\tilde x)}{h}\,.
    \end{align}
    Our goal is to apply Lemma~\ref{lem:limsup_partial_f1_dxi} so that we can take the $\limsup$ on the right-hand side.
     By the mean value theorem,
    \begin{align}\label{eq:taylor_expansion_Gc}
        G(\rho,\tilde x+he^{(i)}) = G(\rho,\tilde x) + h \pxi G(\rho,\xi^{(h)}) 
    \end{align}
    for some $\xi^{(h)} \in [\tilde x,\tilde x+h e^{(i)}]$ for any $\rho\in\P_2(\R^d)$. Taking the $\limsup$ on both sides of \eqref{eq:lower_bound_Gd} and using \eqref{eq:taylor_expansion_Gc}, we have
    \begin{align*}
    \pxi^{\sup +} G_d(\tilde x) &= \limsup_{h \downarrow 0} \frac{G_d(\tilde x+h e^{(i)}) - G_d(\tilde x)}{h}  \\
        &\leq \limsup_{h \downarrow 0} \frac{G(r(\tilde x+h e^{(i)}),\tilde x+h e^{(i)}) - G(r(\tilde x+he^{(i)}),\tilde x)}{h} \\
        &= \limsup_{h \downarrow 0} \pxi G(r(\tilde x+h e^{(i)}),\xi^{(h)})  = \pxi G(\rho,\tilde x)|_{\rho=r(\tilde x)}\,,
    \end{align*}
    where the last line follows from \eqref{eq:lim_partial_G} in Lemma~\ref{lem:limsup_partial_f1_dxi}.
    Now taking $h< 0$, we have
    \begin{align*}
        \frac{G_d(\tilde x+h e^{(i)}) - G_d(\tilde x)}{h} \ge \frac{G(r(\tilde x+h e^{(i)}),\tilde x + h e^{(i)}) - G(r(\tilde x+h e^{(i)}),\tilde x)}{h}\,.
    \end{align*}
    which results in, after taking the $\liminf$ and using the continuity of $r(\cdot)$ from Corollary~\ref{cor:br-cont},
    \begin{align*}
        \pxi^{\inf -}G_d(\tilde x) \ge \pxi G(\rho,\tilde x)|_{\rho=r(\tilde x)}\,.
    \end{align*}
Collecting inequalities, we have shown
    \begin{align*}
        \pxi^{\inf +} G_d(\tilde x) &\ge  \pxi G(\rho,\tilde x)|_{\rho=r(\tilde x)}\,,   &\pxi^{\sup -}G_d(\tilde x) &\le  \pxi G(\rho,\tilde x) |_{\rho=r(\tilde x)}\,, \\
        \pxi^{\sup +}G_d(\tilde x) &\le \pxi G(\rho,\tilde x)|_{\rho=r(\tilde x)}\,,  &\pxi^{\inf -}G_d(\tilde x) &\ge \pxi G(\rho,\tilde x)|_{\rho=r(\tilde x)\,.}
    \end{align*}
    Chaining the inequalities together, we obtain
    \begin{align*}
        \pxi G(\rho,\tilde x)|_{\rho=r(\tilde x)} &\leq \pxi^{\inf +} G_d(\tilde x) \le \pxi^{\sup +} G_d(\tilde x) \le \pxi G(\rho,\tilde x)|_{\rho=r(\tilde x)}\,, \\
        \pxi G(\rho,\tilde x)|_{\rho=r(\tilde x)} &\leq \pxi^{\inf -} G_d(\tilde x) \le \pxi^{\sup -} G_d(\tilde x) \le \pxi G(\rho,\tilde x)|_{\rho=r(\tilde x)}\,,
    \end{align*}
    and therefore $\pxi^{\inf\pm} G_d(\tilde x) = \pxi^{\sup\pm} G_d(\tilde x) = \pxi G_d(\tilde x)$ for any $\tilde x$ and so all partial derivatives of $G_d$ exist at any $x\in\R^d$ with partial derivative given by $\pxi G(\rho, x)|_{\rho=\br( x)}$. If $x_n \rightarrow \bar x$, then $\br(x_n) \rightarrow \br(\bar x)$ in $\Wass_2$ and $G_d(x_n) \rightarrow G_d(\bar x)$ by Corollary~\ref{cor:br-cont}.
    Further, from the expression for $\partial_{x_i} G(\rho,x)$ it is clear that $\partial_{x_i} G_d(x)$ is continuous for all $i$, and so we conclude that $G_d\in C^1(\R^d)$ and $\nabla_x G(\rho,x)|_{\rho=r(x)} = \nabla_x G_d( x)$.
\end{proof}

\begin{proof}[Proof of Theorem \ref{thm:convergence_fast_rho}]
As the energy $G(\rho,x)$ is strongly $\lambda_d$-convex in $x$ for each $\rho$, the energy $G_d(x)=\max_{\rho\in\P(\R^d)} G(\rho,x)$ is also strongly convex as a supremum of strongly convex functions. It follows that $G_d$ is coercive, and has a unique minimizer $x_\infty\in\R^d$. By Proposition~\ref{prop:danskin-br}, $G_d\in C^1(\R^d)$.
Convergence in norm then immediately follows from strong convexity of $G_d$: for solutions $x(t)$ to \eqref{eq:dynamics_competitive_rho_fast}, we have
\begin{align*}
    \frac{1}{2}\frac{\d}{\d t} \|x(t)-x_\infty\|^2 
    = -\nabla_x\left(G_d(x(t))-G_d(x_\infty)\right)\cdot (x(t)-x_\infty)
    \le -\lambda_d \|x(t)-x_\infty\|^2\,.
\end{align*}
A similar result holds for convergence in entropy using the Poly{\'a}k-{\L}ojasiewicz convexity inequality
\begin{align*}
    \frac{1}{2}\norm{\nabla G_d(x)}_2^2 \geq \lambda_d (G_d(x)-G_d(x_\infty))\,,
\end{align*}
which is itself a direct consequence of strong convexity of $G_d$. Then
\begin{align*}
    \frac{\d}{\d t} \left(G_d(x(t))-G_d(x_\infty)\right)
    &= \nabla_x G_d(x(t))\cdot \dot{x}(t) \\
    &= -\|\nabla_x G_d(x(t))\|^2
    \leq -2\lambda_d \left(G_d(x(t))-G_d(x_\infty)\right)\,,
\end{align*}
and so the result in Theorem~\ref{thm:convergence_fast_rho} follows.
\end{proof}

\section{Auxiliary Lemmas}\label{sec:extra_lemmas}

\begin{lemma}\label{lem:Wbar_upper_bd}
Let Assumptions \ref{assump:f_lower}(ii), \ref{assump:V_lower}, and \ref{assump:W_lower} hold with $\lambda_c\coloneqq \min \{\lambda_{f,1}+\lambda_{V,1},\lambda_{f,2}+\lambda_{V,2} \}>0$. 
    For any $\gamma,\tilde\gamma\in\P_2\times\P_2$ for which the left-hand side below is well-defined, the functional $F_c$ satisfies 
    \begin{align*}
        \iint \begin{bmatrix}
              z - \nabla \varphi(z) \\
              x - \nabla \psi(x)
          \end{bmatrix} \cdot \begin{bmatrix}
              \nabla_z\delta_\rho F_c[\tilde \rho,\tilde \mu](\nabla \varphi(z) )-\nabla_z \delta_\rho F_c[\rho,\mu](z)  \\
              \nabla_x \delta_\mu F_c[\rho,\mu](x) - \nabla_x \delta_\mu F_c[\tilde \rho,\tilde \mu](\nabla \psi(x))
          \end{bmatrix} \d \rho(z) \d \mu(x) \ge \lambda_c\Wbar(\gamma,\tilde \gamma)^2 \,,
    \end{align*}
    where $(\varphi, \psi)$ are optimal transport maps such that $\rhot = \nabla \varphi_\# \rho$ and $\tilde \mu = \nabla \psi_\# \mu$. 
\end{lemma}
\begin{proof}
We break the left-hand side of the inequality into four parts; one with the diffusion terms, one with the coupling potential term, one with the convolution terms, and one with the external potentials:
\begin{align*}
    C_1(\gamma,\tilde \gamma) = &-\alpha\int (z-\nabla\varphi(z))\cdot(\nabla \log \rhot(\nabla \varphi(z)) - \nabla \log \rho(z)) \d \rho(z) \\
    &-\beta \int (x-\nabla \psi(x))\cdot(\nabla \log \tilde \mu(\nabla \psi(x))-\nabla \log \mu(x)) \d \mu(x), \\
    C_2(\gamma,\tilde \gamma) = & \iint (z-\nabla\varphi(z))\cdot(\nabla_1 f(\nabla \varphi(z),\nabla \psi(x)) - \nabla_1 f(z,x)) \d \gamma(z,x) \\
    &-\iint (x-\nabla \psi(x))\cdot(\nabla_2 f(\nabla \varphi(z),\nabla \psi(x)) - \nabla_2 f(z,x) ) \d \gamma(z,x) \\
    C_3(\gamma,\tilde\gamma) =& -\int (z-\nabla\varphi(z))\cdot[(\nabla W_1 \ast \rhot)(\nabla \varphi(z)) -(\nabla W_1 \ast \rho)(z)] \d \rho(z) \\
    &- \int (x-\nabla \psi(x))\cdot[(\nabla W_2 \ast \tilde \mu)(\nabla \psi (x)) -(\nabla W_2 \ast \mu) (x)]\d \mu(x)\,,
\end{align*}
and
\begin{align*}
    C_4(\gamma,\tilde\gamma)=&- \int (z-\nabla\varphi(z))\cdot(\nabla V_1(\nabla\varphi(z))-\nabla V_1(z)) \d \rho(z), \\
    &-\int (x-\nabla\psi(x))\cdot(\nabla V_2(\nabla \psi(x))-\nabla V_2(x))\d \mu(x)\,.
\end{align*}
where $\nabla_i$ is the gradient operator with respect to the $i^{th}$ argument. With the above definitions, it holds that 
\begin{align*}
    \sum_{i=1}^4 C_i(\gamma,\tilde \gamma) = 
    \iint \begin{bmatrix}
              z - \nabla \varphi(z) \\
              x - \nabla \psi(x)
          \end{bmatrix} \cdot \begin{bmatrix}
              \nabla_z\delta_\rho F_c[\tilde \rho,\tilde \mu](\nabla \varphi(z) )-\nabla_z \delta_\rho F_c[\rho,\mu](z)  \\
              \nabla_x \delta_\mu F_c[\rho,\mu](x) - \nabla_x \delta_\mu F_c[\tilde \rho,\tilde \mu](\nabla \psi(x))
          \end{bmatrix} \d \gamma(z,x)\,.
\end{align*}
We claim that $C_1 \ge 0 $ for all $\gamma,\tilde \gamma\in P^{ac}_2 \times \P^{ac}_2$\footnote{Note that absolute continuity is required for the left-hand side in Lemma~\ref{lem:Wbar_upper_bd} to be well defined if $\alpha,\beta>0$.}. We prove this claim  for the $\rho$-dependent terms in $C_1$, and the $\mu$-dependent term follows similarly. Let
\begin{align*}
    c_1 = \int \nabla\varphi(z) \cdot (\nabla \log \rhot(\nabla \varphi(z)) - \nabla \log \rho(z)) \d \rho(z)\,.
\end{align*}
We  use the pushforward maps to write
\begin{align*}
    c_1 &= \int \nabla \varphi(z) \cdot \frac{\nabla \rhot(\nabla \varphi(z))}{\rhot(\nabla \varphi(z))} \d \rho(z) - \int \nabla \varphi(z) \cdot \frac{\nabla \rho(z)}{\rho(z)} \d \rho(z) \\
    &= \int z \cdot \frac{\nabla \rhot(z)}{\rhot(z)} \d \rhot(z) - \int \nabla\varphi(z) \cdot \nabla \rho(z) \d z \\
    &= \int z \cdot \nabla\rhot(z) \d z + \int \Delta \varphi (z)  \rho \d z
     = -d+\int \Delta \varphi(z) \d \rho(z)\,,
\end{align*}
Let 
\begin{align*}
    c_2 = -\int z \cdot (\nabla \log \rhot(\nabla \varphi(z)) - \nabla \log \rho(z)) \d \rho(z)\,.
\end{align*}
We use the pushforward maps and the fact that the convex dual $\varphi^*$ of $\varphi$ satisfies $\nabla \varphi^* = (\nabla \varphi)^{-1}$ to write
\begin{align*}
    c_2& = -\int \nabla \varphi^{*}(\nabla\varphi(z)) \cdot \nabla  \log \rhot(\nabla \varphi(z))\d \rho(z) + \int z \cdot \frac{\nabla \rho(z)}{\rho(z)} \rho(z) \d z \\
    &= \int \Delta \varphi^{*}(z) \d \rhot(z) - \int \nabla \cdot z \d \rho(z) = \int \Delta \varphi^{*}(\nabla \varphi(z)) \d \rho(z) - d
\end{align*}
The first integral in the definition of $C_1(\gamma,\tilde\gamma)$ is then given by
\begin{align*}
    \alpha(c_1 + c_2) = \alpha\left(-2d + \int (\Delta \varphi(z) + \Delta \varphi^{*}(\nabla \varphi(z))) \d \rho(z)\right)\,.
\end{align*}
Since $\varphi$ is strictly convex, $\nabla^2 \varphi(z) \succ0$ for all $z\in \R^d$ and the eigenvalues $\lambda_i$ of $\nabla^2 \varphi(z)$ are strictly positive. Therefore, using $\nabla^2 \varphi^{*}(\nabla \varphi(z)) = (\nabla^2 \varphi(z))^{-1}$, we have
\begin{align*}
   \alpha( c_1 + c_2) = \alpha\left(-2d + \int \text{Tr}\left[\nabla^2 \varphi(z) + (\nabla^2 \varphi(z))^{-1}\right] \d \rho(z)\right) =\alpha \sum_{i=1}^d (\lambda_i + \lambda_i^{-1}) \ge \alpha(-2d + 2d) = 0\,.
\end{align*}
The last inequality follows from $\lambda+1/\lambda \ge 2$ for all $\lambda>0$. Therefore $C_1 \ge 0$. 
Next, we expand the expression for $C_2(\gamma,\tilde \gamma)$.
We show the exact form of the Taylor expansion of $\nabla_1 f$ and the expansion for $\nabla_2 f$ follows similarly. Define the stacked variable $y=[z,x]$ and let $y_s\coloneqq (1-s)[z,x] + s[\nabla \varphi(z),\nabla \psi(x)]$. Computing the Taylor expansion with respect to $y$ results in 
\begin{align*}
    \nabla_1 f(\nabla \varphi(z),\nabla \psi(x)) &= \nabla_1 f(z,x) + \int_0^1 \nabla_y \nabla_1 f(y_s) \cdot \begin{bmatrix}
        \nabla \varphi(z)-z \\ \nabla \psi(x)-x
    \end{bmatrix} \d s \\
    & = \nabla_1 f(z,x) - \int_0^1 \begin{bmatrix}
        \nabla_1^2 f(x_s,z_s) \\ \nabla_{12}^2 f(x_s,z_s)
    \end{bmatrix}\cdot \begin{bmatrix}
        z-\nabla \varphi(z) \\ x-\nabla \psi(x)
    \end{bmatrix} \d s \,.
\end{align*}
Plugging this expansion into $C_2$ results in
\begin{align*}
    C_2(\gamma,\tilde \gamma) = &  \int_0^1 \iint \begin{bmatrix}
        z-\nabla\varphi(z) \\ x-\nabla \psi(x)
    \end{bmatrix}^\top \cdot
    \begin{bmatrix}
        -\nabla_{1}^2 f(y_s) & -\nabla_{12}^2 f(y_s) \\
        \nabla_{12}^2 f(y_s) & \nabla_{2}^2 f(y_s)
    \end{bmatrix} \cdot
    \begin{bmatrix}
        z-\nabla\varphi(z) \\ x-\nabla \psi(x)
    \end{bmatrix} \d \gamma(z,x) \d s \,.
\end{align*}
Since $C_2$ is a scalar, we  use that $C_2(\gamma,\tilde\gamma)=\frac{1}{2} C_2(\gamma,\tilde \gamma)+\frac{1}{2}C_2(\gamma,\tilde\gamma)^\top$, giving
\begin{align*}
    C_2(\gamma,\tilde \gamma) = &  \int_0^1 \iint \begin{bmatrix}
        z-\nabla\varphi(z) \\ x-\nabla \psi(x)
    \end{bmatrix}^\top \cdot
    \begin{bmatrix}
        -\nabla_{1}^2 f(y_s) & 0 \\
        0 & \nabla_{2}^2 f(y_s)
    \end{bmatrix} \cdot
    \begin{bmatrix}
        z-\nabla\varphi(z) \\ x-\nabla \psi(x)
    \end{bmatrix} \d \gamma(z,x) \d s \,.
\end{align*}
By Assumption~\ref{assump:f_lower}(ii), we use that $-\nabla_z^2 f(z,x) \succeq \lambda_{f,1} \Id_{d_1}$ and $\nabla_x^2 f(z,x) \succeq \lambda_{f,2} \Id_{d_2}$ to obtain
\begin{align*}
    C_2(\gamma,\tilde\gamma) \ge \lambda_{f,1} \Wass_2(\rho,\rhot)^2 + \lambda_{f,2}\Wass_2(\mu,\tilde \mu)^2\,.
\end{align*}
Next, we show a lower-bound for $C_3$. We show the exact Taylor expansion calculation explicitly for the first term:
\begin{align*}
    (\nabla W_1 \ast \rhot)(\nabla \varphi(z)) -(\nabla W_1 \ast \rho)(z) &= \int \nabla W_1(\nabla \varphi(z)-z') \d \rhot(z') - \int \nabla W_1(z-z') \d \rho(z')\\
    &= \int [\nabla W_1(\nabla \varphi(z) -\nabla \varphi(z')) - \nabla W_1(z-z')]\d \rho(z')\,.
\end{align*}
Define $z_s\coloneqq (1-s)(z-z') + s (\nabla \varphi(z) - \nabla \varphi(z'))$. The expansion of $\nabla W_1$ around $z-z'$ is
\begin{align*}
    \nabla W_1(\nabla \varphi(z)-\nabla \varphi(z')) = \nabla W_1(z-z') + \int \nabla^2 W_1(z_s) (\nabla \varphi(z)-z)\d s\\
    - \int \nabla^2 W_1(z_s) (\nabla \varphi(z') - z') \d s\,,
\end{align*}
which, after integrating against $(z-\nabla \varphi(z))\rho(z')\rho(z)$ and 
relabeling, results in a quadratic and can be bound via Assumption~\ref{assump:W_lower} using that $\nabla^2 W_1$ is even:
\begin{align*}
    &\frac{1}{2}\iint\int_0^1 \begin{bmatrix}
        z-\nabla\varphi(z) \\ z'-\nabla\varphi(z')
    \end{bmatrix}^\top
    \begin{bmatrix}
        \nabla^2 W_1(z_s) &  -\nabla^2 W_1(z_s) \\  -\nabla^2 W_1(z_s) &  \nabla^2 W_1(z_s)
    \end{bmatrix}
    \begin{bmatrix}
        z-\nabla\varphi(z) \\ z'-\nabla\varphi(z')
    \end{bmatrix}\d s \d \rho(z') \d \rho(z) \ge 0\,,
\end{align*}
since the eigenvalues of $\begin{bmatrix}
    \nabla^2 W_1(z) & -\nabla^2 W_1(z) \\ -\nabla^2 W_1(z) & \nabla^2 W_1(z)
\end{bmatrix}$ have a tight lower-bound of zero $(\lambda_{W,1}\ge 0)$.
Computing a similar bound for the term dependent on $\mu$, we have
\begin{align*}
   C_3(\gamma,\tilde\gamma) \ge 0\,.
\end{align*}
Lastly, we show the bound for $C_4$. Computing for just the $\rho$-dependent term, we have
\begin{align*}
    \int (z-\nabla\varphi(z))&\cdot(\nabla V_1(\nabla\varphi(z))-\nabla V_1(z)) \d \rho(z)\\
    &=- \int \int_0^1 (z-\nabla \varphi(z)) \nabla^2 V_1((1-s)z+s\nabla \varphi(z)) (z-\nabla \varphi(z)) \d \rho(z) \\
    &\le - \lambda_{V,1} \int \norm{z-\nabla \varphi(z)}^2 \d \rho(z) = -\lambda_{V,1} \Wass_2(\rho,\rhot)^2\,,
\end{align*}
using Assumption~\ref{assump:V_lower}.
Computing similarly for the term dependent on $\mu$, we have $C_4(\gamma,\tilde \gamma) \ge \lambda_{V,1}\Wass(\rho,\rhot)^2 + \lambda_{V,2}\Wass(\mu,\tilde \mu)^2$.
We now have shown that $C_1(\gamma,\tilde\gamma)+C_2(\gamma,\tilde\gamma)+C_3(\gamma,\tilde\gamma)+C_4(\gamma,\tilde\gamma) \ge \lambda_c \Wbar(\gamma,\tilde \gamma)^2$, concluding the proof.
\end{proof}

\begin{lemma}\label{lem:propagating_dirac_contraction}
    Let Assumptions \ref{assump:f_lower}(ii), \ref{assump:V_lower}, and \ref{assump:W_lower} hold with $\lambda_c>0$. Fix $T>0$. Let $\gamma_t$ and $\gamma_t'$ be any two solutions of the dynamics \eqref{eq:dynamics_competitive}, with initial conditions $\rho_0,\rho_0'\in\P_2^{ac}(\R^{d_1})$ and $\mu_0=\delta_{x_0},\mu_0'=\delta_{x_0'}$, with $\alpha>0$ and $\beta=0$ such that $\rho_t,\rho_t'\in\P_2^{ac}(\R^{d_1})$ and $\mu_t=\delta_{x_t},\mu_t'=\delta_{x_t'}$ for all $t\in (0,T)$. Assume $D_c(\gamma_0)<\infty$, $D_c(\gamma_0')<\infty$ and  $\nabla_z\delta_\rho F_c[\gamma_t](z), \nabla_z\delta_\rho F_c[\gamma_t'](z)$ are locally Lipschitz in $z$ for all $t\in [0,T)$.
    Then $\gamma_t$ and $\gamma_t'$ satisfy
    \begin{align*}
        \Wbar(\gamma_t,\gamma_t') \le e^{-\lambda_c t} \Wbar(\gamma_0,\gamma_0') \quad \text{ for all } t\in [0,T)\,.
    \end{align*}
\end{lemma}
\begin{proof}
     Define $\nabla \varphi_t(z)$ and $T_t(x)$ such that $\rho_t = \nabla \varphi_{t\#} \rho_t'$ and $x_t = T_t(x_t')$. Note that $\nabla \varphi_t(z)$ is invertible because $\rho_t,\rho_t'$ are absolutely continuous. The time derivative of $\Wbar(\gamma_t,\gamma_t')^2$ is
     \begin{align*}
         \ddt \Wbar(\gamma_t,\gamma_t')^2 = \ddt \Wass_2(\rho_t,\rho_t')^2 + \ddt \norm{x_t -x_t'}^2 \,.
     \end{align*}
     To compute the time derivative of the Wasserstein-2 squared distance via \cite[Theorem 23.9]{Villani07}, we use that 
          \begin{align*}
         \int \norm{\nabla_z \delta_\rho F_c[\gamma_t](z)}^2 \d \rho_t(z) + \norm{\nabla_x \delta_\mu F_c[\gamma_t](x_t)}^2  < \infty\,,
     \end{align*}
     thanks to Lemma~\ref{lem:velocities_in_L2},
     with the same holding true for $\gamma_t'$. Then by \cite[Theorem 23.9]{Villani07}, the time derivative of $\Wass_2(\rho_t,\rho_t')^2$ is
      \begin{align*}
          \ddt \Wass_2(\rho_t,\rho_t')^2 &= 2\int \<\nabla_z \delta_\rho F_c[\gamma_t'](z),z-\nabla \varphi_{t}(z)> \d \rho_t'(z) \\
          &\quad +2 \int \<\nabla_z \delta_\rho F_c[\gamma_t](z),z-(\nabla \varphi_{t})^{-1}(z)> \d \rho_t(z) \\
          &= 2\int \<\nabla_z \delta_\rho F_c[\gamma_t'](z)-\nabla_z \delta_\rho F_c[\gamma_t](\nabla \varphi_t(z)),z-\nabla \varphi_{t}(z)> \d \rho_t'(z)  
      \end{align*}
      Since $\mu$ evolves as a Dirac for all time, the dynamics for $x_t$ are
      \begin{align*}
          \dot x_t =v(\rho_t,x_t) \coloneqq -\nabla_x \left( \int f(z,x_t)\d\rho_t(z) + V_2(x_t)\right) \,, 
      \end{align*}
      Note that the interaction term does not appear in the dynamics because
    \begin{align*}
       \nabla_x \int \delta_{x_t} W_2 \ast \delta_{x_t} = \nabla_x W_2(x_t-x_t) = \nabla_x W_2(0) = 0 \,.
    \end{align*}
    The time derivative of $\norm{x_t - x_t'}^2$ is therefore
    \begin{align*}
        \ddt &\norm{x_t-x_t'}^2 = 2 \<x_t-x_t',v(\rho_t,x_t)-v(\rho_t',x_t')> \\
        &= 2\<x_t'-T_t(x_t'),v(\rho_t',x_t')-v(\rho_t,T_t(x_t')) > \\
        &= -2\<x_t'-T_t(x_t'),\nabla_x \left(\int ( f(z,x_t')-f(\nabla \varphi_t(z),T_t(x_t')))\d \rho_t'(z) + V_2(x_t') -V_2(T_t(x_t'))\right) > \\
        &= -2\int \<x-T_t(x),\nabla_x \delta_\mu F_c[\gamma_t'](x)-\nabla_x \delta_\mu F_c[\gamma_t](T_t(x))> \d \mu_t'(x)\,,
    \end{align*}
    where $\mu_t' = \delta_{x_t'}$.
    Summing $\ddt \norm{x_t - x_t'}^2$ and $\ddt\Wass_2(\rho_t,\rho_t')^2$ results in
    \begin{align*}
        \ddt \Wbar(\gamma_t,\gamma_t')^2 = -2\iint \begin{bmatrix}
            z- \nabla \varphi_t(z) \\
            x_t - T_t(x) 
        \end{bmatrix} \cdot \begin{bmatrix}
             \nabla_z \delta_\rho F_c[\gamma_t](\nabla \varphi_t(z)) - \nabla_z \delta_\rho F_c[\gamma_t'](z) \\
            \nabla_x \delta_\mu F_c[\gamma_t'](x) - \nabla_x \delta_\mu F_c[\gamma_t](T_t(x)) 
        \end{bmatrix} \d \gamma_t'(z,x)\,,
    \end{align*}
    which is the same expression given in the proof of contraction in the setting where both $\mu_t,\rho_t\in\P_2^{ac}$ (Proposition~\ref{prop:contraction}) by choosing $\psi_t(x):=\frac{1}{2}\norm{x-x_t'+x_t}^2$ resulting in $T_t=\nabla\psi_t$ on $\supp\mu_t'$. Contraction, boundedness of the second moments, and convergence to the steady state follow similarly.
\end{proof}

\begin{lemma}\label{lem:rho_log_rho_lower_bound}
    For any $\rho\in\P_2(\R^d)$ and any $\eps>0$, it holds that $\int \rho \log \rho \ge -\eps \int \norm{z}^2 \d \rho(z) - c_\eps$, for some $c_\eps\in[0,\infty)$.
\end{lemma}
\begin{proof}
    Consider the function $a(x):[0,\infty) \to \R$ defined as $a(x) = x \log x$. The Legendre dual, given by
    $a^*(y) = \sup_{x\ge0} \left[x\cdot y - a(x)\right]\,, $
    is $a^*(y) = e^{y-1}$. By definition, for all $x,y\in [0,\infty)$, it holds that $a(x)+a^*(y) \ge x\cdot y$. Selecting $x=\rho(z)$ and 
    $y=-\eps \norm{z}^2$ for any $\eps>0$, for any value of $z\in\R^d$, we have
    \begin{align*}
        \rho(z) \log \rho(z) + e^{-\eps\norm{z}^2-1} \ge -\norm{z}^2\eps\rho(z)\,, 
    \end{align*}
    and integrating this function over all $z$ results in
    \begin{align*}
        \int \rho \log \rho \,\d z\ge -\eps \int \norm{z}^2 \d \rho(z)  - c_\eps\,,
    \end{align*}
    where $c_\eps := \int e^{-\eps \norm{z}^2-1}\d z \in [0,\infty)$.
\end{proof}
\begin{lemma}[Lower Semicontinuity of Entropy]\label{lem:lsc_entropy_santambrogio}
    For every sequence $\rho_n \rightharpoonup \rho$ converging weakly according to Definition~\ref{def:weakcv}, we have $\li \int \rho_n \log \rho_n \ge \int \rho \log \rho$.
\end{lemma}
\begin{proof}
    A uniform bound on the second moment implies a uniform bound on the first moment:
    \begin{align*}
        \int \norm{z} \d \rho_n(z) \le \int \left(\norm{z}^2 + \frac{1}{4}\right)\d \rho_n(z) \le C + \frac{1}{4} \,,
    \end{align*}
    then the result follows from \cite[Proposition 2.1]{santambrogio_dealing_2015}.
\end{proof}

\begin{lemma}[Upper and Lower Bounds on $\m$]\label{lem:upper_lower_bounds_f}
    Let $\m(z,x) \coloneqq -f_1(z,x) -\alpha \log \rhot(z) \in C^2(\R^d\times\R^d)$ and $x_n\rightarrow \bar x$ in $\R^d$.
    \begin{enumerate}
        \item [(1)] Under Assumptions~\ref{assump:f_lower}(ii) and \ref{assump:V_lower}, there exist constants $\hat c_1,\hat c_2 \in \R$ such that 
    \begin{align*}
        \m(z,x_n) \ge -\hat c_1 + \hat c_2 \norm{z}^2 \quad \forall\, z\in\R^{d}\,, n\in\nat\,.
    \end{align*}
    Further, if $\lambda_b>0$, then one can choose $\hat c_2 > 0$.
    \item [(2)] Under Assumptions~\ref{assump:f_upper} and \ref{assump:V_upper}, there exists constants $c_1\in \R$ and $c_2>0$ such that
    \begin{align*}
        \m(z,x_n) \le c_1 + c_2 \norm{z}^2 \quad \forall\, z\in\R^{d}\,, n\in\nat\,.
    \end{align*}
    \end{enumerate}
\end{lemma}
\begin{proof}
    Taylor expanding $\m(z,x)$ first in $z$ and then the second term in $x$, we have
\begin{align*}
  \m(z,x) &= \m(0,x) + z^\top \nabla_z \m(0,x) + \frac{1}{2} z^\top \nabla_{z}^2 \m(\xi_2,x) z \\
  &= \m(0,x) + z^\top \nabla_z \m(0,0) + z^\top \nabla_{xz}^2 \m(0,\xi_1) x + \frac{1}{2} z^\top \nabla_{z}^2 \m(\xi_2,x) z \,,
\end{align*}
with $\xi_1=sx$ for some $s\in[0,1]$, $\xi_2= \tau z$ for some $\tau\in[0,1]$.

Proof of (1): By Assumptions \ref{assump:f_lower}(ii) and \ref{assump:V_lower}, we have for some continuous function $\hat c(x)>0$
 \begin{align*}
\m(z,x) \ge \m(0,x) + z^\top \nabla_z \m(0,0) - L \|z\| \| x \| + \frac{\lambda_b}{2} \norm{z}^2
    \ge \hat c(x) + \hat c_2 \norm{z}^2\,,
 \end{align*}
with $\hat c(x) \in \R$ and $\hat c_2 >0$ if $\lambda_b > 0$.
Substituting $x=x_n$ in the above estimate and using that $x_n$ converges to $\bar x$, there exists a constant $\hat c_1\in\R$ such that
\begin{align}\label{eq:fn-lowerbound}
    \m(z,x_n) \ge -\hat c_1 + \hat c_2 \norm{z}^2\,.
\end{align}
Proof of (2): Similarly for the upper bound,
\begin{align*}
   \m(z,x) 
   &\le  \m(0,x) + z^\top \nabla_z \m(0,0) + L \|z\| \| x \| + \frac{-\ell_1+\alpha \tilde \Lambda}{2} \norm{z}^2
    \le c(x) + c_2 \norm{z}^2\,,
\end{align*}
where $c_2>0$
and $c(x)$ is a continuous function. The upper bounds $-\ell_1$ and $L$ come from Assumption~\ref{assump:f_upper}
and $\Lambdat$ from Assumption~\ref{assump:V_upper}. Again using that $x_n$ converges to $\bar x$, there exists a constant $c_1$ such that
\begin{align*}
    \m(z,x_n) \le c_1 + c_2 \norm{z}^2\,.
\end{align*}
\end{proof}
\begin{lemma}\label{lem:f_aui}
Let $m\in C^2(\R^d\times\R^d)$ such that $\|\nabla^2_zm(z,x)\|_2\ge \lambda$ and $\|\nabla^2_{xz}m(z,x)\|_2\le L$ for all $z,x\in\R^d$ for some $\lambda,L >0$.
Consider a sequence of vectors $x_n\in \R^d$ converging to some limit $\bar x\in \R^d$ and any sequence of measures $(\mu_n)_{n\in\nat}\in\P(\R^d)$ narrowly converging to some limit $\bar\mu\in\P(\R^d)$.
    Then there exists a constant $c\in\R$ such that $m_n(z)\coloneqq \m(z,x_n)\ge c$ for all $z\in\R^d$, for all $n\ge 0$. In particular, $(m_n^-)_{n\in\nat}$ is asymptotically uniformly integrable with respect to $(\mu_n)_{n\in\nat}$, where $m_n^-(z)=-\min\{ m_n(z),0 \}$.
\end{lemma}
\begin{proof}
Let $z^*(x):\R^d \to \R^d$, $\m^*(x):\R^d\to \R$, and $\bar z \in \R^d$ be given by 
\begin{align*}
    \m^*(x) \coloneqq \min_{z\in\R^d} m(z,x)\,, \quad z^*(x) \coloneqq \amin_{z\in \R^d} \m(z,x)\,, \quad \bar z = \amin_{z\in\R^d} \m(z,\bar x)\,.
\end{align*}
Showing that $(\m^*(x_n))_{n\in\nat}$ has a uniform lower bound is sufficient to achieve the desired a.u.i result for $(m_n^-)_{n\in\nat}$. We will show this uniform lower bound
using the implicit function theorem \cite[Theorem C.7]{evans_partial_2010}. By definition of the best response,
\begin{align*}
    \nabla_z \m(z,x)\big|_{z=z*(x)} = 0 \quad \Rightarrow \quad \nabla_{z}^2 \m(z^*(x),x) \nabla_x z^*(x) + \nabla_{zx}^2 \m(z^*(x),x) = 0\,,
\end{align*}
and since $\nabla_{z}^2 \m(z^*(x),x)$ is invertible due to strong convexity,
\begin{align*}
    \nabla_x z^*(x) & = -\left( \nabla_{z}^2 \m(z,x) ^{-1} \nabla_{zx}^2 \m(z,x)\right) \big|_{ z = z^*(x)} \,.
\end{align*}
Using the matrix bounds results in the following bound for $\nabla_x z^*(x)$
\begin{align*}
    \norm{\nabla_x z^*(x)}_2 \le \frac{L}{\lambda} \quad \forall \ x\in\R^d\,.
\end{align*}
This gradient bound provides a bound on the distance of $z^*(x_n)$ from $z^*(\bar x)$ as
\begin{align*}
    \norm{z^*(x_n) - z^*(\bar x)}_2 \le \frac{L}{\lambda} \norm{x_n-\bar x}\,.
\end{align*}
In particular, $z^*(x)$ is continuous and therefore $\m^*(x) \coloneqq \m(z^*(x),x)$ is continuous in $x$. 
Hence $\m(z,x_n)\ge \m^*(x_n) \ge c$ for all $n>0$ for some $c\in\R$. If $c \ge0$, then the a.u.i condition is immediately satisfied because $m_n^-(z) = 0$ for all $n$. If $c < 0$, then for all $K>|c|$, 
    \begin{align*}
        \limsup_{n\rightarrow\infty} \int |m_n^-(z)|\mathbbm{1}\{ z: |m_n^-(z)|\ge K\} \d\mu_n(z) = 0
    \end{align*}
    and $m_n^-$ is therefore asymptotically uniformly integrable.
\end{proof}

\section*{Acknowledgments} 
The authors are grateful for helpful discussions with Jos\'e A. Carrillo, Filippo Santambrogio and Andre Wibisono, as well as Eitan Levin, Matthieu Darcy and Pau Batlle. In particular, Filippo Santambrogio suggested the strategy for the proof of Proposition~\ref{prop:danskin-br} using the $\Gamma$-convergence result Proposition~\ref{prop:gamma_cv}, and pointed us to the version of Lemma~\ref{lem:rho_log_rho_lower_bound} in \cite{santambrogio_dealing_2015}. We thank Andre Wibisono for providing a correction for the proof of Lemma~\ref{lem:Wbar_upper_bd}, and Oliver Tse for suggesting Remark~\ref{rem:mean_field_SDE}. LC is funded by an NDSEG fellowship from the AFOSR and a PIMCO fellowship, FH is supported by start-up funds at the California Institute of Technology and by NSF CAREER Award 2340762, EM is supported in part from NSF award 2240110, LJR is supported in part by NSF 1844729, NSF 2312775, and ONR YIP N000142012571.

\printbibliography
\end{document}